\documentclass[a4paper, 15pt]{article}
\usepackage[utf8]{inputenc}
\usepackage[T1]{fontenc}
\usepackage[french,english]{babel}    
\usepackage[left=2cm,right=2cm,top=2cm,bottom=2cm]{geometry} 
\usepackage{graphicx}					
\usepackage{amssymb}
\usepackage[backend=bibtex,style=numeric,maxnames=1000]{biblatex}
\usepackage{lineno}
\usepackage{amsthm}
\usepackage{amsmath}
\usepackage{amssymb}
\usepackage{amsfonts}
\usepackage{mathtools}
\usepackage{amssymb}
\usepackage{latexsym}
\usepackage{hyperref}
\usepackage{mathrsfs}
\usepackage[utf8]{inputenc}
\usepackage[parfill]{parskip}
\usepackage{xcolor}
\usepackage[affil-it]{authblk}
\usepackage{enumitem}

\numberwithin{equation}{section}
\newtheorem{defi}{Definition}[section]
\newtheorem{unTheorem}{Theorem}[section]

\newtheorem{propal}[defi]{Proposition}
\newtheorem{rem}[defi]{Remark}

\newtheorem{lem}[defi]{Lemma}
\newtheorem{ass}[defi]{Assumptions}
\newtheorem{nota}[defi]{Notation}

\numberwithin{equation}{section}
\title{Multiphase high-frequency solutions to Klein-Gordon-Maxwell equations in Lorenz gauge in the (3+1) Minkowski spacetime}
\author[]{Tony Salvi}
\date{}
\begin{document}

\maketitle

\begin{center}
    \textbf{Abstract}
    \hfill\begin{minipage}{\dimexpr\textwidth-0.25cm}
We study a 1-parameter family $(A_\lambda,\Phi_\lambda)_{0<\lambda}$ of multiphase high-frequency solutions to Klein-Gordon-Maxwell equations in Lorenz gauge in the (3+1)-dimensional Minkowski spacetime. This family is based on an ansatz for the initial data. We prove that for $\lambda$ small enough the family of solutions exists on an interval uniform in $\lambda$ depending only on the initial ansatz. These solutions are close to an approximate solution constructed by geometric optics. The initial ansatz needs to be regular enough, to satisfy a polarization condition and to satisfy the constraints for Maxwell in Lorenz gauge, but there is no need for smallness of any kind. The phases need to interact in a coherent way. We also observe that the limit $(A_0,\Phi_0)$ is not solution to Klein-Gordon-Maxwell equations but to a Klein-Gordon-Maxwell null-transport type system, i.e., there is a backreaction.
\end{minipage}
\end{center}
\tableofcontents

\section{Introduction}
\label{section:Intro}
In this article, we construct one-parameter families of multiphase high-frequency solutions to the Klein-Gordon-Maxwell (KGM) equations
\begin{equation}
\begin{cases}
\label{eq:KGMintro}
\partial_\alpha F^{\alpha\beta}=-\Im(\Phi\overline{\partial^{\beta}\Phi})+A^{\beta}|\Phi|^2,\\
(\partial^\alpha+iA^\alpha)(\partial_\alpha+iA_\alpha)\Phi=0,
\end{cases}
\end{equation}
in the Minkowski spacetime on $\mathbb{R}^{3+1}$. The unknowns are the charged scalar field $\Phi$ (a complex function) and the electromagnetic four-potential $A$. We also define the Faraday tensor (the electromagnetic field) by $F=dA$, that is, in coordinates, $F_{\alpha\beta}=\partial_\alpha A_\beta-\partial_\beta A_\alpha$.
Our family of high-frequency solutions $(A_\lambda,\Phi_\lambda)_{0<\lambda}$ is indexed by a small parameter $\lambda$, where $\lambda\to0$ represents the characteristic wavelength, and its initial data are constructed from the following WKB-type ansatz:
 \begin{align}
     \label{eq:initialansatzintro1}
        a_{1\lambda}^\alpha&=a_0^\alpha+\lambda^{1/2}\sum_{\mathscr{A}\in\mathbb{A}}\left(p_\mathscr{A}^\alpha \cos\left(\frac{v_\mathscr{A}}{\lambda}\right)+q_\mathscr{A}^\alpha \sin\left(\frac{v_\mathscr{A}}{\lambda}\right)\right),\\
        \phi_{1\lambda}&=\phi_0+\lambda^{1/2}\sum_{\mathscr{A}\in\mathbb{A}}\psi_\mathscr{A}e^{i\frac{v_\mathscr{A}}{\lambda}}.
 \end{align}
Here, $(a_0^\alpha,\phi_0)$ (together with $(\dot{a}_0^\alpha,\dot{\phi}_0)$ for their time derivatives) denotes the background initial data, on top of which we superpose a finite sum ($|\mathbb{A}|=N\in\mathbb{N}$) of small perturbations, scaled by $\lambda^{1/2}$. These perturbations consist of amplitudes $(p_\mathscr{A}^\alpha,q_\mathscr{A}^\alpha ,\psi_\mathscr{A})_{\mathscr{A}\in\mathbb{A}}$ multiplied by highly oscillating factors with non-stationary phases $(v_\mathscr{A})_{\mathscr{A}\in\mathbb{A}}$. This framework naturally falls within the viewpoint of geometric optics, in which we describe (electromagnetic or charged scalar fields) waves by localized rays. \\\\
Our result is twofold:
\begin{itemize}
\item We prove that the multiphase high-frequency approximation of solutions to \eqref{eq:KGMintro} constructed via geometric optics (and thus by the WKB expansion method) is stable and rigorous. In particular, the result holds for any number of superposed phases $N$, provided the phases interact coherently. This includes all possible plane wave interactions.\\
\item We exhibit a backreaction phenomenon. The dynamics of the limit $(A_0,\Phi_0)$ is non-trivial: small-scale inhomogeneities, corresponding to the high-frequency waves of low amplitude modelled by geometric optics in \eqref{eq:initialansatzintro1}, induce large-scale effects on the background $(A_0,\Phi_0)$ and act as an effective charge flux. 
\end{itemize}

Regarding the first point, we refer to \cite{METIVIER2009169} and \cite{10.1215/S0012-7094-93-07007-X} for a complete review of the WKB approximation for first-order hyperbolic systems. These results are not directly applicable to KGM since recovering the hyperbolicity requires fixing a gauge. Moreover, KGM equations \eqref{eq:KGMintro} are of second order. Nevertheless, the WKB approximation for KGM (and for the more general Yang-Mills-Higgs-Dirac system, of which KGM is a subcase) is known to be stable in the monophase case since the works \cite{zbMATH02124168,zbMATH01799448}, where approximate solutions are constructed at arbitrary order. Here, we extend the stability result to the multiphase setting and provide a sharper method that requires only a first-order approximate solution (as opposed to at least third-order ones in \cite{zbMATH02124168,zbMATH01799448}). As a consequence, our method allows for less regular initial data in the ansatz \eqref{eq:initialansatzintro1} than is allowed in \cite{zbMATH02124168,zbMATH01799448}. We state the following rough Theorem (see Section \ref{section:results} for the complete version):
\begin{unTheorem}
\label{unTheorem:Th1main}
Assuming an approximate gauge condition, a polarization condition, that the constraints imposed by Maxwell’s equations hold, and coherent phase interactions on the initial ansatz \eqref{eq:initialansatzintro1}, there exists (at least) one family of solutions $(A_\lambda,\Phi_\lambda)_{0<\lambda}$ based on \eqref{eq:initialansatzintro1}, with a uniform time of existence for $0<\lambda<\lambda_0$, for some $\lambda_0\in\mathbb{R}$ small enough.
\end{unTheorem}
In this paper, the gauge condition is the Lorenz  one: $\partial_\alpha A^\alpha=0$. To have an approximate gauge condition for the initial data, one must ask for $\partial_i a^i+\dot{a}^0=0$, but also for the polarization condition $-|\nabla v_\mathscr{A}|p_\mathscr{A}^0+\partial_i v_\mathscr{A}p_\mathscr{A}^i=-|\nabla v_\mathscr{A}|q_\mathscr{A}^0+\partial_i v_\mathscr{A}q_\mathscr{A}^i=0$. The latter corresponds to the standard polarization of light, at the level of the potential. The constraint is given by the 0-component of Maxwell's equations, also known as Gauss's law,
\begin{align}
\label{eq:constrmaxin}
\partial_\alpha F^{\alpha0}=J^0
\end{align}
which equates the divergence of the electric field with the charge density $J^0$. This yields an elliptic equation for the initial data $a^0_0$. The coherence assumption is related to the alignment of couples of phase gradients $(\nabla v_\mathscr{A},\nabla v_\mathscr{B})_{\mathscr{A},\mathscr{B}\in\mathbb{A}}$, which are non-zero due to non-stationarity.\\\\ Under the Lorenz gauge condition, the KGM system \eqref{eq:KGMintro} reduces to the schematic form
\begin{equation}
\label{eq:schemeqIntro}
    \Box \textbf{F}_\lambda=\textbf{F}_\lambda\partial\textbf{F}_\lambda+(\textbf{F}_\lambda)^3,
\end{equation}
where $\textbf{F}_{\lambda}$ denotes $(A_\lambda,\Phi_\lambda)$. This notation will be used throughout the rest of the paper. The basic strategy underlying the proof of Theorem \ref{unTheorem:Th1main} is to seek solutions to \eqref{eq:schemeqIntro} under the form
\begin{equation}
        \label{eq:schemfullparametrixIntro}       
        \textbf{F}_{\lambda}=\underbrace{\textbf{F}_0+\lambda^{1/2}\sum_{\mathscr{A}\in\mathbb{A}}e^{i\frac{u_\mathscr{A}}{\lambda}}\textbf{F}_\mathscr{A}}_\text{$\textbf{F}_{1\lambda}$}+\textbf{Z}_{\lambda},
\end{equation}
where $(u_\mathscr{A})_{\mathscr{A}\in\mathbb{A}}$ are phases, $(\textbf{F}_{1\lambda})_{0<\lambda}$ are approximate solutions to \eqref{eq:schemeqIntro} and $(\textbf{Z}_{\lambda})_{0<\lambda}$ denotes error terms as $\lambda\to0$. Then, we show that the Lorenz gauge condition is satisfied, ensuring that $(\textbf{F}_{\lambda})_{0<\lambda}$ indeed solves the KGM system \eqref{eq:KGMintro}. We also address the construction of initial data for the error term satisfying the constraint \eqref{eq:constrmaxin}, a point that is absent from \cite{zbMATH02124168,zbMATH01799448}. To obtain a uniform time of existence, we rely on a bootstrap argument on the smallness of the error term $(\textbf{Z}_{\lambda})_{0<\lambda<\lambda_0}$, for $0<\lambda_0$ small enough. As explained below, the most delicate nonlinearity to deal with is the $\textbf{F}_\lambda\partial\textbf{F}_\lambda$ term. This nonlinearity is also important because it gives rise to $O(1)$ interaction terms, which constitute the main mechanism responsible for the backreaction. This leads us to our second main subject.\\\\
The backreaction phenomenon in high-frequency limits is well understood for a related system: the Einstein vacuum equations. We refer to \cite{Choquetbruhat1969ConstructionDS} for the first mathematical work on the subject and \cite{huneau2022trilinear}, \cite{zbMATH07009768}, \cite{zbMATH06921833}, \cite{luk2020highfrequency}, \cite{huneau2024burnetts}, \cite{zbMATH07719650}, \cite{zbMATH07732086} and \cite{touati2024reverse} for recent contributions. These works are dedicated to the proof of the Burnett conjecture (or its reverse counterpart) enunciated in \cite{zbMATH04086482}. Roughly speaking, this conjecture states that the limits of such high-frequency solutions\footnote{The assumptions are that the convergence of the derivative of the metric is only a weak convergence, see \cite{zbMATH04086482} for a more precise statement.} to the vacuum Einstein equations solve the Einstein equations coupled with a massless Vlasov field, thereby exhibiting a backreaction through an effective stress–energy tensor. Thus, we establish in the present work an analog of the Burnett conjecture for the KGM system, though weaker for reasons stated in Section \ref{subsection:Einsteincomp}, together with other comparisons. From a structural perspective, it is not surprising to observe a common behavior for high-frequency limits of solutions to KGM and Einstein equations. Indeed, both systems are gauge-invariant nonlinear wave equations and share analogous nonlinear structures, for example, a partial null structure (see \cite{zbMATH03966906} and \cite{zbMATH03989952} for the introduction of this notion). Our second main result is the following (see Section \ref{section:results} for the complete version):
\begin{unTheorem}
 We consider the family $(A_\lambda,\Phi_\lambda)_{0<\lambda}$ of solutions to KGM given by the previous theorem. The $L^2$-limit $(A_0,\Phi_0)=\lim_{\lambda\to0}(A_\lambda,\Phi_\lambda)$ is solution to a Klein-Gordon-Maxwell-null-transport (KGMn) system 
\begin{equation}
\begin{cases}
\label{eq:KGMnullintro}
\partial_\alpha (F_0)^{\alpha\beta}=-\Im(\Phi_0\overline{\partial^{\beta}\Phi_0})+(A_0)^{\beta}|\Phi_0|^2+\sum_{\mathscr{A}\in\mathbb{A}}\partial^\beta u_{\mathscr{A}}\rho_\mathscr{A},\\
(\partial^\alpha+i(A_0)^\alpha)(\partial_\alpha+i(A_0)_\alpha)\Phi_0=0,\\
\forall\mathscr{A}\in\mathbb{A},\;\partial^\beta u_{\mathscr{A}}\partial_\beta u_{\mathscr{A}}=0,\\
\forall\mathscr{A}\in\mathbb{A},\;\partial^\beta u_{\mathscr{A}}\partial_\beta\rho_{\mathscr{A}}+\Box u_{\mathscr{A}}\rho_{\mathscr{A}}=0.
\end{cases}
\end{equation}
\end{unTheorem}
Here, $(u_\mathscr{A})_{\mathscr{A}\in\mathbb{A}}$ are a family of characteristic phases for the d'Alembertian $\Box=-\partial_{tt}^2+\Delta$ (see Section \ref{subsection:geoopt}), corresponding to the initial phases $(v_\mathscr{A})_{\mathscr{A}\in\mathbb{A}}$, and $(\rho_\mathscr{A})_{\mathscr{A}\in\mathbb{A}}$ denotes charge densities transported along the null rays of $(u_\mathscr{A})_{\mathscr{A}\in\mathbb{A}}$ from the initial densities $(|\psi_\mathscr{A}|^2)_{\mathscr{A}\in\mathbb{A}}$. The limit of high-frequency solutions to the KGM system \eqref{eq:KGMintro} is not itself solution to KGM \eqref{eq:KGMintro}, but rather to KGMn \eqref{eq:KGMnullintro}, i.e., there is a backreaction.\\\\
We now present a few basic computations to clarify the origin of the backreaction. Applying the wave operator, that is, the main operator of \eqref{eq:schemfullparametrixIntro}, to the first-order approximation $\textbf{F}_{1\lambda}$ yields
\begin{align*}
\Box\textbf{F}_{1\lambda}&=\Box\textbf{F}_{0}+\Box\left( \lambda^{1/2}\sum_{\mathscr{A}\in\mathbb{A}}e^{i\frac{u_\mathscr{A}}{\lambda}}\textbf{F}_\mathscr{A}\right)\\
&=\underbrace{\Box\textbf{F}_{0}}_\text{Background}+\lambda^{-3/2}\sum_{\mathscr{A}\in\mathbb{A}}\underbrace{-\partial_\alpha u_\mathscr{A}\partial^\alpha u_\mathscr{A}}_\text{\;$=0$, eikonal equation}e^{i\frac{u_\mathscr{A}}{\lambda}}\textbf{F}_\mathscr{A}+\lambda^{-1/2}\sum_{\mathscr{A}\in\mathbb{A}}\underbrace{i(2\partial^\alpha u_\mathscr{A}\partial_\alpha+\Box u_{\mathscr{A}})(\textbf{F}_\mathscr{A})}_\text{Transport} e^{i\frac{u_\mathscr{A}}{\lambda}}+O(\lambda^{1/2}).
\end{align*}
 At the highest amplitudes, we observe two of the main features of the WKB method: the eikonal equation and the transport of the profiles. Therefore, we naturally recover that the phases $(u_\mathscr{A})_{\mathscr{A}\in\mathbb{A}}$ are isotropic and that the profiles are transported along the associated null rays, as are the densities in \eqref{eq:KGMnullintro}. Non-schematically, the term $\lambda^{1/2}\sum_{\mathscr{A}\in\mathbb{A}}e^{i\frac{u_\mathscr{A}}{\lambda}}\textbf{F}_\mathscr{A}$ represents\footnote{For simplicity, we restrict ourselves to the case $q_\mathscr{A}=0$.} both oscillatory perturbations $\lambda^{1/2}\sum_{\mathscr{A}\in\mathbb{A}}P_\mathscr{A}^\alpha \cos\left(\frac{u_\mathscr{A}}{\lambda}\right)$ and $\lambda^{1/2}\sum_{\mathscr{A}\in\mathbb{A}}e^{i\frac{u_\mathscr{A}}{\lambda}}\Psi_\mathscr{A}$. Substituting the first-order approximation $\textbf{F}_{1\lambda}$ into the nonlinearity $\textbf{F}_\lambda\partial\textbf{F}_\lambda$, we obtain 
\begin{align*}
\Im(\Phi\overline{\partial^{\beta}\Phi})&=\underbrace{\Im(\Phi_0\overline{\partial^{\beta}\Phi_0})}_\text{Background}+\underbrace{\sum_{\mathscr{A}\in\mathbb{A}}\partial^\beta u_\mathscr{A}|\Psi_\mathscr{A}|^2}_\text{Backreaction}+\lambda^{-1/2}\sum_{\mathscr{A}\in\mathbb{A}}\underbrace{\Im\left(-i\Phi_0e^{i\frac{-u_\mathscr{A}}{\lambda}}\overline{\Psi_\mathscr{B}}\right)}_\text{Absorbed in the transport}\\
&+\underbrace{\sum_{\mathscr{A},\mathscr{B}\in\mathbb{A},\mathscr{A},\neq\mathscr{B}}\Im\left(-ie^{i\frac{u_\mathscr{A}-u_\mathscr{B}}{\lambda}}\Psi_\mathscr{A}\overline{\Psi_\mathscr{B}}\right)}_\text{$=\tilde{\Xi}_{A\lambda}$}+O(\lambda^{1/2})
\end{align*}
for the Maxwell equation, and 
\begin{align*}
-2iA^\alpha\partial_\alpha\Phi&=\underbrace{-2iA^\alpha_0\partial_\alpha\Phi_0}_\text{Background}+\underbrace{\sum_{\mathscr{A}\in\mathbb{A}}\Psi_\mathscr{A}\partial_\alpha u_\mathscr{A}P^\alpha_\mathscr{A}(e^{i2\frac{u_\mathscr{A}}{\lambda}}+1)}_\text{$=0$, by polarization}+\underbrace{\sum_{\mathscr{A}\in\mathbb{A}}2A^\alpha_0\partial_\alpha u_\mathscr{A}e^{i\frac{u_\mathscr{A}}{\lambda}}\Psi_\mathscr{A}}_\text{Absorbed by the transport}\\
&+\underbrace{\sum_{\mathscr{A},\mathscr{B}\in\mathbb{A},\mathscr{A},\neq\mathscr{B}}\Psi_\mathscr{A}\partial_\alpha u_\mathscr{A}P^\alpha_\mathscr{B}(e^{i\frac{u_\mathscr{A}+u_\mathscr{B}}{\lambda}}+e^{i\frac{u_\mathscr{A}-u_\mathscr{B}}{\lambda}})}_\text{$=\tilde{\Xi}_{\Phi\lambda}$}+O(\lambda^{1/2})
\end{align*} 
for the Klein-Gordon equation. The backreaction, namely the $O(1)$ non-highly-oscillating term, is clearly visible in the Maxwell equation. Since it acts at the level of the background, it must be incorporated into the limiting dynamics of the latter. In contrast, in the Klein-Gordon equation, the potential backreaction term is canceled by the polarization condition $\partial_\alpha u_\mathscr{A}P^\alpha_\mathscr{A}=0$, which is propagated from the initial data. Then, the remaining terms $\tilde{\Xi}_{A\lambda}$ and $\tilde{\Xi}_{\Phi\lambda}$ are detailed below. For the nonlinearity $(\textbf{F}_\lambda)^3$, there are no effects at this level of the analysis. Overall, we see the heuristic behind the limiting dynamics given by \eqref{eq:KGMnullintro}.\\\\
In the previous computations, we have essentially identified the main ingredients required to construct the first-order approximate solutions $(\textbf{F}_{1\lambda})_{0<\lambda}$. A first key advantage of our method is that it does not require the construction of higher-order approximate solutions, thereby avoiding the derivation of a larger hierarchy of equations. A second advantage is reflected in the coherence assumption. Specifically, we assume that pairwise phase interactions are either everywhere non-resonant, namely $\partial_\alpha (u_\mathscr{A}\pm u_\mathscr{B})\partial^\alpha(u_\mathscr{A}\pm u_\mathscr{B})\neq0$, or everywhere resonant, namely $\partial_\alpha (u_\mathscr{A}\pm u_\mathscr{B})\partial^\alpha(u_\mathscr{A}\pm u_\mathscr{B})=0$. We emphasize that the restriction is propagated from the initial data and that it applies only to pairwise interactions, which is natural since higher-order interactions do not appear at the level of the first-order approximation. In particular, this assumption is sufficient to treat the $\tilde{\Xi}_{A\lambda}$ and $\tilde{\Xi}_{\Phi\lambda}$ terms with the error term, despite  their apparent lack of smallness. Indeed, these terms are highly oscillatory, with phases arising from pairwise interactions, and may therefore be separated into two categories corresponding to non-resonant and resonant interactions. In both cases, the terms $\tilde{\Xi}_{A\lambda}$ and $\tilde{\Xi}_{\Phi\lambda}$ are absorbed by introducing additional unknowns in the error decomposition, namely $\boldsymbol{\mathcal{E}}_\lambda^{ell}$ and $(\breve{\textbf{F}}^+_{\mathscr{A}\pm\mathscr{B}})_{(\mathscr{A},\mathscr{B})\in\mathscr{C}}$, for $\mathscr{C}$ the set of resonant interactions. On the other hand, since the error term $\textbf{Z}_\lambda$ is constructed only after the first-order approximation, it admits uniform control of at most half a derivative. The coefficient $1/2$ corresponds to the scale-invariant $\dot{H}^{1/2}$ norm for the KGM system \eqref{eq:KGMintro} and is naturally related to the conjectured minimal regularity threshold for its well-posedness, see \cite{zbMATH00567356,Cuccagna01011999,zbMATH01782982,zbMATH02063481,keel2010globalwellposednessmaxwellkleingordonequation,zbMATH07196901}. These poor uniform bounds on the error term $\textbf{Z}_\lambda$ create additional difficulties when estimating the nonlinear term in the right-hand side of \eqref{eq:schemfullparametrixIntro}. Nonetheless, only specific directions of oscillation are responsible for this bad behavior. We therefore isolate these contributions by introducing further unknowns ($(\textbf{F}^+_{\lambda\mathscr{A}})_{\mathscr{A}\in\mathbb{A}}$) into the error. In the end, the error term has the form 
\begin{equation}
\label{eq:schemparaerrorIntro}
\textbf{Z}_\lambda=\sum_{(\mathscr{A},\mathscr{B})\in\mathscr{C}}\lambda^{1}\breve{\textbf{F}}^+_{\mathscr{A}\pm\mathscr{B}}e^{i\frac{u_{\mathscr{A}}\pm u_{\mathscr{B}}}{\lambda}}+\boldsymbol{\mathcal{E}}_\lambda^{ell}+\sum_{\mathscr{A}\in\mathbb{A}}\lambda^{1}\textbf{F}^+_{\lambda\mathscr{A}}e^{i\frac{u_\mathscr{A}}{\lambda}}+\boldsymbol{\mathcal{E}}_\lambda^{evo},
\end{equation} 
whose precise structure is more detailed in Sections \ref{subsection:Idea} and \ref{section:exactsolut}. The various components of $\textbf{Z}_\lambda$ are either computed explicitly with respect to the background or solve a wave-transport system. Despite an apparent loss of derivatives, this system is proved to be well-posed through the introduction of appropriate auxiliary functions. The refined structure of $\textbf{Z}_\lambda$ allows us to close the bootstrap and, in particular, to control the nonlinearity $\textbf{F}_\lambda\partial\textbf{F}_\lambda$. This nonlinearity generates dangerous error-perturbation interaction terms, which are absorbed by the variable $(\textbf{F}^+_{\lambda\mathscr{A}})_{\mathscr{A}\in\mathbb{A}}$, as well as error-error interaction terms. The latter are dealt with using precise product estimates involving Strichartz inequalities. Such interaction terms also appear in \cite{zbMATH02124168,zbMATH01799448}. However, in those works, they are controlled by constructing a smaller error term based on a third-order approximate solution. By contrast, our approach avoids this requirement. Finally, we add that we are able to construct generic families of error terms satisfying Maxwell's constraints and possessing the required smallness. This step is non-trivial, notably because our approximate solutions are constructed for the reduced system in Lorenz gauge \eqref{eq:schemeqIntro} rather than for the full KGM system \eqref{eq:KGMintro}. The main ingredients for this construction are once again the coherence assumption, the polarization condition, the non-stationarity of the phases, and the precise structure of $\tilde{\Xi}_{A\lambda}$.  \\\\
Overall, the steps of the proof of both Theorems are more detailed in Section \ref{subsection:Idea}.
\subsection{Generalities on KGM and KGMn}
\label{subsection:geneKGM}
We reintroduce the Klein-Gordon-Maxwell (KGM) system
\begin{equation}
\begin{cases}
\label{eq:KGMgeneKGM}
\partial_\alpha F^{\alpha\beta}=- \Im(\Phi\overline{\partial^{\beta}\Phi})+A^{\beta}|\Phi|^2,\\
(\partial^\alpha+iA^\alpha)(\partial_\alpha+iA_\alpha)\Phi=0
\end{cases}
\end{equation}
and the Klein-Gordon-Maxwell-null-transport (KGMn) system
\begin{equation}
\begin{cases}
\label{eq:KGMnullgeneKGM}
\partial_\alpha F^{\alpha\beta}=-\Im(\Phi\overline{\partial^{\beta}\Phi})+A^{\beta}|\Phi|^2+\sum_{\mathscr{A}\in\mathbb{A}}\partial^\beta u_{\mathscr{A}}\rho_\mathscr{A},\\
(\partial^\alpha+iA^\alpha)(\partial_\alpha+iA_\alpha)\Phi=0,\\
\forall\mathscr{A}\in\mathbb{A},\;\partial^\beta u_{\mathscr{A}}\partial_\beta u_{\mathscr{A}}=0,\\
\forall\mathscr{A}\in\mathbb{A},\;\partial^\beta u_{\mathscr{A}}\partial_\beta\rho_{\mathscr{A}}+\Box u_{\mathscr{A}}\rho_{\mathscr{A}}=0.
\end{cases}
\end{equation}
The KGM system is a classical model of quantum electrodynamics describing charged, spinless particles, which in our setting are moreover massless, interacting with an electromagnetic field. The KGMn system is obtained by introducing an additional charge source, given by a finite collection of charge densities $(\rho_\mathscr{A})_{\mathscr{A}\in\mathbb{A}}$ transported along prescribed (irrotational) flux directions $(\textbf{d}u^\#_\mathscr{A})_{\mathscr{A}\in\mathbb{A}}$. While these charges influence the electromagnetic field, their trajectories (following the flow lines of $(\textbf{d}u^\#_\mathscr{A})_{\mathscr{A}\in\mathbb{A}}$) are free and are not affected by the field itself. \\\\
Both systems are invariant under $U(1)$-gauge transformation. More precisely, let $\chi$ be a real function on $\mathbb{R}^{3+1}$ and define 
\begin{align*}
    &A'_\alpha=A_\alpha+\partial_\alpha\chi, &\Phi'=e^{-i\chi}\Phi.
\end{align*}
Then, formally, $(A',\Phi')$ (resp. $(A',\Phi',(\rho_\mathscr{A})_{\mathscr{A}\in\mathbb{A}},(u_\mathscr{A})_{\mathscr{A}\in\mathbb{A}})$) is a solution to KGM (resp. KGMn) if and only if $(A,\Phi)$ (resp. $(A,\Phi,(\rho_\mathscr{A})_{\mathscr{A}\in\mathbb{A}},(u_\mathscr{A})_{\mathscr{A}\in\mathbb{A}})$) is a solution too. One may therefore fix a gauge and remove the indeterminacy by imposing an additional equation. In this paper, we adopt the \textbf{Lorenz gauge} 
\begin{equation}
    \label{eq:lorenzgaugegeneKGM}
    \partial_\alpha A^\alpha=0
\end{equation}
which allows us to recover a system of semilinear wave equations. Under this choice \eqref{eq:KGMgeneKGM} reduces to
\begin{equation}
\begin{cases}
\label{eq:KGMLgeneKGM}
\Box A^\beta=- \Im(\Phi\overline{\partial^{\beta}\Phi})+A^{\beta}|\Phi|^2,\\
\Box\Phi=-2iA^\alpha\partial_\alpha\Phi+A^\alpha A_\alpha\Phi.
\end{cases}
\end{equation}
We refer to \eqref{eq:KGMLgeneKGM} as KGML. For simplicity, we also refer to the reduced Maxwell and reduced Klein-Gordon equations as Maxwell and Klein-Gordon equations. A completely analogous reduced formulation is obtained for the KGMn system in Lorenz gauge. \\
Other gauge choices frequently used in the literature include the temporal gauge and the Coulomb gauge. These gauges have notably been employed to establish global existence results, as in \cite{zbMATH03781718}, as well as global or local existence for low-regularity classes of solution, as in \cite{dfb0a54d-bfd7-3b95-ae19-04c1b390d84d}, \cite{10.1215/S0012-7094-94-07402-4}, \cite{zbMATH02063481} and \cite{keel2010globalwellposednessmaxwellkleingordonequation}.\\\\
By Noether's Theorem, gauge invariance implies the conservation of the charge flux
\begin{align*}
    &J^\beta=- \Im(\Phi\overline{\partial^{\beta}\Phi})+A^{\beta}|\Phi|^2  &\text{(resp.}\;\;J^\beta=- \Im(\Phi\overline{\partial^{\beta}\Phi})+A^{\beta}|\Phi|^2+\sum_{\mathscr{A}\in\mathbb{A}}\partial^\beta u_{\mathscr{A}}\rho_\mathscr{A})
\end{align*}
which satisfies
\begin{equation}
\label{eq:divchargegeneKGM}
    \partial_\alpha J^\alpha=0.
\end{equation}
This conservation law plays a central role in the propagation of the Lorenz gauge from the initial data. 
\subsubsection{Lorenz gauge propagation}
\label{subsubsection:propaggauegeneKGM}
To solve \eqref{eq:KGMgeneKGM} (or similarly \eqref{eq:KGMnullgeneKGM}), one may first prescribe initial data $(A^\alpha|_{t=0}=a^\alpha,\partial_tA^\alpha|_{t=0}=\dot{a}^\alpha,\Phi|_{t=0}=\phi,\partial_t\Phi|_{t=0}=\dot{\phi})$ satisfying both the Lorenz gauge condition and the Maxwell constraint equation arising from Gauss's law, namely
\begin{equation}
\begin{cases}
\label{eq:constraintsgeneKGM}
\partial_j a^j+\dot{a}^0=0,\\
\Delta a^0+\partial_j\dot{a}^j=J_0.
\end{cases}
\end{equation}
One then solves the hyperbolic KGML system \eqref{eq:KGMLgeneKGM} (or the corresponding reduced system associated with \eqref{eq:KGMnullgeneKGM}) and subsequently shows that the Lorenz gauge condition propagates in time. To this end, we take the divergence of the Maxwell equation in \eqref{eq:KGMLgeneKGM} and obtain that 
\begin{equation}
\label{eq:propaggaugegeneKGM}
     \Box \partial_\beta A^\beta=-\partial_\beta A^{\beta}|\Phi|^2,
\end{equation}
where we have used the conservation of charge \eqref{eq:divchargegeneKGM}. From there, since the initial data are chosen so that $\partial_\alpha A^\alpha|_{t=0}=0$ and $\partial_t(\partial_\alpha A^\alpha)|_{t=0}=0$ as a consequence of the constraints, it follows that the Lorenz gauge condition \eqref{eq:lorenzgaugegeneKGM} remains satisfied for all time, i.e., the gauge propagates.\\
\subsection{Geometric optics}
\label{subsection:geoopt}
\subsubsection{Plane wave}
\label{subsubsection:planewave}
Before turning to the standard WKB method and multiphase high-frequency solutions, we briefly recall the plane wave solutions of the Klein-Gordon-Maxwell system linearized around the vacuum state, 
\begin{equation}
    \begin{cases}
    \label{eq:KGMlineargeneKGM}
        \partial_\alpha F^{\alpha\beta}_{p.w.}=\Box A^{\beta}_{p.w.}-\partial^\beta\partial_{\alpha}A^\alpha_{p.w.}=0,\\
        \Box\Phi_{p.w.}=0.
    \end{cases}
\end{equation}
Plane wave solutions are constant along affine hyperplanes orthogonal to the direction of propagation. They can be written in the form 
\begin{align*}
    & A^\beta_{p.w.}=P^\beta \cos(k_\alpha x^\alpha)+Q^\beta \sin(k_\alpha x^\alpha), &\Phi_{p.w.}=\Psi e^{ik'_\alpha x^\alpha}
\end{align*}
for the four-potential and the scalar field, respectively. The quantities $P,Q,\Psi,k$ and $k'$ are constant vectors or scalars. Seeking solutions to \eqref{eq:KGMlineargeneKGM} leads us to the following constraints for the parameters\footnote{One may verify that the Maxwell vacuum equation also admits plane wave solutions with non-characteristic phases, i.e.\ $k_\alpha k^\alpha\neq0$. These correspond to pure gauge solutions and are therefore invisible at the level of the Faraday tensor.}
\begin{align*}
    &k_\alpha k^\alpha=k'_\alpha k'^\alpha=0, &k_\alpha P^\alpha=k_\alpha Q^\alpha=0.
\end{align*}
The waves propagate in a null direction, i.e., at the speed of light, and the amplitudes of the potential are polarized\footnote{The Faraday tensor itself is polarized. The electromagnetic waves oscillate perpendicularly to the direction of propagation.}. It is not surprising that the initial ansatz considered in this work is assumed to satisfy the same polarization.
\subsubsection{Brief overview of geometric optics}
\label{subsection:revgeoopt}
Geometric optics is a method for constructing high-frequency approximations of solutions to partial differential equations. It originates from classical optics, whose purpose is to describe the propagation of light, that is, electromagnetic waves governed by the Maxwell equations. In this framework, propagation occurs along rays and fits the particle-like viewpoint of light. This high-frequency approximation method is generalized to first-order symmetric hyperbolic systems of the form
\begin{equation}
\label{eq:symhypsysgeoopt}
    M(U)^\alpha\partial_\alpha U=F(U),
\end{equation}
where $U$ is a vector-valued unknown defined on $\mathbb{R}^{3+1}$. We refer to \cite{METIVIER2009169} and \cite{10.1215/S0012-7094-93-07007-X} for very complete descriptions of the subject. Typically, by analogy with electromagnetic waves, the amplitudes are propagated at the speeds given by the eigenvalues of the system and are polarized in the corresponding eigenspaces. In the monophase case, we look for approximate solutions in the asymptotic regime $\lambda\to0$ under the form of a WKB expansion
\begin{equation}
\label{eq:WKBexp1geoopt}
    U_{app}(t,x)=\lambda^p\sum_{j\geq0} \lambda^{j/2}U_{j}\left(t,x,\frac{u}{\lambda}\right),
\end{equation}
where $\lambda$ denotes the typical wavelength (or period), $p$ is a tuning parameter, $u$ is a phase function and where the profiles $U_j(t,x,\theta)$ are periodic with respect to their third variable. The function $U_{app}$ is an approximate solution at order $m$ if, for $0\leq j\leq m$, the profiles $U_j$ are solutions to a cascade of equations depending on the phase $u$, each equation being designed to cancel a corresponding power of $\lambda$.\\
The two main features of this hierarchy are the so-called eikonal equation\footnote{Written here for linear systems.} for the phase
\begin{equation}
\label{eq:eikonalgeoopt}
    \det(M^\alpha\partial_\alpha u)=0,
\end{equation}
and its associated polarisation condition 
\begin{equation}
\label{eq:polargeoopt}
    U_0\in \ker(M^\alpha\partial_\alpha u).
\end{equation}
A phase satisfying the eikonal equation is said to be characteristic. We do not detail here the transport equation governing the evolution of the amplitudes.
In the linear case, there is by definition neither interaction nor generation of harmonics. Consequently, one may consider a more precise ansatz of the form 
\begin{equation}
\label{eq:WKBexp2geoopt}
    U_{app}(t,x)=\lambda^p\sum_{j\geq0} \lambda^{j/2}e^{i\frac{u}{\lambda}}U_{j}(t,x).
\end{equation}
In the nonlinear case\footnote{In this chapter, we restrict ourselves to the semilinear case, where $M^{\alpha}$ does not depend on $U$.}, such a simple description is generally no longer valid, since harmonics might play a crucial role. In fact, it strongly depends on the parameter $p$. If the amplitude is large, that is, for $p$ small, nonlinear terms appear at leading order in the transport equations and produce harmonics. This regime is commonly referred to as weakly nonlinear optics. This highlights the fact that geometric optics takes into account the scaling properties of the problem. Nevertheless, certain structures in the nonlinearities might still prevent the appearance of nonlinearities in the transport equations, we then talk about transparency \cite{zbMATH01517171,zbMATH06295894}, and the generation of harmonics. A notable example is provided by null structures (see \cite{zbMATH06295894}). \\\\
In the multiphase case, the general ansatz becomes 
\begin{equation}
\label{eq:WKBexpmultigeoopt}
    U_{app}(t,x)=\lambda^p\sum_{j\geq0} \lambda^{j/2}U_{j}\left(t,x,\frac{\mathscr{U}}{\lambda}\right),
\end{equation}
where $\mathscr{U}$ denotes the vector containing all the phases. We recall that only finitely many phases are considered, so that the group of frequencies, which encodes phase interactions, is isomorphic to $\mathbb{Z}^d$ for some finite $d$. In order to apply geometric optics techniques, one must qualify the interactions between phases. A standard assumption is the strong coherence condition, namely that for any $z\in\mathbb{Z}^d$ and any phase $u_z=z\cdot\mathscr{U}$
\begin{equation}
\label{eq:coherencegeoopt}
      \forall(t,x)\in\mathbb{R}^{3+1}  \det(M^\alpha\partial_\alpha u_z)(t,x)=0 \;\;\textbf{or}\;\; \forall(t,x)\in\mathbb{R}^{3+1}  \det(M^\alpha\partial_\alpha u_z)(t,x)\neq0.
\end{equation}
In other words, the arising phase $u_z$ is either everywhere characteristic or nowhere characteristic. Under such an assumption, one can generally build approximate solutions at arbitrary order, see \cite{METIVIER2009169} for more details. \\

The next step consists of studying the stability of such approximate solutions, that is, the existence of an exact solution remaining close to the approximation on a fixed interval (independent of $\lambda$). The exact solution is constructed as the sum of the approximate solution (obtained by truncating the WKB expansion \eqref{eq:WKBexpmultigeoopt} at order $m$) and an error term $Z_\lambda$, leading to the parametrix
\begin{equation}
\label{eq:parametrixgeoopt}
    U_{\lambda}(t,x)=\lambda^p\sum^m_{j\geq0} \lambda^{j/2}U_{j}\left(t,x,\frac{\mathscr{U}}{\lambda}\right)+Z_\lambda.
\end{equation}
A natural question is to determine the minimal order of expansion required to ensure the existence of such an exact solution. As in many applications of the WKB method, regularity is typically lost at each order of expansion (see Section \ref{subsection:regubg}). Consequently, constructing the remainder term at low order allows one to work with a less regular initial ansatz. More generally, higher-order constructions necessitate the analysis of increasingly complicated hierarchical systems of equations.
\subsubsection{Application of geometric optics to KGM}
\label{subsubsection:applyKGM}
In this paper, we begin with the initial ansatz
 \begin{equation}
 \label{eq:initialansatzgeoopt}
        a_{1\lambda}^\alpha=a_0^\alpha+\lambda^{1/2}\sum_{\mathscr{A}\in\mathbb{A}}\left(p_\mathscr{A}^\alpha \cos\left(\frac{v_\mathscr{A}}{\lambda}\right)+q_\mathscr{A}^\alpha \sin\left(\frac{v_\mathscr{A}}{\lambda}\right)\right),\;\;\;\;
        \phi_{1\lambda}=\phi_0+\lambda^{1/2}\sum_{\mathscr{A}\in\mathbb{A}}\left(\psi_\mathscr{A}e^{i\frac{v_\mathscr{A}}{\lambda}}\right)
 \end{equation}
and seek solutions to \eqref{eq:KGMintro} constructed from this ansatz. A first fundamental obstruction preventing the direct application of the general results of \cite{METIVIER2009169} and \cite{10.1215/S0012-7094-93-07007-X} lies in the fact that the hyperbolic character of the KGM system depends on the choice of gauge. In our approach, we apply geometric optics directly to the hyperbolic reduced system KGML \eqref{eq:KGMLgeneKGM} (or schematically \eqref{eq:schemeqIntro}), obtained after imposing the Lorenz gauge condition. Then, we are able to recover exact solutions to \eqref{eq:KGMintro} provided that the initial data in \eqref{eq:initialansatzgeoopt} are polarized, satisfy the approximate Lorenz gauge condition and solve the Maxwell constraint. We refer to Section \ref{subsection:initansatz}, and in particular to Definition \ref{defi:admissiblebginitansatz}, for a precise notion of admissible initial data. In our approach, all gauge-related difficulties are ultimately dealt with in the error term. The second reason preventing the direct use of the results of \cite{METIVIER2009169} and \cite{10.1215/S0012-7094-93-07007-X} is that both KGM \eqref{eq:KGMintro} and KGML \eqref{eq:KGMLgeneKGM} are second order\footnote{Although it is possible to rewrite these systems in first-order form, this comes at the expense of clarity.} systems. In particular, nonlinearities involving derivatives, such as $\textbf{F}\partial\textbf{F}$, appear. This feature naturally leads us to consider a $\lambda^{1/2}$ scaling for the amplitude for the high-frequency perturbations. At this critical scaling we capture non-trivial effects such as backreaction, as shown in explicit computations in the introduction, while preserving the required stability of the WKB analysis. We also refer to \cite{zbMATH00567356,Cuccagna01011999,zbMATH01782982,zbMATH02063481,keel2010globalwellposednessmaxwellkleingordonequation,zbMATH07196901} for further discussion on the relationship between the $1/2$ coefficient, the scale-invariant regularity of solutions to KGM and the associated low-regularity Cauchy theory.\\\\
The initial ansatz \eqref{eq:initialansatzgeoopt} yields first-order approximate solutions of the form
 \begin{equation}
 \label{eq:WKBexp3geoopt}       A_{1\lambda}^\alpha=A_0^\alpha+\lambda^{1/2}\sum_{\mathscr{A}\in\mathbb{A}}\left(P_\mathscr{A}^\alpha \cos\left(\frac{u_\mathscr{A}}{\lambda}\right)+Q_\mathscr{A}^\alpha \sin\left(\frac{u_\mathscr{A}}{\lambda}\right)\right), 
            \;\;\;\; \Phi_{1\lambda}=\Phi_0+\lambda^{1/2}\sum_{\mathscr{A}\in\mathbb{A}}\left(\Psi_\mathscr{A}e^{i\frac{u_\mathscr{A}}{\lambda}}\right).
 \end{equation}
 The leading-order equations in the WKB hierarchy associated with KGM \eqref{eq:KGMintro} consist of the eikonal equation
 \begin{align*}
 \forall\mathscr{A}\in\mathbb{A},\;\partial^\beta u_{\mathscr{A}}\partial_\beta u_{\mathscr{A}}=0,
 \end{align*}
 together with the polarization condition  
  \begin{align*}
 \forall\mathscr{A}\in\mathbb{A},\;\partial_\beta u_{\mathscr{A}}P^\beta_{\mathscr{A}}=\partial_\beta u_{\mathscr{A}}Q^\beta_{\mathscr{A}}=0.
 \end{align*}
Then, the transport equations for the profiles read 
\begin{align}
\label{eq:transportgeoopt}
 &\forall\mathscr{A}\in\mathbb{A},\;(2\partial^\alpha u_{\mathscr{A}}\partial_\alpha+\Box u_{\mathscr{A}})\Psi_{\mathscr{A}}=-i2A^\alpha_0\partial_\alpha u_\mathscr{A}\Psi_\mathscr{A},\\
& \forall\mathscr{A}\in\mathbb{A},\;(2\partial^\alpha u_{\mathscr{A}}\partial_\alpha+\Box u_{\mathscr{A}})W_{\mathscr{A}}^\beta=i\partial^\beta u_{\mathscr{A}}\overline{\Psi_\mathscr{A}}\Phi_0,
\end{align}
where $W_{\mathscr{A}}=P_{\mathscr{A}}+iQ_{\mathscr{A}}$. Finally, the background satisfies 
\begin{align}
\Box A_0^\beta&=-\Im(\Phi_0\overline{\partial^{\beta}\Phi_0})+(A_0)^{\beta}|\Phi_0|^2+\sum_{\mathscr{A}\in\mathbb{A}}\partial^\beta u_{\mathscr{A}}|\Psi_{\mathscr{A}}|^2,\\
\Box\Phi_0&=-2i(A_0)^\alpha\partial_\alpha\Phi_0+(A_0)^\alpha (A_0)_\alpha\Phi_0.
\end{align}
To initiate our discussion on the approximate solution \eqref{eq:WKBexp3geoopt}, we first disregard interactions between distinct phases. Although, for general approximate solutions, one must consider expansions of the form \eqref{eq:WKBexp1geoopt} in order to capture harmonic generation, the approximate solutions \eqref{eq:WKBexp3geoopt} closely resemble WKB expansions of the linear case \eqref{eq:WKBexp2geoopt}. The reason is that no harmonics appear at order $O(\lambda^{-1/2})$, which corresponds precisely to the level of the transport equations. Moreover, due to a triangular structure of KGM, the transport equations \eqref{eq:transportgeoopt} are linear in the profiles. We emphasize that it is noteworthy to derive the linear transport equation at such a critical scaling. As one can see, the first equation of \eqref{eq:transportgeoopt} corresponds to the transport of the densities $|\Psi_\mathscr{A}|^2$ in KGM-null-transport \eqref{eq:KGMnullintro}. As for the second equation, we observe that, while the transport of $(W_{\mathscr{A}})_{\mathscr{A}\in\mathbb{A}}$ depends on the background, the phases and $(\Psi_{\mathscr{A}})_{\mathscr{A}\in\mathbb{A}}$, it does not feed back into their evolution. In particular, it does not contribute to the limiting dynamics of $(A_0,\Phi_0)$, namely the dynamics given by KGM-null-transport \eqref{eq:KGMnullintro}. We refer to Section \ref{section:Approx} for a detailed discussion of the approximate solution.\\\\
In fact, no harmonics appear at order $O(1)$ either. As explained in the introduction, in the Klein-Gordon equation, the polarization cancels the $O(1)$ interaction of the perturbations, thereby suppressing harmonic generation. In the Maxwell equation, only the zeroth harmonic survives and takes the form of the backreaction. We distinguish the backreaction from the related rectification phenomenon (see \cite{METIVIER2009169}), where the zero mode emerges dynamically and is not present in the initial data. In our case, the presence of the zero mode initially is induced by the Maxwell constraint equation, namely the second equation of \eqref{eq:constraintsgeneKGM}, which also exhibits the backreaction.\\\\
When phase interactions are taken into account, we observe $O(1)$ highly-oscillating terms arising in both Klein-Gordon and Maxwell equations. Under the strong coherence assumption \eqref{eq:coherencegeoopt}, each newly generated phase falls into one of two categories. If the emerging phase is non-characteristic (for non-resonant interactions), we proceed as in \cite{zbMATH07009768}, for example (see Section \ref{subsubsection:infoEell}). The d'Alembertian behaves (microlocally) like an elliptic operator for non-characteristic oscillations. Thus, we gain smallness while inverting (by hand) the wave operator via the elliptic error variable $\boldsymbol{\mathcal{E}}_\lambda^{ell}$. If, on the other hand, it is characteristic, we exploit the triangular structure of KGM to introduce a smaller perturbation, namely the error variable $(\breve{\textbf{F}}^+_{\mathscr{A}\pm\mathscr{B}})_{(\mathscr{A},\mathscr{B})\in\mathscr{C}}$, at the level of the error term, which absorbs these oscillating terms (see Section \ref{subsubsection:infofab}). Interactions involving these terms in the nonlinearities then produce smaller remainders. Implementing those methods at a higher order would require more general profile expansions and additional assumptions on the initial ansatz.\\\\
Finally, we address the gauge issue. This appears only at the level of the initial data for the error term. Indeed, as shown in Section \ref{subsubsection:propaggauegeneKGM}, if a solution to the KGML \eqref{eq:KGMLgeneKGM} system satisfies the constraints \eqref{eq:constraintsgeneKGM} at initial time, then the gauge condition is satisfied throughout the lifespan of the solution. Specifically, we know that, because of the approximate Lorenz gauge condition and the polarization,
\begin{align}
 \partial_\alpha (A_{1\lambda})^\alpha|_{t=0}&=O(\lambda^{1/2}),
 \end{align}
so that the initial gauge defect is of order $O(\lambda^{1/2})$ and can be absorbed into the divergence of the error term. Then, the constraint equations are of the form of the full KGM system \eqref{eq:KGMintro}. Thus, they contain the term $(\partial^\beta\partial_\alpha(A_{1\lambda})^\alpha)|_{t=0}$, which is of order $O(\lambda^{-1/2})$, by propagation of gauge and polarization for the background, and oscillates with non-stationary phases. Since Maxwell's constraint is elliptic for the time-component of the potential, we compensate for the apparent lack of smallness of  $(\partial^\beta\partial_\alpha(A_{1\lambda})^\alpha)|_{t=0}$ when inverting the operator, ensuring that the error is indeed small initially. We refer to Section \ref{subsection:Initconstr} for more details on the initial data of the error and to Remark \ref{rem:defectremApprox} on the treatment of the gauge defect.\\\\
We do not detail the bootstrap argument on the error term $\textbf{Z}_\lambda$, nor the precise role of $\boldsymbol{\mathcal{E}}_\lambda^{evo}$ or $(\textbf{F}^+_{\lambda\mathscr{A}})_{\mathscr{A}\in\mathbb{A}}$ at this stage of the analysis. We instead refer to Section \ref{subsection:Idea} for a rigorous description and to Section \ref{section:exactsolut} for the complete proof.
\subsection{Null forms and backreaction}
\label{subsection:nullforms}
We now extend the discussion on the null structure of KGM and, more specifically, on the $\textbf{F}\partial\textbf{F}$ nonlinearities. Non-schematically, these correspond to
\begin{align*}
&-2iA^\alpha\partial_\alpha\Phi, &-\Im(\Phi\overline{\partial^\beta\Phi}),
\end{align*}
for Klein-Gordon and Maxwell equations, respectively. We recall that null forms are exhaustively characterized in \cite{zbMATH03966906} and \cite{zbMATH03989952} as linear combinations of the following quadratic forms 
\begin{align}
\label{eq:nullforms}
&Q_0(f,g)=\partial_\alpha f\partial^\alpha g&Q_{\mu\nu}(f,g)=\partial_\mu f\partial_\nu g-\partial_\mu g\partial_\nu f. 
 \end{align}
In Coulomb gauge, KGM equations are a hyperbolic-elliptic system, where the component $A^0$ is determined by inverting a Laplacian. It is shown in \cite{10.1215/S0012-7094-94-07402-4,dfb0a54d-bfd7-3b95-ae19-04c1b390d84d} that the Klein-Gordon nonlinearity, as well as the spatial components of the Maxwell nonlinearity, can essentially be decomposed as 
a sum of null forms (for the right unknowns) plus elliptic terms involving $A^0$. This structural property, combined with smoothing estimates (see \cite{zbMATH00567356,zbMATH01782982}), allows the authors to prove that the Cauchy problem associated with KGM \eqref{eq:KGMintro} in Coulomb gauge is (globally) well-posed in energy space. This result was refined in \cite{keel2010globalwellposednessmaxwellkleingordonequation} to attain lower regularity classes. Later, in \cite{selberg2010finiteenergyglobalwellposednessmaxwellkleingordon}, it is shown that, in Lorenz gauge this time, the Klein-Gordon nonlinearity still enjoys a null structure, whereas the Maxwell nonlinearity does not. The null structure only appears at the level of the equation for the electromagnetic field
 \begin{align}
 \label{eq:nullformonF}
 \Box F_{\alpha\beta}=\partial_\alpha J_\beta-\partial_\beta J_\alpha.
 \end{align}
 We obtain \eqref{eq:nullformonF} using Maxwell's equation and the fact that $F$ is a closed form. Exploiting this structure, the authors of \cite{selberg2010finiteenergyglobalwellposednessmaxwellkleingordon} are able to establish the global well-posedness of KGM in Lorenz gauge in energy norms.\\\\The absence of a full null structure in the Maxwell nonlinearity is also reflected in the present work. Indeed, the backreaction arises precisely from this non-null interaction. A direct computation shows that null forms \eqref{eq:nullforms} prevent the apparition of backreactions. Thus, the emergence of the backreaction is a genuine manifestation of the failure of the null condition.
\subsection{Comparison with Einstein equations and the Burnett conjecture}
\label{subsection:Einsteincomp}
The Einstein vacuum equations describe a vacuum spacetime and take the form 
\begin{equation}
\label{eq:eve}
Ric(g)_{\mu\nu}=0
\end{equation}
where the unknown $g$ is a Lorentzian metric on a $(3+1)$-dimensional manifold $\mathcal{M}$ and where $Ric(g)$ denotes the Ricci curvature tensor associated with $g$. These equations are (gauge-)invariant under a change of coordinates. The Burnett conjecture \cite{zbMATH04086482} asserts that, given a family $(g_\lambda)_{0<\lambda<1}$ of solutions to \eqref{eq:eve} such that $g_\lambda\to g_0$ uniformly in compact sets and $g_\lambda\rightharpoonup g_0$ in $L^2_{loc}$, the limit metric $g_0$ is a weak solution to  the massless Vlasov-Einstein equations:
\begin{equation}
\begin{cases}
\label{eq:evevlasov}
Ric(g_0)_{\mu\nu}-\frac{1}{2}R(g_0)(g_0)_{\mu\nu}=\int_{g_0(p,p)=0}f(p)p_{\mu}p_{\nu}dp\\
p^\alpha\partial_\alpha f-p^\alpha p^\beta\Gamma(g_0)_{\alpha\beta}^\gamma\partial_\gamma f=0
\end{cases}
\end{equation}
where $R(g_0)$ is the scalar curvature, $\Gamma(g_0)$ the Christoffel symbols associated with $g_0$ and $f$ the density of massless particles on $T\mathcal{M}$. In other words, the Einstein vacuum equations exhibit a backreaction phenomenon in high-frequency limit. As discussed in Section \ref{subsection:nullforms}, this observation reflects the fact that the Einstein equations do not enjoy a complete null structure, but rather a weak null structure (or weak null-condition) as stated in \cite{zbMATH01981308}. The Burnett conjecture (or its reverse counterpart) is largely proved in various cases in  \cite{huneau2022trilinear,zbMATH07009768,zbMATH06921833,luk2020highfrequency,huneau2024burnetts,zbMATH07719650,zbMATH07732086,touati2024reverse} (see \cite{zbMATH07879915} for a review). The backreaction is, in particular, easier to read when working in wave gauge  $g^{\alpha\beta}\Gamma(g)_{\alpha\beta}^\gamma=0$, where the Einstein equations take the schematic form 
\begin{equation}
\label{eq:evesche}
\Box_{g} g=\partial g\partial g.
\end{equation}
A major distinction between KGM \eqref{eq:KGMintro} (or schematically \eqref{eq:schemeqIntro}) and Einstein \eqref{eq:eve} (or schematically \eqref{eq:evesche}) lies in the quasilinear nature of the latter, the wave operator $\Box_{g}$ depends on $g$, and the fact that it enjoys stronger quadratic nonlinearities. These stronger nonlinearities naturally lead to (first-order) approximate solutions under the form 
\begin{align}
g_{1\lambda}=g_0+\lambda\sum_{\mathscr{A}\in\mathbb{A}}\cos\left(\frac{u_\mathscr{A}}{\lambda}\right)g_1
\end{align}
where the scaling in $\lambda$ matches, in particular, the quadratic nonlinearity $\partial g\partial g$. We refer to \cite{zbMATH07009768}, with a $U(1)$-symmetry assumption, or \cite{touati2024reverse} in an approximate wave gauge. We observe that the amplitudes in this ansatz are smaller than those considered in Section \ref{subsubsection:applyKGM} for KGM. Due to the weaker nonlinearities in KGM, we are able to significantly relax our assumptions in comparison with the Einstein case. In particular, in our work, no smallness condition is imposed on the background initial data and resonant interactions can be treated effectively. The latter feature relies on an intrinsic triangular structure of the KGM system, which induces a shift in the cascade of WKB equations: resonant interactions arise at a scale strictly lower than that of the transport equations for the leading profiles. Specifically, applying our strategy to the Einstein equations would generate an infinite feedback loop. We further note that the shift in the WKB hierarchy of KGM is responsible for the linear nature of the transport at the critical scaling we consider. For the Einstein equations, the mechanism that provides the linearity is more subtle and genuinely corresponds to transparency, see \cite{zbMATH01517171,zbMATH06295894} and the discussion on this subject in \cite{touati2024reverse}.\\\\
The treatment of the gauge also differs between the two systems. For example, the Lorenz condition in KGM being linear allows for the treatment of high-frequency solutions in an exact gauge, while the analysis of \cite{touati2024reverse} provides solutions in an approximate wave gauge. Nevertheless, the two analyses share some similarities in that the gauge is linked with the polarization and that the polarization plays a role in managing the harmonics. \\\\
Finally, we point out that the backreaction of KGM is strictly weaker than that of Einstein. Indeed, the KGM-null-transport system \eqref{eq:KGMnullintro} possesses a triangular structure: the transport of the densities does not see the electromagnetic field. In contrast, the massless Vlasov-Einstein \eqref{eq:evevlasov} equations are self-consistent, since the densities are influenced by gravity through the $\Gamma(g_0)$ appearing in the transport equations. In other words, the dynamics of the backreaction and the background are coupled in the Einstein setting. 
\subsection{Acknowledgment}
The author thanks Cécile Huneau for her precious advice and her guidance during the elaboration of this paper, and Arthur Touati for the helpful discussions. The actual form of the paper is also due to the reviewer. The author thanks him for his time and his corrections, notably for the error spotted in the definition of angular separations of phases. 
\subsection{Outline of the paper}
\label{subsection:outline}
\begin{itemize}
    \item \label{item:sect1outline} In Section \ref{section:Notdef}, we provide some Notations and Definitions on the functional spaces, on the phases and on the multiphase high-frequency initial ansatz.
      \item \label{item:sect2outline} In Section \ref{section:results}, we state the two main Theorems and the steps of the proof.
      \item \label{item:sect3outline}In Section \ref{section:Approx}, we derive the evolution equations for the approximate solution. 
      \item  \label{item:sect4outline}In Section \ref{section:Init}, we give details on the initial data for the error term.  
        \item \label{item:sect5outline}In Section \ref{section:exactsolut}, we give the details on the error term, the existence of a local in time solution to KGML, the propagation of gauge and finally, the uniform time of existence with the smallness results and the bootstrap.
         \item \label{item:sect6outline}In Section \ref{section:conclu}, we conclude and put everything together. We make some other general remarks.
          \item \label{item:sect7outline} In Appendix \ref{section:Appendix}, we give technical parts of the proofs. 
         
\end{itemize}
\section{Notations and definitions}
\label{section:Notdef}
\subsection{General notations}
\label{subsection:genenot}
\begin{defi}
\label{defi:weightedSobgenenot}
    Let $m\in\mathbb{N}$, $1\leq p<+\infty$ and $\delta\in\mathbb{R}$. We define the weighted Sobolev space $W^{m,p}_{\delta}$ as the closure of $C^\infty_0$ with respect to the norm 
$$||u||_{W^{m,p}_{\delta}}=(\sum_{j\leq m}||(1+|x|^2)^{\frac{\delta+j}{2}}\partial^ju||_{L^p}^p)^{\frac{1}{p}}.$$
We also set $H^k_{\delta}=W^{k,2}_{\delta}$ and $L^p_\delta=W^{0,p}_{\delta}$. All these are Banach spaces.\\
\end{defi}
\begin{defi}
\label{defi:weightedholdergenenot}
 Let $m\in\mathbb{N}$ and $\delta\in\mathbb{R}$. We define the weighted Hölder space $C^{m}_{\delta}$ as the Banach space of $m$-times continuously differentiable functions on $\mathbb{R}^3$ endowed with the norm
$$||u||_{C^{m}_{\delta}}=\sum_{j\leq m}\underset{x\in\mathbb{R}^3}{\sup}|(1+|x|^2)^{\frac{\delta+j}{2}}\partial^ju(x)|.$$
\end{defi}
\begin{nota}
\label{nota:supportgenenot}
  For a function $f$ defined on $\mathbb{R}^3$, we denote by $Supp(f)$ its support. For a function of space and time, defined on $[0,t]\times\mathbb{R}^3$ for some $t>0$, we slightly abuse the notation and write $Supp(f)=U_{s\in[0,t]}Supp(f(s))$.\\
\end{nota}
\begin{nota}
\label{nota:supportpropaggenenot}
  For $\Omega$ a compact set of $\mathbb{R}^3$, we define the domain of influence
  $\mathscr{I}(t,\Omega):=\{x\in\mathbb{R}^3|\exists y\in\Omega,\; |x-y|\leq{t} \}$.\\
\end{nota}
\begin{defi}
    \label{defi:projfreq}
Given $\kappa>0$, we define the projectors on frequency
    \begin{equation}
\label{eq:projfreqIdea}
      \Pi_{\kappa,-}f(x)=\mathscr{F}^{-1}(\mathbf{1}_{|\xi|\leq \frac{1}{\lambda^\kappa}}\hat{f}(\xi))(x),
\end{equation}
and $\Pi_{\kappa,+}=1-\Pi_{\kappa,-}$.
\end{defi}

\subsection{Phases}
\label{subsection:phases}
\subsubsection{Generalities}
\label{subsubsection:genephase}
Here, we gather the definitions and basic properties related to phases.\\
\begin{defi}  
\label{defi:phasegenephase}
A phase $u(x,t)$ is defined as a scalar real-valued function of space and time whose spacetime gradient $\textbf{d}u^{\#}$ never vanishes, i.e., it has no critical points. In this paper, a phase is said to be \textbf{characteristic} if it satisfies the so-called \textbf{eikonal} equation associated with the wave operator for a Lorentzian metric $g$  
\begin{equation}
    \label{eq:eikonalgenephase}
    \partial^\alpha u\partial_\alpha u=g(\textbf{d}u^\#,\textbf{d}u^\#)=0.
\end{equation}
In this case, the phase is \textbf{isotropic} for the metric. For the Minkowski metric in Cartesian coordinates, this condition is equivalent to
\begin{equation}
    \label{eq:eikonalminkgenephase}
   -(\partial_t u)^2+|\nabla u|^2=0.
\end{equation}
Moreover, the phase is said to be future-directed if 
\begin{equation}
    \label{eq:futuredirgenephase}
    g(\textbf{d}u^\#,\partial_t)<0.
\end{equation}
\end{defi}
\begin{propal}
\label{propal:geodcharacgenephase}
 If the phase $u$ is characteristic for a Lorentzian metric $g$, then its gradient $\textbf{d}u^\#$ satisfies the geodesic equation 
 \begin{equation}
 \label{eq:geodgenephase}
    g^{\alpha\beta}\partial_\beta u \partial^2_{\alpha\gamma} u-\Gamma(g)^\beta_{\gamma\alpha}\partial^\alpha u\partial_{\beta} u=0
\end{equation}
where $\Gamma(g)^{\beta}_{\alpha\gamma}=\frac{1}{2}g^{\beta\mu}(\partial_\gamma g_{\alpha\mu}+\partial_\alpha g_{\mu\gamma}-\partial_\mu g_{\alpha\gamma})$ denote the Christoffel symbols of $g$.
\end{propal} 
\begin{proof} 
Differentiating the eikonal equation yields
\begin{align*}
    0&=\partial_\gamma(g^{\alpha\beta}\partial_\beta u \partial_{\alpha} u)\\
    &=2g^{\alpha\beta}\partial_\beta u\partial^2_{\alpha\gamma} u -2g^{\alpha\mu}\Gamma^\gamma_{\mu\beta}\partial_ \alpha u\partial_{\gamma}u.
\end{align*}
\end{proof}
\begin{rem}
\label{rem:christogenephase}
  The previous formulation is valid for any Lorentzian metric $g$ and for any system of coordinates. For the Minkowski metric expressed in Cartesian coordinates, the Christoffel symbols are 0.\\
\end{rem}
From now on, we restrict ourselves to the Minkowski metric in Cartesian coordinates.
\begin{defi}
\label{defi:initialphasegenephase}
   A pair $(v,\dot{v})\in C^2(\mathbb{R}^3)\times C^1(\mathbb{R}^3)$ is called \textbf{eikonal initial data} if
    \begin{equation}
        \label{eq:initeikonalgenephase}  
        (\dot{v})^2=|\nabla v|^2,
    \end{equation}
with $\dot{v}<-\eta$ for some $\eta>0$.\\
\end{defi}
\begin{propal}
\label{propal:initialphasegenephase}
   Given \textbf{eikonal initial data} $(v,\dot{v})\in C^2(\mathbb{R}^3)\times C^1(\mathbb{R}^3)$ in the sense of Definition \ref{defi:initialphasegenephase}, there exists some time $T_v>0$ such that  one can construct a unique future-directed characteristic phase $u(t,x)\in C^2([0,T_v]\times\mathbb{R}^{3})$ with initial data $u|_{t=0}=v$ and $\partial_tu|_{t=0}=\dot{v}$.
\end{propal}
\begin{proof}
We seek a solution $u$ to the eikonal equation
 \begin{equation}
 \label{eq:eikonalforpropal}
  \partial^\alpha u \partial_{\alpha} u=0,
\end{equation}
which leads to the geodesic equation 
 \begin{equation}
  \label{eq:geodesicforpropal}
  \partial^\alpha u \partial_{\alpha}\partial_{\beta} u=0,
\end{equation}
as established in Proposition \ref{propal:geodcharacgenephase}. Now, we can solve \eqref{eq:geodesicforpropal} locally with classical results for self-transport equations. The solution admits the representation formula 
\begin{equation}
\label{eq:represgenephase}
    \textbf{d}u^\#(\chi(y,t),t)=|\nabla v(y)|(1,\xi(y)),
\end{equation}
or, equivalently,
\begin{equation}
\begin{cases}
    \partial_tu(\chi(y,t),t)=-|\nabla v(y)|,\\
    \nabla u(\chi(y,t),t)=\nabla v(y),
\end{cases}
\end{equation}
where $\xi=\frac{\nabla v}{|\nabla v|}$ and $\chi(y,t)=y+t\xi(y)$. The solution exists as long as $\chi(t)$ remains a diffeomorphism of $\mathbb{R}^3$. The value of $u$ is then recovered by time integration of  $\partial_tu$.
\end{proof}
In this article, we restrict ourselves to only two types of interactions that contain all possible plane wave interactions (see the geometric optics Section \ref{subsection:geoopt}). The interaction of two phases must be either everywhere resonant or nowhere resonant. This distinction is encoded in the behavior of the scalar product of the phase gradients. 
\begin{defi} 
\label{defi:angularsepgenephase}
Two phases $u_\mathscr{A}$ and $u_\mathscr{B}$ are said to be \textbf{separated} at time $t$ if $$\forall x\in\mathbb{R}^3\;\frac{\partial_iu_\mathscr{A}\partial^iu_\mathscr{B}}{|\nabla u_\mathscr{A}||\nabla u_\mathscr{B}|}(x,t)<1.$$
\end{defi}
\begin{defi} 
\label{defi:angularcolgenephase}
Two phases $u_\mathscr{A}$ and $u_\mathscr{B}$ are said to be \textbf{positively aligned} at time $t$ if 
$$\forall x\in\mathbb{R}^3\;\frac{\partial_iu_\mathscr{A}\partial^iu_\mathscr{B}}{|\nabla u_\mathscr{A}||\nabla u_\mathscr{B}|}(x,t)=1.$$
\end{defi}
\begin{rem}
These notions naturally extend to couples of time-independent phases $(v_\mathscr{A},v_\mathscr{B})$, in particular to initial data for phases.
   Moreover, without any mention of $t$, positively aligned and separated is understood to hold throughout the common interval of existence of the corresponding couple $(u_\mathscr{A},u_\mathscr{B})$.
\end{rem}
\begin{propal}
\label{propal:angularforevergenephase}
Let $(v_\mathscr{A},\dot{v}_\mathscr{A})\in C^2(\mathbb{R}^3)\times C^1(\mathbb{R}^3)$ and $(v_\mathscr{B},\dot{v}_\mathscr{B})\in C^2(\mathbb{R}^3)\times C^1(\mathbb{R}^3)$ be eikonal initial data as in Definition \ref{defi:initialphasegenephase}, and let $u_\mathscr{A}$ and $u_\mathscr{B}$ be their respective developments from Proposition \ref{propal:initialphasegenephase} on some common interval $[0,T^\star]$ for $T^\star>0$.
\begin{itemize}
\item If $u_\mathscr{A}$ and $u_\mathscr{B}$ are positively aligned at initial time $t=0$, then they remain positively aligned for all $t\in[0,T^\star]$.
\item If $u_\mathscr{A}$ and $u_\mathscr{B}$ are separated at initial time $t=0$, then they remain separated for all $t\in[0,T^\star]$. 
\end{itemize}
\end{propal}
\begin{proof}
Assume first that the two phases are positively aligned at time $0$. Then, there exists $k:\mathbb{R}^3\to\mathbb{R}^\star_+$ such that  
\begin{align*}
 &\forall x\in\mathbb{R}^3,\;\nabla v_\mathscr{A}(x)=k(x)\nabla v_\mathscr{B}(x)\\
 &\Leftrightarrow\forall x\in\mathbb{R}^3,\;\nabla u_\mathscr{A}(x,0)=k(x)\nabla u_\mathscr{B}(x,0)\\
&\Leftrightarrow\forall x\in\mathbb{R}^3,\;\forall t\in[0,T*],\;\nabla u_\mathscr{A}(\chi_\mathscr{A}(x,t),t)=k(x)\nabla u_\mathscr{B}(\chi_\mathscr{A}(x,t),t)\\
&\Leftrightarrow\forall y\in\mathbb{R}^3,\forall t\in[0,T*],\;\nabla u_\mathscr{A}(y,t)=k(\chi_\mathscr{A}^{-1}(y,t))\nabla u_\mathscr{B}(y,t),
\end{align*}
where we use the representation formula and its inverse for both phases. In particular, we see that we substitute $\chi_\mathscr{B}(x,t)$ with $\chi_\mathscr{A}(x,t)$ because $\chi_\mathscr{B}(x,t)$ only depends on $\xi_\mathscr{B}=\frac{\nabla v_\mathscr{B}}{|\nabla v_\mathscr{B}|}=\frac{\nabla v_\mathscr{A}}{|\nabla v_\mathscr{A}|}=\xi_\mathscr{A}$ and so $\chi_\mathscr{B}(x,t)=\chi_\mathscr{A}(x,t)$. At a time $t$, two phases are positively aligned if and only if they were positively aligned at time $0$. This is due to the constant speed of propagation, i.e., $\xi_\mathscr{A}$ and $\xi_\mathscr{B}$ are normalized.\\\\
 The reverse argument applies for separated couples of phases. The idea is to look at the problem backward in time, with the new unknowns $(\nabla w_\mathscr{A}(y,s),\nabla w_\mathscr{B}(y,s))=(\nabla u_\mathscr{A}(y,t_1-s),\nabla u_\mathscr{B}(y,t_1-s))$ defined on $[0,t_1]$ for some final time $t_1\in[0,T*]$. We can also use the representation formula on $(\nabla w_\mathscr{A},\nabla w_\mathscr{B})$, where the respective flows also have constant speeds and only depend on the value of $(\nabla w_\mathscr{A}(\cdot,0),\nabla w_\mathscr{B}(\cdot,0))$. Thus, if there exists a point $y\in\mathbb{R}^3$ such that $\nabla w_\mathscr{A}(y,0)$ and $\nabla w_\mathscr{B}(y,0)$ are locally positively aligned, then there exists a point $x\in\mathbb{R}^3$ such that $\nabla w_\mathscr{A}(x,t_1)$ and $\nabla w_\mathscr{B}(x,t_1)$ are locally positively aligned. Translating this fact in terms of $(\nabla u_\mathscr{A},\nabla u_\mathscr{B})$ yields that if there exists a point $y\in\mathbb{R}^3$ such that $\nabla u_\mathscr{A}(y,t_1)$ and $\nabla u_\mathscr{B}(y,t_1)$ are locally positively aligned, then there must exist a point $x\in\mathbb{R}^3$ such that $\nabla u_\mathscr{A}(x,0)$ and $\nabla u_\mathscr{B}(x,0)$ are locally positively aligned. Thus, if two phases $u_\mathscr{A}$ and $u_\mathscr{B}$ are separated initially, they cannot become positively aligned at any point and thus remain separated.
\end{proof} 

\begin{propal}  
\label{propal:newnotcharasetphase}
Let $u_\mathscr{A},u_\mathscr{B}\in C^2([0,T],\mathbb{R}^3)$ be two future-directed characteristic phases. Then:
\begin{itemize}
\item The phases $u_\mathscr{A}$ and $u_\mathscr{B}$ are positively aligned if and only if $u_\mathscr{A}\pm u_\mathscr{B}$ is characteristic (resonant interaction). 
\item The phases $u_\mathscr{A}$ and $u_\mathscr{B}$ are separated if and only if $u_\mathscr{A}\pm u_\mathscr{B}$ is not characteristic (non-resonant interaction). 
\end{itemize}
\end{propal}
\begin{proof}
Recall that for $u_\mathscr{A}$ a characteristic future-directed phase, $\partial_tu_\mathscr{A}=-|\nabla u_\mathscr{A}|<0$. A direct computation gives
\begin{align*}
\partial_\beta(u_\mathscr{A}\pm u_\mathscr{B})\partial^\beta(u_\mathscr{A}\pm u_\mathscr{B})&=\pm2\partial_\beta u_\mathscr{A}\partial^\beta u_\mathscr{B},\\
&=\pm2(-|\nabla u_\mathscr{A}||\nabla u_\mathscr{B}|+\partial_i u_\mathscr{A}\partial^i u_\mathscr{B}).
\end{align*}
This implies that $u_\mathscr{A}$ and $u_\mathscr{B}$ are positively aligned, namely $\frac{\partial_i u_\mathscr{A}\partial^i u_\mathscr{B}}{|\nabla u_\mathscr{A}||\nabla u_\mathscr{B}|}=1$, if and only if $u_\mathscr{A}\pm u_\mathscr{B}$ is characteristic, namely $\partial_\beta(u_\mathscr{A}\pm u_\mathscr{B})\partial^\beta(u_\mathscr{A}\pm u_\mathscr{B})=0$. The same argument applies to the second point.
\end{proof}
 \begin{nota}
\label{nota:transportgenenot}
    For a given characteristic phase $u_\mathscr{A}$ (resp. $u_\mathscr{A}\pm u_\mathscr{B}$), the transport operator $2\partial^\alpha u_{\mathscr{A}}\partial_\alpha+\Box u_{\mathscr{A}}$ (resp. $2\partial^\alpha (u_\mathscr{A}\pm u_\mathscr{B})\partial_\alpha+\Box (u_\mathscr{A}\pm u_\mathscr{B})$) is denoted as $\mathscr{L}_\mathscr{A}$ (resp. $\mathscr{L}_{\mathscr{A}\pm\mathscr{B}}$).\\
\end{nota}
\subsubsection{Phase set}
\label{subsubsection:setphase}
\begin{defi} 
\label{defi:initialphasesetsetphase}
Let $|\mathbb{A}|=N$. A set $\{v_\mathscr{A},\dot{v}_\mathscr{A}|\mathscr{A}\in\mathbb{A}\}$ is called a (characteristic) \textbf{initial phase set} if each $(v_\mathscr{A},\dot{v}_\mathscr{A})$ is an \textbf{eikonal initial data} in the sense of Definition \ref{defi:initialphasegenephase}. For simplicity, we exclude the degenerate situation where $\nabla v_\mathscr{A}=\nabla v_\mathscr{B}$ for two distinct $\mathscr{A}$ and $\mathscr{B}$. We denote $\{u_\mathscr{A}|\mathscr{A}\in\mathbb{A}\}$ the associated \textbf{phase set} emerging naturally from the initial phase set with Proposition \ref{propal:initialphasegenephase}. Each $u_\mathscr{A}$ is a future-directed characteristic phase defined on $[0,T_{min}]\times\mathbb{R}^3$, where $T_{min}$ is the minimum of all individual existence times. We denote by $\mathscr{U}_\mathbb{A}$ the vector of size $N$ collecting all phases.
\end{defi}
\begin{defi} 
\label{defi:angularadaptedsetphase}
An initial phase set $\{v_\mathscr{A},\dot{v}_\mathscr{A}|\mathscr{A}\in\mathbb{A}\}$ (respectively a phase set $\{u_\mathscr{A}|\mathscr{A}\in\mathbb{A}\}$) is said to be \textbf{angularly adapted} or \textbf{strongly coherent},  if every couple $(v_\mathscr{A},v_\mathscr{B})$ (respectively $(u_\mathscr{A},u_\mathscr{B})$) with $\mathscr{A},\mathscr{B}\in\mathbb{A}$ the corresponding phases are either positively aligned or separated. 
\end{defi}

\begin{defi} 
\label{defi:cutphasesetsetphase}
For a given phase set $\{u_\mathscr{A}|\mathscr{A}\in\mathbb{A}\}$ of smooth future-directed characteristic phases, we denote by $\mathscr{C}\subset\mathbb{A}\times\mathbb{A}$ the set of couples of distinct phases that are positively aligned (or equivalently, that lead to resonant interactions) and by $\mathscr{S}\subset\mathbb{A}\times\mathbb{A}$ the set of couples of phases that are separated (or equivalently, that lead to non-resonant interactions). If the phase set is angularly adapted, then $\mathscr{C}\cup\mathscr{S}\cup_{\mathscr{A}\in\mathbb{A}}(\mathscr{A},\mathscr{A})=\mathbb{A}\times\mathbb{A}$.
\end{defi}
Under the strong coherence assumption, one obtains uniform lower bounds for several quantities that play a crucial role in the analysis.
\begin{propal}  
\label{propal:smallnesscontrolsetphase}
Let $\{v_\mathscr{A}|\mathscr{A}\in\mathbb{A}\}$ be a given strongly coherent characteristic initial phase set. Then, for $\{u_\mathscr{A}|\mathscr{A}\in\mathbb{A}\}$ the associated phase set, it holds that for any compact set $\Omega\subset\mathbb{R}^3$ and any $T_{min}>\tau>0$ there exists $\eta_0$, depending on $(\Omega,\tau)$, such that for all $(x,t)\in\Omega\times[0,\tau]$,
\begin{equation}
\label{eq:eta1}
\min_{\mathscr{A}\in \mathbb{A}}|\partial^0u_\mathscr{A}(x,t))|>\eta_0,
\end{equation}
\begin{equation}
\label{eq:eta5}
\min_{\mathscr{A}\in \mathbb{A}}|\nabla u_\mathscr{A}(x,t)|>\eta_0,
\end{equation}
\begin{equation}
\label{eq:eta3}
\min_{(\mathscr{A},\mathscr{B})\in \mathscr{S}}|\partial_i(u_\mathscr{A}- u_\mathscr{B})\partial^i(u_\mathscr{A}- u_\mathscr{B})(x,t)|>\eta_0,
\end{equation}
\begin{equation}
\label{eq:eta4}
\min_{(\mathscr{A},\mathscr{B})\in \mathscr{C}}|\partial_i(u_\mathscr{A}\pm u_\mathscr{B})\partial^i(u_\mathscr{A}\pm u_\mathscr{B})(x,t)|>\eta_0
\end{equation}
\begin{equation}
\label{eq:eta2}
\;\;\min_{(\mathscr{A},\mathscr{B})\in \mathscr{S}}|\partial_\beta(u_\mathscr{A}\pm u_\mathscr{B})\partial^\beta(u_\mathscr{A}\pm u_\mathscr{B})(x,t)|>\eta_0.
\end{equation}
\end{propal} 
\begin{proof}
We recall that the index set is finite $|\mathbb{A}|=N$ and that the phases have no critical points, see Definition \ref{defi:initialphasesetsetphase}. Firstly, the estimates \eqref{eq:eta1} and \eqref{eq:eta5} are inherited from the initial data of Definition \ref{defi:initialphasegenephase}, and even hold on $\Omega\times[0,\tau]$ for $\Omega$ non-compact.
Then, the estimate \eqref{eq:eta3} follows from the fact that $\mathscr{S}$ contains the couples of separated phases, which satisfy $-\partial_iu_\mathscr{A}\partial^iu_\mathscr{B}>-|\nabla u_\mathscr{A}||\nabla u_\mathscr{B}|$. In particular, the vector $\nabla(u_\mathscr{A}+u_\mathscr{B})$, contrary to $\nabla(u_\mathscr{A}-u_\mathscr{B})$, could be zero. Next, the estimate \eqref{eq:eta4} relies on the positive alignment of couples of phases in $\mathscr{C}$, together with the assumption that all phase gradients are distinct and non-vanishing. Finally, the estimate \eqref{eq:eta2} is a direct consequence of Proposition \ref{propal:newnotcharasetphase}.
\end{proof} 
\subsection{Initial ansatz}
\label{subsection:initansatz}
In this section, we introduce the required material to formulate the assumptions on the initial data for the background and the initial ansatz. Throughout this section and the rest of the paper, we abuse terminology and refer to all components of the initial ansatz, and later of the first-order approximate solution, as the background, since they are independent of the parameter $\lambda$. 
\begin{defi} 
\label{defi:bginitansatz} We call a \textbf{background initial data set} the 
 following set $(a^\alpha_0,\dot{a}^\alpha_0,\phi_0,\dot{\phi}_0,v_{\mathscr{A}},\dot{v}_\mathscr{A},\psi_\mathscr{A},w^\alpha_\mathscr{A})$, for $\mathscr{A}\in\mathbb{A}$ with $|\mathbb{A}|=N$,
 such that $(a^\alpha_0,\dot{a}^\alpha_0,\dot{\phi}_0)$ are real valued functions, $(\phi_0,\dot{\phi}_0,\psi_\mathscr{A},w^\alpha_\mathscr{A})$ are complex valued functions and where $(v_{\mathscr{A}},\dot{v}_\mathscr{A})_{\mathscr{A}\in\mathbb{A}}$ constitutes a characteristic initial phase set in the sense of Definition \ref{defi:initialphasesetsetphase}. We adopt the notation $p^\alpha_{\mathscr{A}}+iq^\alpha_{\mathscr{A}}=w^\alpha_{\mathscr{A}}$ for convenience.\\
\end{defi}
\begin{defi} 
\label{defi:initansatzinitansatz}
 For a given $\lambda$, we define the \textbf{initial ansatz} as 
 \begin{equation}
     \label{eq:truncinitpotential}
        a_{1\lambda}^\alpha=a_0^\alpha+\lambda^{1/2}\sum_{\mathscr{A}\in\mathbb{A}}\left(p_\mathscr{A}^\alpha \cos\left(\frac{v_\mathscr{A}}{\lambda}\right)+q_\mathscr{A}^\alpha \sin\left(\frac{v_\mathscr{A}}{\lambda}\right)\right),
 \end{equation}
 \begin{equation}
     \label{eq:truncinitwave}
        \phi_{1\lambda}=\phi_0+\lambda^{1/2}\sum_{\mathscr{A}\in\mathbb{A}}\psi_\mathscr{A}e^{i\frac{v_\mathscr{A}}{\lambda}}.
 \end{equation}
 We refer to $a_{1\lambda}^\alpha$ as the first-order approximation of the electromagnetic vector potential and to $\phi_{1\lambda}$ as the first-order approximation of the wave function. The functions $\phi_0$, $\phi_{1\lambda}$, $\psi_\mathscr{A}$ and $w^\alpha_{\mathscr{A}}$ are complex functions.\\
\end{defi}

\begin{defi}
\label{defi:constraintbginitansatz}
We call the \textbf{background constraint equations} the following system associated with a given background initial data set
       \begin{equation}
    \label{eq:constraintbginitansatz}
        -\Delta a_0^0-\partial_i(\dot{a_0}^i)=-\Im(\phi_0\overline{\dot{\phi_0}})-a^0_0|\phi_0|^2+\sum_{\mathscr{A}\in\mathbb{A}}\dot{v}_{\mathscr{A}}|\psi_{\mathscr{A}}|^2,
    \end{equation} 
    \begin{equation}
    \label{eq:initlorenzbginitansatz}   
    \dot{a}^0_0=-\partial_ia^i_0,
\end{equation}
    \begin{equation} 
    \label{eq:initeikonalinitansatz}   
   \forall \mathscr{A}\in\mathbb{A},\; (\dot{v}_\mathscr{A})^2=|\nabla v_\mathscr{A}|^2,
    \end{equation}
\begin{equation}
    \label{eq:initpolarinitansatz} 
     \forall \mathscr{A}\in\mathbb{A},\; \partial_i v_\mathscr{A}w_\mathscr{A}^i+\dot{v}_\mathscr{A}w_\mathscr{A}^0=0.
\end{equation}
\end{defi}
\begin{rem}
\label{rem:explainconstraintinitansatz}
The first three equations correspond to the constraints for the initial data of the KGM-null transport system in Lorenz gauge, whose velocities are given by $(|\nabla v_\mathscr{A}|,\nabla v_\mathscr{A})$ at initial time. More precisely, \eqref{eq:constraintbginitansatz} is the Maxwell constraint, \eqref{eq:initlorenzbginitansatz} is the Lorenz gauge condition, and \eqref{eq:initeikonalinitansatz} is the eikonal equation, which is automatically satisfied for characteristic initial phase sets.\\
The last equation \eqref{eq:initpolarinitansatz} is a polarization condition: the first-order term in the expansion \eqref{eq:truncinitpotential} is orthogonal to the direction of propagation, as discussed in Section \ref{subsection:geoopt}.
\end{rem}
The construction of the full high-frequency ansatz relies on the notion of \textbf{admissible} background initial data set. \\
\begin{defi}
\label{defi:admissiblebginitansatz}
    A background initial data set $(a^\alpha_0,\dot{a}^\alpha_0,\phi_0,\dot{\phi}_0,v_{\mathscr{A}},\dot{v}_\mathscr{A},\psi_\mathscr{A},w^\alpha_\mathscr{A})$ is said to be \textbf{admissible} if it satisfies the background constraint equations \eqref{eq:constraintbginitansatz}, \eqref{eq:initlorenzbginitansatz}, \eqref{eq:initeikonalinitansatz} and \eqref{eq:initpolarinitansatz}, if its initial phase set is strongly coherent in the sense of Definition \eqref{defi:angularadaptedsetphase}, and if the following regularity bounds hold
\begin{align*}
\label{regularityinitansatz}
 &||\phi_0||_{H^{3}}+||a^i_0||_{H^{3}}+ ||\dot{a}^i_0||_{H^2}+ ||\dot{\phi}_0||_{H^2}+ ||a^0_0||_{H^{3}_{\delta_0}}+||\dot{a}^i_0||_{H^2_{\delta_0}}\leq c'_0, \\
 &\max_{\mathscr{A}\in\mathbb{A}}(||\psi_{\mathscr{A}}||_{H^{3}}+||w_{\mathscr{A}}||_{H^{3}}+||v_{\mathscr{A}}||_{H^{5}_{\delta_1}}+||\dot{v}_{\mathscr{A}}||_{H^{4}_{\delta_1+1}})\leq c'_0,   
\end{align*}
for\footnote{We defined the weighted Sobolev norms in Definition \ref{defi:weightedSobgenenot}.} $-3/2<\delta_0$  and $\delta_1<-5/2$. No smallness assumption is imposed on the constant $c'_0>0$. \\
Moreover, we assume that 
\begin{align*} 
Supp(a^i_0,\dot{a}^i_0,\phi_0,\dot{\phi}_0,\psi_\mathscr{A},w^\alpha_\mathscr{A})\subset B_{S'},
\end{align*}
for some $S'\in\mathbb{R_+}$. Only the phases $(v_\mathscr{A})_{\mathscr{A}\in\mathbb{A}}$ and their gradients are not compactly supported, while $\dot{a}^0_0$ and $a^0_0$ need not be compactly supported. In particular, neither need belong to $L^2$.
\end{defi}
\begin{rem}
\label{rem:weightinitansatz}
The condition $-3/2<\delta_0$ is a compatibility condition imposed so that the background component $a^0_0$ belongs to the same weighted space $H^2_\delta$ as the full solution, where $-3/2<\delta<-1/2$ satisfying $\delta\leq \delta_0$ is later chosen to invert a Laplacian on $\mathbb{R}^3$, see Theorem \ref{unTheorem:laplcianAx}. Invertibility of the Laplacian is only required at the level of the error term, not for the background data. The condition $\delta_1<-5/2$ is chosen so as to include plane waves, namely phases whose gradients are constant and, in particular, do not decay.
\end{rem}

We seek exact solutions to Klein-Gordon-Maxwell equations in Lorenz gauge based on the initial ansatz of Definition \ref{defi:initansatzinitansatz}. To this end, we introduce an error that depends on the background initial data, as detailed in Section \ref{section:Init}.
\begin{defi}
\label{defi:initparaerrorinitansatz}
The \textbf{initial parametrices} are of the form
        \begin{align*}
        &a^\alpha_\lambda=a^\alpha_{1\lambda}+z^\alpha_\lambda, \\  &\phi_\lambda=\phi_{1\lambda}+\zeta_\lambda,
    \end{align*}
where $(z^\alpha_\lambda ,\zeta_\lambda)$ are referred to as \textbf{error initial data}.
\end{defi}
\subsection{Notion of approximate solution}
In this work, we construct approximate solutions to the reduced Klein-Gordon-Maxwell system in Lorenz gauge \eqref{eq:KGMLgeneKGM}, rather than to the full KGM system \eqref{eq:KGMintro}.  Moreover, the multiphase high-frequency framework naturally generates nonlinear interactions between distinct phases, producing highly oscillatory terms of large amplitude that only the strong coherence assumption allows us to treat. Thus, these terms must be distinguished from other (possibly non-oscillating) smaller remainders. These two specifications motivate the following notion of an almost approximate solution.
\begin{defi}
\label{defi:almostapproxresults}
        We say that $(A_{m\lambda},\Phi_{m\lambda})_{0<\lambda}$ is an \textbf{almost approximate solution} of order $m$ in $\lambda$ to KGM in Lorenz gauge (KGML) on a time interval $[0,T]$ if it satisfies 
    \begin{equation}
    \begin{cases}
    \label{eq:almostapproxsysresults}
\Box A^\beta_{m\lambda}+\Im(\Phi_{m\lambda}\overline{(\partial^{\beta}+iA^{\beta}_{m\lambda})\Phi_{m\lambda}})=\lambda^{m/2}\Xi^\beta_{A\lambda}+\lambda^{m/2-1/2}\tilde{\Xi}_{A\lambda}^\beta,\\
(\partial^{\alpha}+iA^{\alpha}_{m\lambda})(\partial_{\alpha}+i(A_{m\lambda})_\alpha )\Phi_{m\lambda}=\lambda^{m/2}\Xi_{\Phi\lambda}+\lambda^{m/2-1/2}\tilde{\Xi}_{\Phi\lambda}, \\
\partial_\alpha A^\alpha_{m\lambda}=\lambda^{m/2}\Xi_{L\lambda}.
    \end{cases}
    \end{equation}
      The remainders $\Xi_{A\lambda}$,  $\Xi_{\Phi\lambda}$ and $\Xi_{L\lambda}$ are assumed to be uniformly bounded in $L^2$ with respect to $\lambda$. The terms $\tilde{\Xi}_{A\lambda}$ and $\tilde{\Xi}_{\Phi\lambda}$ are finite sums ($N'\in\mathbb{N}$) of highly oscillatory contributions of the form
    \begin{align*}
        \lambda^{m-1/2}\sum_{N'\geq k\geq 0}e^{i\frac{u_k}{\lambda}}\nu_k,
    \end{align*} 
    where:
\begin{itemize}
\item the phases $(u_k)_{0\leq k\leq N'}$ are independent of $\lambda$, belong to
$\bigcap_{j=0}^{4+m} C^j([0,T],H^{4+m-j}_{\delta_1+j})$ for some $\delta_1<-5/2$, and are either everywhere characteristic or nowhere characteristic,
\item the amplitudes $(\nu_k)_{0\leq k\leq N'}$ are independent of $\lambda$, compactly supported and in
$\bigcap_{j=0}^{2+m} C^j([0,T],H^{2+m-j})$.
\end{itemize}
\end{defi}
\begin{rem}
\label{rem:jeanneapproxresults}
    In the monophase case, a similar definition for approximate solutions is introduced in \cite{zbMATH01799448}, see Théorème 2 therein, for example. 
\end{rem}
\begin{rem}
\label{rem:defectremApprox}
An important feature of Definition \ref{defi:almostapproxresults} is that the Lorenz gauge remainder $\Xi_{L\lambda}$ is itself highly oscillatory. Consequently, the Maxwell equation exhibits a worse apparent remainder than the Klein-Gordon equation. Indeed, the Maxwell equation involves the divergence of the Faraday tensor
$F_{\alpha\beta}=\partial_\alpha A_\beta-\partial_\beta A_\alpha$. For Definition \ref{defi:almostapproxresults}, at first order, one has
\begin{equation}
    \Box A^\beta_{1\lambda}+\Im(\phi_{1\lambda}\overline{(\partial^{\beta}+iA^{\beta}_{1\lambda})\phi_{1\lambda}})=\lambda^{1/2}\Xi^\beta_{A\lambda}(\lambda)+\tilde{\Xi}_{A\lambda}^\beta(\lambda)
\end{equation}
and therefore
 \begin{equation}
 \label{eq:maxdefect}
 \partial_\alpha F^{\alpha\beta}_{1\lambda}+\Im(\phi_{1\lambda}\overline{(\partial^{\beta}+iA^{\beta}_{1\lambda})\phi_{1\lambda}})=\lambda^{1/2}\Xi^\beta_{A\lambda}+\lambda^{0}\tilde{\Xi}_{A\lambda}^\beta-\lambda^{1/2}\partial^\beta \Xi_{L},
\end{equation}
which is of order $O(\lambda^{-1/2})$. This issue is particularly relevant for the constraint equation that is of the form of \eqref{eq:maxdefect} and must be solved at the level of the error term in order to obtain an exact solution.  However, since all contributions at order
$\lambda^{-1/2}$ are high-frequency and the constraint operator is elliptic, the mandatory smallness of the error can be recovered.
\end{rem}
\section{Main results}
\label{section:results}
This section presents the main results of the paper.
Theorem \ref{unTheorem:mainth1results} states that the multiphase WKB analysis for KGM is stable for the high-frequency ansatz considered. It is to compare with the result of \cite{zbMATH01799448}, which addresses the monophase case using a different approach. We refer to Section \ref{subsection:jeanneompare}. Theorem \ref{unTheorem:mainth2results} establishes the presence of backreaction at the high-frequency limit of solutions to KGM. We refer to discussions in the introductory Section \ref{section:Intro} and to Section \ref{subsection:Einsteincomp} for comparison with the Einstein case.
\subsection{Main Theorems}
\begin{unTheorem}
   \label{unTheorem:mainth1results}
Let $(a^\alpha_0,\dot{a}^\alpha_0,\phi_0,\dot{\phi}_0,v_{\mathscr{A}},\dot{v}_\mathscr{A},\psi_\mathscr{A},w^\alpha_\mathscr{A})$ be an admissible background initial data set in the sense of Definition \ref{defi:admissiblebginitansatz} and let $-3/2<\delta<-1/2$ with $\delta\leq\delta_0$. Then, there exists $\lambda_0>0$ such that for $0<\lambda<\lambda_0$:

\begin{enumerate}
\item \label{item:item1th1results} There exists a family $(A_\lambda,\Phi_\lambda)_{0<\lambda<\lambda_0}$ of high-frequency solutions to KGM in Lorenz gauge based on the initial ansatz of Definition \ref{defi:initansatzinitansatz} corresponding to the background initial data set. This family satisfies $(A_\lambda^i,\Phi_\lambda)\in (\cap^2_{j=0}C^j([0,T],H^{2-j}))^2$ and $A_\lambda^0\in \cap^2_{j=0}C^j([0,T],H^{2-j}_\delta)$, where the time $T>0$ depends only on the background initial data set and, in particular, is independent of $0<\lambda<\lambda_0$.\\
   \item \label{item:item2th1results} These solutions admit the following decomposition, under the form of parametrices,
\begin{align*}
        &A^\alpha_\lambda=\underbrace{A^\alpha_0+\lambda^{1/2}\sum_{\mathscr{A}\in\mathbb{A}}\left(\cos\left(\frac{u_\mathscr{A}}{\lambda}\right)P^\alpha_\mathscr{A}+\sin\left(\frac{u_\mathscr{A}}{\lambda}\right)Q^\alpha_\mathscr{A}\right)}_\text{$A^\alpha_{1\lambda}$}+Z_\lambda^\alpha, &\Phi_\lambda=\underbrace{\Phi_0+\lambda^{1/2}\sum_{\mathscr{A}\in\mathbb{A}}e^{i\frac{u_\mathscr{A}}{\lambda}}\Psi_\mathscr{A}}_\text{$\Phi_{1\lambda}$}+\mathcal{Z}_\lambda,
    \end{align*}
    where $(A^\alpha_{1\lambda},\Phi_{1\lambda})$ is an almost approximate solution of order one to KGM in Lorenz gauge in the sense of Definition \ref{defi:almostapproxresults}, while $(Z_\lambda,\mathcal{Z}_\lambda)$ denotes the error term. The background\footnote{We write $W^\beta_\mathscr{A}=P^\beta_\mathscr{A}+iQ^\beta_\mathscr{A}$ for convenience.} $(A_0^0,A_0^i,\Phi_0,\Psi_\mathscr{A},W_\mathscr{A},u_\mathscr{A})$ belongs to $\left(\cap^3_{j=0}C^{j}([0,T],H^{3-j}_{\delta_0+j})\right)\times\left(\cap^3_{j=0}C^{j}([0,T],H^{3-j})\right)^4\times\left(\cap^5_{j=0}C^{j}([0,T],H^{5-j}_{\delta_1+j})\right)$ and the error $(Z_{\lambda}^0,Z_{\lambda}^i,\mathcal{Z}_{\lambda})$ to $\left(\cap^2_{j=0}C^{j}([0,T],H^{2-j}_{\delta_0+j})\right)\times\left(\cap^2_{j=0}C^{j}([0,T],H^{2-j})\right)^2$.\\
\item \label{item:item3th1results}For some $C>0$, depending only on $c_0$ from Definition \ref{defi:admissiblebginitansatz}, such that the following uniform bounds hold 
\begin{align*}
    &\forall 0<\lambda<\lambda_0,\; \forall t\in[0,T],\;||A_\lambda^0(t)||_{H^{1/2}_{\delta}}+||A_\lambda^i(t)||_{H^{1/2}}+||\Phi_\lambda(t)||_{H^{1/2}}<C,\\
    &\forall 0<\lambda<\lambda_0,\; \forall t\in[0,T],\;||Z_{\lambda}^0(t)||_{H^{1/2}_{\delta}}+||Z_{\lambda}^i(t)||_{H^{1/2}}+||\mathcal{Z}_{\lambda}(t)||_{H^{1/2}}<\lambda^{1/2}C,
\end{align*}
and more generally
\begin{align*}
    &\forall 0<\lambda<\lambda_0,\; \forall s\in[0,1/2],\;  \forall t\in[0,T],\;||Z_{\lambda}^0(t)||_{H^{1-s}_{\delta}}+||Z_{\lambda}^i(t)||_{H^{1-s}}+||\mathcal{Z}_{\lambda}(t)||_{H^{1-s}}<\lambda^{s}C.
\end{align*}
\end{enumerate}
\end{unTheorem}
\begin{unTheorem}
      \label{unTheorem:mainth2results}
     Under the assumptions of Theorem \ref{unTheorem:mainth1results}:
     \begin{enumerate}
\item  \label{item:item1th2results} The following convergence holds
     \begin{align*}
    &\forall t\in[0,T],\;(A^i_\lambda,\Phi_\lambda)(t)\xrightarrow{L^2}(A^i_0,\Phi_0)(t),
     &A^0_\lambda(t)\xrightarrow{L^2_{\delta}}A_0^0(t),\\
    &\forall t\in[0,T],\;(\partial A^i_\lambda,\partial\Phi_\lambda)(t)\xrightharpoonup{L^2}(\partial A^i_0,\partial\Phi_0)(t), 
     &\partial A^0_\lambda(t)\xrightharpoonup{L^2}\partial A_0^0(t),\\
     &\forall 1\leq p<\infty,\; (A^\alpha_\lambda,\Phi_\lambda)\xrightarrow{L^p([0,T],L^\infty)}(A_0^\alpha,\Phi_0).
     \end{align*}
     \item  \label{item:item2th2results} The limit $(A_0,\Phi_0)$ is solution to the Klein-Gordon-Maxwell-null-transport system \eqref{eq:KGMnullintro}, where the velocities are given by the family $(\textbf{d}u_\mathscr{A}^\#)_{\mathscr{A}\in\mathbb{A}}$ and the associated densities by $(|\Psi_\mathscr{A}|^2)_{\mathscr{A}\in\mathbb{A}}$.
     \end{enumerate}
\end{unTheorem}
\subsection{Steps of the proof}
\label{subsection:Idea}
The proof of both Theorems relies on the construction of the multiphase high-frequency solutions (point \ref{item:item1th1results} of Theorem \ref{unTheorem:mainth1results}). Thus, we focus on this part of the argument and describe in detail the structure of the error term and its various components. \\
Given an initial ansatz satisfying the assumptions of Theorem \ref{unTheorem:mainth1results}, namely admissible in the sense of Definition \ref{defi:admissiblebginitansatz}, the five main steps of the construction are:
\begin{enumerate}[label=(\roman*)]
\label{itemIdea}
    \item\label{item:item1Idea} Construction of an approximate solution based on the initial ansatz via geometric optics.
    \item\label{item:item2Idea} Preparation of suitable initial data for the error components.
    \item\label{item:item3Idea} Construction of an exact solution to KGML around the approximate solution, with a specific structure for the error term.
    \item\label{item:item4Idea} Proof that the time of existence of the solutions is independent of $0<\lambda<\lambda_0$ by means of a bootstrap argument.
        \item\label{item:item5Idea} Propagation of the Lorenz gauge condition and recovery of a solution to KGM.
\end{enumerate}
We do not detail here step \ref{item:item1Idea}, since the construction of the approximate solution is standard and has already been presented Section \ref{subsubsection:applyKGM}, see Section \ref{section:Approx} for the full derivation. We postpone the discussion of step \ref{item:item2Idea} because the decomposition of the error term is more transparent at the level of the evolution equations and the constraint equations are closely related to the gauge issue of step \ref{item:item5Idea}. Consequently, we begin with the analysis of the hyperbolic system of equations KGML \eqref{eq:KGMLgeneKGM}, namely points \ref{item:item3Idea} and \ref{item:item4Idea}, and we do not yet consider the problem of the gauge.\\\\
The structure of KGML equations \eqref{eq:KGMLgeneKGM} is captured by the following schematic equation for $\textbf{F}_\lambda$, 
\begin{equation}
\label{eq:schemeqIdea}
    \Box \textbf{F}_\lambda=\textbf{F}_\lambda\partial\textbf{F}_\lambda+(\textbf{F}_\lambda)^3,
\end{equation}
with 
    \begin{equation}
        \label{eq:schemfullparametrixIdea}       \textbf{F}_\lambda=\underbrace{\textbf{F}_0+\lambda^{1/2}\sum_{\mathscr{A}\in\mathbb{A}}e^{i\frac{u_\mathscr{A}}{\lambda}}\textbf{F}_\mathscr{A}}_\text{$F_{1\lambda}$}+\textbf{Z}_\lambda,
    \end{equation}
    where $\textbf{Z}_\lambda$ is the \textbf{error term}, while $F_{1\lambda}$ forms a \textbf{first-order} almost approximate solution to KGM \eqref{eq:KGMintro} in the sense of Definition \ref{defi:almostapproxresults} constructed on the \textbf{background} $\textbf{F}_0$ and $\textbf{F}_\mathscr{A}$ obtained in Section \ref{section:Approx}. \\\\
Restricting the expansion to the first order $F_{1\lambda}$ does not provide much smallness in high-order Sobolev norms. In particular, one does not obtain uniform pointwise bounds in $\lambda$ on compact sets. This makes the bootstrap argument delicate. On the other hand, this choice significantly simplifies the description of the approximate solution and the treatment of phase interactions.\\\\
In what follows, the Big-$O$ notation is understood in the $L^2$ sense and may involve high-frequency terms, which typically exhibit worse behavior in higher Sobolev norms. In contrast to Definition \ref{defi:almostapproxresults}, a classical definition of an approximate solution of order one yields the following error equation
\begin{equation}
\label{eq:erroreq1Idea}
     \Box\textbf{Z}_\lambda=O(\lambda^{1/2})+[\ldots],
\end{equation}
where the $O(\lambda^{1/2})$-term is high-frequency and depends on the background, and where the $[\ldots]$-term involves $\textbf{Z}_\lambda$. Then, one could hopefully obtain\footnote{Ignoring weighted Sobolev norms, for simplicity.}
\begin{align*}
 &\sum_{k\leq 1}( \lambda^k||\textbf{Z}_\lambda(t)||_{H^{k+1}}+\lambda^k||\partial_t\textbf{Z}_\lambda(t)||_{H^k})+\lambda||\partial^2_{tt}\textbf{Z}_\lambda(t)||_{L^2}\sim \lambda^{1/2},
 \end{align*}
with natural energy estimates, and close the bootstrap argument. There are obstacles to such an ideal situation. Indeed, in our case, the evolution equation for the error corresponds to
\begin{equation}
\label{eq:erroreq2Idea}
     \Box\textbf{Z}_\lambda=\tilde{\boldsymbol{\Xi}}_{\lambda}+\textbf{Z}_\lambda\partial(\textbf{F}_{1\lambda})+\textbf{Z}_\lambda\partial\textbf{Z}_\lambda+O(\lambda^{1/2})+[\ldots],
\end{equation}
where $[\ldots]$ now gathers harmless terms. To address these difficulties, we introduce the refined parametrix 
\begin{equation}
\label{eq:schemparaerrorIdea}
\textbf{Z}_\lambda=\sum_{\mathscr{A}\in\mathbb{A}}\lambda^{1}\textbf{F}^+_{\lambda\mathscr{A}}e^{i\frac{u_\mathscr{A}}{\lambda}}+\sum_{(\mathscr{A},\mathscr{B})\in\mathscr{C}}\lambda^{1}\breve{\textbf{F}}^+_{\mathscr{A}\pm\mathscr{B}}e^{i\frac{u_{\mathscr{A}}\pm u_{\mathscr{B}}}{\lambda}}+\boldsymbol{\mathcal{E}}_\lambda^{ell}+\boldsymbol{\mathcal{E}}_\lambda^{evo}.
\end{equation}  
We now examine in greater detail the various obstructions appearing in the terms of \eqref{eq:erroreq2Idea}:
\begin{enumerate}[label=\arabic*)]
    \item \label{item:obst1Idea} The term $\tilde{\boldsymbol{\Xi}}_{\lambda}$ contains $O(1)$ high-frequency interaction terms.
    \item \label{item:obst2Idea} The term $\textbf{Z}_\lambda\partial(\textbf{F}_{1\lambda})$ contains schematically $\textbf{Z}_\lambda\lambda^{-1/2}\sum_{\mathscr{A}\in\mathbb{A}}e^{i\frac{u_\mathscr{A}}{\lambda}}\partial u_\mathscr{A}\textbf{F}_\mathscr{A}$, which exhibit a loss of smallness.
       \item \label{item:obst3Idea} The term $\Box\textbf{Z}_\lambda$ contains $\lambda\sum_{\mathscr{A}\in\mathbb{A}}e^{i\frac{u_\mathscr{A}}{\lambda}}\Box \textbf{F}^+_{\lambda\mathscr{A}}$  which induces a loss of derivative.
       \item \label{item:obst4Idea} The term $\textbf{Z}_\lambda\partial\textbf{Z}_\lambda$ contains $\boldsymbol{\mathcal{E}}_\lambda^{evo}\partial\boldsymbol{\mathcal{E}}_\lambda^{evo}$ that cannot be controlled by Sobolev product estimates and embeddings. 
       \item  \label{item:obst5Idea} The term $\textbf{Z}_\lambda\partial\textbf{Z}_\lambda$ contains $i\boldsymbol{\mathcal{E}}_\lambda^{evo}\sum_{\mathscr{A}\in\mathbb{A}}\partial u_\mathscr{A}e^{i\frac{u_\mathscr{A}}{\lambda}}\textbf{F}^+_{\lambda\mathscr{A}}$ that cannot be controlled either. 
   \end{enumerate}
\subsubsection*{\ref{item:obst1Idea}-Phase interactions}
Due to the strong coherence assumption on the admissible initial data Definition \ref{defi:admissiblebginitansatz}, we have the decomposition 
\begin{align*}
\tilde{\boldsymbol{\Xi}}_\lambda=\sum_{(\mathscr{A},\mathscr{B})\in\mathscr{C}}e^{i\frac{u_\mathscr{B}\pm u_\mathscr{B}}{\lambda}}\boldsymbol{K}_{(\mathscr{A}\pm\mathscr{B})}+\sum_{(\mathscr{A},\mathscr{B})\in\mathscr{S}}e^{i\frac{u_\mathscr{B}\pm u_\mathscr{B}}{\lambda}}\textbf{S}_{(\mathscr{A}\pm\mathscr{B})}.
\end{align*}
The resonant interactions, indexed by $\mathscr{C}$, are dealt with $(\breve{\textbf{F}}^+_{\mathscr{A}\pm\mathscr{B}})_{(\mathscr{A},\mathscr{B})\in\mathscr{C}}$. A direct computation yields
\begin{align*}
\Box\left(\sum_{(\mathscr{A},\mathscr{B})\in\mathscr{C}}\lambda^{1}\breve{\textbf{F}}^+_{\mathscr{A}\pm\mathscr{B}}e^{i\frac{u_{\mathscr{A}}\pm u_{\mathscr{B}}}{\lambda}}\right)&=-\sum_{(\mathscr{A},\mathscr{B})\in\mathscr{C}}\lambda^{-1}\breve{\textbf{F}}^+_{\mathscr{A}\pm\mathscr{B}}\underbrace{\partial_\alpha(u_{\mathscr{A}}\pm u_{\mathscr{B}})\partial^\alpha(u_{\mathscr{A}}\pm u_{\mathscr{B}})}_\text{$=0$}e^{i\frac{u_{\mathscr{A}}\pm u_{\mathscr{B}}}{\lambda}}\\
&+ i\sum_{(\mathscr{A},\mathscr{B})\in\mathscr{C}}\lambda^{0}\mathscr{L}_{\mathscr{A}\pm\mathscr{B}} \breve{\textbf{F}}^+_{\mathscr{A}\pm\mathscr{B}}e^{i\frac{u_{\mathscr{A}}\pm u_{\mathscr{B}}}{\lambda}}+O(\lambda),
\end{align*}
where the transport operator $\mathscr{L}_{\mathscr{A}\pm\mathscr{B}}$ is defined in Notation \ref{nota:transportgenenot}. Thus, the $O(1)$ resonant interactions are eliminated by solving
\begin{equation}
\label{eq:FABeqIdea}
    \mathscr{L}_{\mathscr{A}\pm\mathscr{B}} \breve{\textbf{F}}^+_{\mathscr{A}\pm\mathscr{B}} =\boldsymbol{K}_{(\mathscr{A}\pm\mathscr{B})}+[\ldots].
\end{equation}
We refer to Section \ref{subsubsection:infofab} for detailed estimates. Then, the non-resonant interactions, indexed by $\mathscr{C}$, are dealt with $\boldsymbol{\mathcal{E}}_\lambda^{ell}$. The idea is to invert the d'Alembertian by hand. We set 
\begin{align*}
\boldsymbol{\mathcal{E}}_\lambda^{ell}=-\lambda^2\sum_{(\mathscr{A},\mathscr{B})\in\mathscr{S}}e^{i\frac{u_\mathscr{B}\pm u_\mathscr{B}}{\lambda}}\frac{\textbf{S}_{(\mathscr{A}\pm\mathscr{B})}}{\partial_\alpha(u_\mathscr{B}\pm u_\mathscr{B})\partial^\alpha(u_\mathscr{B}\pm u_\mathscr{B})}
\end{align*}
so that 
\begin{align*}
\Box\left(\boldsymbol{\mathcal{E}}_\lambda^{ell}\right)=\sum_{(\mathscr{A},\mathscr{B})\in\mathscr{S}}e^{i\frac{u_\mathscr{B}\pm u_\mathscr{B}}{\lambda}}\textbf{S}_{(\mathscr{A}\pm\mathscr{B})}+O(\lambda),
\end{align*}
where we gain smallness but lose derivatives on the background. We refer to Section \ref{subsubsection:infoEell} for detailed estimates.
Both $\boldsymbol{\mathcal{E}}_\lambda^{ell}$ and $\breve{\textbf{F}}^+_{\mathscr{A}\pm\mathscr{B}}$ depend only on the background, and so their time of existence too. Therefore, we treat the two \textbf{decoupled error components} $\boldsymbol{\mathcal{E}}_\lambda^{ell}$ and $(\breve{\textbf{F}}^+_{\mathscr{A}\pm\mathscr{B}})_{(\mathscr{A},\mathscr{B})\in\mathscr{C}}$ as background terms in the sense that their estimates are not part of the bootstrap argument. Nonetheless, these terms are of the order of the error term in the decomposition \eqref{eq:schemparaerrorIdea} and therefore correspond to error terms.
\subsubsection*{\ref{item:obst2Idea}-Error system}
The problematic term $\textbf{Z}_\lambda\lambda^{-1/2}\sum_{\mathscr{A}\in\mathbb{A}}e^{i\frac{u_\mathscr{A}}{\lambda}}\partial u_\mathscr{A}\textbf{F}_\mathscr{A}$ is highly oscillating along \textbf{characteristic} phases. This structure allows it to be absorbed by introducing new oscillatory components, $(\textbf{F}^+_{\lambda\mathscr{A}})_{\mathscr{A}\in\mathbb{A}}$, that also oscillate along these phases. At this stage, the only remaining component of the error that must be controlled dynamically is $\boldsymbol{\mathcal{E}}_\lambda^{evo}$, which absorbs all remaining contributions and is coupled to $(\textbf{F}^+_{\lambda\mathscr{A}})_{\mathscr{A}\in\mathbb{A}}$. See Sections \ref{subsection:errterm} and \ref{subsubsection:evosys} for the construction.\\\\
The quantities $\boldsymbol{\mathcal{E}}_\lambda^{evo}$ and $\textbf{F}^+_{\lambda\mathscr{A}}$ must satisfy the coupled system 
\begin{equation}
\begin{cases}
\label{eq:schemsys1Idea}
&\Box\boldsymbol{\mathcal{E}}_\lambda^{evo}=\textbf{Z}_\lambda\partial\textbf{Z}_\lambda-\sum_{\mathscr{A}\in\mathbb{A}}\lambda^{1}\Box\textbf{F}^+_{\lambda\mathscr{A}}e^{i\frac{u_\mathscr{A}}{\lambda}}+O(\lambda^{1/2})+[\ldots],\\
&\mathscr{L}_\mathscr{A} \textbf{F}^+_{\lambda\mathscr{A}} =\lambda^{-1/2}\partial u_\mathscr{A}\boldsymbol{\mathcal{E}}_\lambda^{evo}\textbf{F}_\mathscr{A}  +[\ldots].\\
\end{cases}
\end{equation}
The next logical step is to show that \eqref{eq:schemsys1Idea} is actually well-posed. Classical energy estimates for a wave-transport system give us 
\begin{align}
&\partial\boldsymbol{\mathcal{E}}_\lambda^{evo}\sim\partial\partial \textbf{F}^+_{\lambda\mathscr{A}},\\
&\textbf{F}^+_{\lambda\mathscr{A}}\sim\boldsymbol{\mathcal{E}}_\lambda^{evo},
\end{align}
in terms of derivatives. We observe a loss of derivative due to $\Box\textbf{F}^+_{\lambda\mathscr{A}}$. Without this term, the system is well-posed for $\boldsymbol{\mathcal{E}}_\lambda^{evo}$ and $\textbf{F}^+_{\lambda\mathscr{A}}$ in $H^2$ by standard results. This leads us to problem \ref{item:obst3Idea}. 
\subsubsection*{\ref{item:obst3Idea}-Loss of derivative}
To compensate for the apparent loss of derivative, we introduce an auxiliary variable $\textbf{G}^+_{\lambda\mathscr{A}}:=\Box \textbf{F}_\mathscr{A}^+$. Since $\textbf{F}^+_{\lambda\mathscr{A}}$ satisfies a transport equation, one can improve the regularity of $\Box \textbf{F}^+_{\lambda\mathscr{A}}$ (if the initial data are well chosen, see Section \ref{subsection:criteria}) by deriving a transport equation for $\Box \textbf{F}_\mathscr{A}^+$. The auxiliary variable $\textbf{G}^+_{\lambda\mathscr{A}}$ satisfies better estimates than two standard derivatives of $\textbf{F}_\mathscr{A}^+$. The system \eqref{eq:schemeqIdea} is therefore replaced by 
\begin{equation}
\begin{cases}
\label{eq:schemsys2Idea}
&\Box\boldsymbol{\mathcal{E}}_\lambda^{evo}=\textbf{Z}_\lambda\partial\textbf{Z}_\lambda-\sum_{\mathscr{A}\in\mathbb{A}}\lambda^{1}\textbf{G}^+_{\lambda\mathscr{A}}e^{i\frac{u_\mathscr{A}}{\lambda}}+O(\lambda^{1/2})+[\ldots],\\
&\mathscr{L}_\mathscr{A} \textbf{F}^+_{\lambda\mathscr{A}} =\lambda^{-1/2}\partial u_\mathscr{A}\boldsymbol{\mathcal{E}}_\lambda^{evo}\textbf{F}_\mathscr{A}  +[\ldots],\\
&\mathscr{L}_\mathscr{A} \textbf{G}^+_{\lambda\mathscr{A}} =[\mathscr{L}_\mathscr{A},\Box]\textbf{F}_{\lambda\mathscr{A}}^+ +\Box(\mathscr{L}_\mathscr{A} \textbf{F}^+_{\lambda\mathscr{A}} ).\\
\end{cases}
\end{equation}
The commutator estimates are given in Section \ref{subsection:commut}. The quantities $\boldsymbol{\mathcal{E}}_\lambda^{evo}$,  $\textbf{F}^+_{\lambda\mathscr{A}}$ and $\textbf{G}^+_{\lambda\mathscr{A}}$ are the \textbf{coupled error components}.
The last term that needs special care is $\textbf{Z}_\lambda\partial\textbf{Z}_\lambda$. This brings us to problems \ref{item:obst4Idea} and \ref{item:obst5Idea}.
\subsubsection*{\ref{item:obst4Idea}-Product estimates}
To control $\boldsymbol{\mathcal{E}}_\lambda^{evo}\partial\boldsymbol{\mathcal{E}}_\lambda^{evo}$ one must resort to dispersion. Without spacetime product estimates, the standard Sobolev product estimates and embeddings are not sufficient, see Remark \ref{rem:badestiestiEevo}. Thus, we state refined spacetime product estimates involving Strichartz estimates as defined in Lemmas \ref{lem:lemma1Ax} and \ref{lem:lemma2Ax}. In general, we refer to Section \ref{subsubsection:estiEevo} for more details.
\subsubsection*{\ref{item:obst5Idea}-Projector on frequency}
Finally, one must control $i\boldsymbol{\mathcal{E}}_\lambda^{evo}\sum_{\mathscr{A}\in\mathbb{A}}\partial u_\mathscr{A}e^{i\frac{u_\mathscr{A}}{\lambda}}\textbf{F}^+_{\lambda\mathscr{A}}$. This term is quadratic of the form $\boldsymbol{\mathcal{E}}_\lambda^{evo} \textbf{F}^+_{\lambda\mathscr{A}}$ and does not possess any extra smallness that could be exploited through dispersive estimates. Nonetheless, it enjoys a favorable structure as it oscillates along characteristic phases. Thus, rather than absorbing this term into the RHS of the wave equation for $\boldsymbol{\mathcal{E}}_\lambda^{evo}$, we can incorporate it into the transport equation for
$\textbf{F}^+_{\lambda\mathscr{A}}$, as done for the first problematic term \ref{item:obst2Idea}.
 \begin{itemize}
     \item[--]\label{item:projecplace1Idea} On the RHS of $\Box\boldsymbol{\mathcal{E}}_\lambda^{evo}$, this term is lacking smallness. Since it is nonlinear, it requires a smallness of $O(\lambda^{1/2+\varepsilon})$ for some $\varepsilon>0$ to close the bootstrap, see \ref{item:item2improveass} of Section \ref{subsubsection:ass}. On the other hand, it is sufficiently regular as it involves no derivatives and is only quadratic.
     \item[--]\label{item:projecplace2Idea}  On the RHS of $\mathscr{L}_\mathscr{A} \textbf{F}^+_{\lambda\mathscr{A}}$, this term enjoys enough smallness, but it fails to have sufficient regularity. Indeed, the transport equation for $\Box\textbf{F}^+_{\lambda\mathscr{A}}$ is deduced by commuting the transport for $\textbf{F}^+_{\lambda\mathscr{A}}$ with $\Box$. This generates the source term $\partial\textbf{F}^+_{\lambda\mathscr{A}}\partial\boldsymbol{\mathcal{E}}_\lambda^{evo}$, which does not belong to $H^1$, where $\Box\textbf{F}^+_{\lambda\mathscr{A}}$ needs to be, since $\partial\textbf{F}^+_{\lambda\mathscr{A}}$ and $\partial\boldsymbol{\mathcal{E}}_\lambda^{evo}$ lie only in $H^1$.
 \end{itemize}
The idea is to exploit the best of both scenarios by decomposing $\boldsymbol{\mathcal{E}}_\lambda^{evo} \textbf{F}^+_{\lambda\mathscr{A}}$ using a projector on frequency.
We place the high-frequency component $\Pi_{\kappa,+}(\boldsymbol{\mathcal{E}}_\lambda^{evo} \textbf{F}^+_{\lambda\mathscr{A}})$ in the RHS of $\Box\boldsymbol{\mathcal{E}}_\lambda^{evo}$. For any $\kappa>0$, this yields an additional smallness of order $\lambda^{\kappa/2}$: in other words, we exchange regularity for smallness. Conversely, in the equation for $\textbf{F}_\mathscr{A}^+$ (and so in the equation for $
\Box\textbf{F}_\mathscr{A}^+$), we place the low frequency component $\Pi_{\kappa,-}(\boldsymbol{\mathcal{E}}_\lambda^{evo} \textbf{F}^+_{\lambda\mathscr{A}})$. This allows us to gain regularity. The cost of a derivative is arbitrarily small (since $\kappa$ may be chosen as small as desired) and can be covered by the natural smallness of $\boldsymbol{\mathcal{E}}_\lambda^{evo}$: we exchange smallness for regularity. In the equation for $\textbf{F}_\mathscr{A}^+$ (and so in the equation for $
\Box\textbf{F}_\mathscr{A}^+$), we place the low frequency part, that is, $\Pi_{\kappa,-}(\boldsymbol{\mathcal{E}}_\lambda^{evo} \textbf{F}^+_{\lambda\mathscr{A}})$. The value of $\kappa$ is not fixed at this stage. It only needs to be "small enough" to perform the bootstrap, see Sections \ref{subsubsection:estiFplus}, \ref{subsubsection:estiGplus}, \ref{subsubsection:estiEevo} and \ref{subsubsection:bootstrap}.
\subsubsection*{General remarks}
From the previous discussion, we have basically established
 \begin{align*}
&||\textbf{Z}_\lambda||_{H^{1/2}}\sim \lambda^{1/2},\\
&||\textbf{Z}_\lambda||_{H^1}\sim 1.
\end{align*}
Although the error term might appear insufficiently small in high-order Sobolev norms to handle the estimates of the RHS of \eqref{eq:schemeqIdea}, which contains cubic terms, its structure, that is, the fact that $\textbf{F}^+_{\lambda\mathscr{A}} e^{i\frac{u_{\mathscr{A}}}{\lambda}}$ and $\breve{\textbf{F}}^+_{\mathscr{A}\pm\mathscr{B}} e^{i\frac{u_{\mathscr{A}}\pm u_{\mathscr{B}}}{\lambda}}$ carry all the lack of smallness, compensates for this apparent defect.\\
It is noteworthy to point out that no smallness assumption is required on the initial ansatz (hence the background initial data) to close the bootstrap. The smallness of $\lambda$ is sufficient. This completes the discussion on constructing exact solutions to KGML \ref{item:item3Idea} and the bootstrap argument \ref{item:item4Idea}.
\subsubsection*{Error initial data and gauge}
The previous arguments rely, in fact, on the choice of \textbf{well-prepared error initial data} (see Definition \ref{defi:admissibleKGMLcriteria}) satisfying the \textbf{required smallness} condition of Definition \ref{defi:reqsmallnesscriteria}. By well-prepared, we mean morally with the right regularity and chosen so that the initial data for the auxiliary function
$\textbf{G}^+_{\lambda\mathscr{A}}$ are well-defined. Moreover, we want an exact solution to KGM. This requires the Lorenz gauge condition
\begin{equation}
\label{eq:lorenzgaugeIdea}
    \partial_\alpha A_\lambda^\alpha=0.
\end{equation}
To obtain this condition, the initial data must satisfy the \textbf{constraint equation for KGM in Lorenz gauge} of Definition \ref{defi:admissibleKGMgaugecriteria}. Once imposed, gauge propagation is straightforward, see Section \ref{subsubsection:propaggauegeneKGM}. 
Even under these constraints, for each $\lambda>0$ and each initial ansatz there remains some freedom in the choice of the error initial data, and hence in the resulting solutions. The constraints are underdetermined. We therefore construct a class of generic error initial data satisfying all three requirements above. Each element of the class leads to a one-parameter family of multiphase high-frequency solutions to KGM. 
The construction proceeds as follows
\begin{itemize}
\label{pointlistinitialdataIdea}
\item[--] We first construct \textbf{initial data} $\textbf{z}_\lambda$ for the \textbf{error term} $\textbf{Z}_\lambda$:
\begin{itemize}
\item[--] We choose the free part of the error initial data with compact support and generic smallness properties, imposing one additional smallness condition in Proposition \ref{propal:invertsystinitLap} to control the first time derivative of the gauge.
\item[--] We solve the constraint equations for Maxwell in Lorenz gauge with respect to the initial ansatz and the previously chosen free initial data. In particular, solving the Maxwell constraints boils down to inverting a perturbed Laplacian, which is done in some weighted Sobolev spaces introduced in Definition \ref{defi:weightedSobgenenot}. The RHS of the equation corresponds to that of the Maxwell evolution equation and thus reproduces the difficulties described above. Moreover, as explained in Remark \ref{rem:defectremApprox}, the gauge term generates a $O(\lambda^{-1/2})$ contribution. Nevertheless, the ellipticity of the constraint and its decoupled nature simplify the analysis.\\
\end{itemize}
\item[--] We decompose the \textbf{error term initial data} $\textbf{z}_\lambda$ into initial data for the \textbf{error components} (see the \textbf{precise error term} \eqref{eq:schemparaerrorIdea}):
\begin{itemize}
\item[--] We express $\boldsymbol{\mathcal{E}}_\lambda^{ell}|_{t=0}$ in terms of the initial background.
\item[--] We set $\textbf{F}^+_{\lambda\mathscr{A}}|_{t=0}=0$ and $\breve{\textbf{F}}^+_{\mathscr{A}\pm\mathscr{B}}|_{t=0}=0$ in order to ensure the smallness of $\boldsymbol{\mathcal{E}}_\lambda^{evo}$ and the smallness and the regularity of $\textbf{G}^+_{\lambda\mathscr{A}}$.
\item[--] We set $\boldsymbol{\mathcal{E}}_\lambda^{evo}|_{t=0}$ with respect to $\textbf{z}_\lambda$ and the previous choices so that the initial data for the precise error term solves the constraints.
\end{itemize}
\end{itemize}
\section{First-order approximation}
\label{section:Approx}
In this Section, we construct the first-order approximation based on the initial ansatz to have an almost approximate solution. This first-order approximation is composed of \textbf{background} terms, which \textbf{do not depend on $\lambda$}. \\
\begin{propal}
\label{propal:firstapproxApprox}
For a given admissible background initial data set in the sense of Definition \ref{defi:admissiblebginitansatz}:
\begin{enumerate}
    \item \label{item:approxit1Approx}
    There exists a \textbf{background} $(A^\alpha_0,\Phi_0,u_{\mathscr{A}},\Psi_\mathscr{A},W^\alpha_\mathscr{A})$ and a time $T>0$ such that it is solution to 
\begin{equation}
    \begin{cases}
    \label{eq:firstapproxsysApprox}
    \Box A^{\beta}_0=-\Im(\Phi_0\overline{\partial^{\beta}\Phi_0})+A^{\beta}_0|\Phi_0|^2+\sum_{{\mathscr{A}}}\partial^{\beta}u_{\mathscr{A}}|\Psi_{\mathscr{A}}|^2,\\
\Box \Phi_0+2iA^{\alpha}_0\partial_{\alpha}\Phi_0-A^{\alpha}_0A_{0\alpha}\Phi_0=0,\\
\mathscr{L}_{\mathscr{A}}\Psi_{\mathscr{A}}=-i2A^\alpha_0\partial_\alpha u_\mathscr{A}\Psi_\mathscr{A},\\
\partial^\beta u_{\mathscr{A}}\partial_{\beta} u_\mathscr{A}=0,\\
\mathscr{L}_{\mathscr{A}}W_{\mathscr{A}}^\beta=i\partial^\beta u_{\mathscr{A}}\overline{\Psi_\mathscr{A}}\Phi_0,\\
    \end{cases}
\end{equation}
with initial data
\begin{equation}
\label{eq:initdatabackgroundApprox}
\begin{split}
&(A^\alpha_0|_{t=0}=a_0,\partial_tA^\alpha_0|_{t=0}=\dot{a}_0,\Phi_0|_{t=0}=\phi_0,\partial_t\Phi_0|_{t=0}=\dot{\phi}_0), \\
&(u_{\mathscr{A}}|_{t=0}=v_{\mathscr{A}},\partial_tu_{\mathscr{A}}|_{t=0}=\dot{v}_{\mathscr{A}},\Psi_\mathscr{A}|_{t=0}=\psi_\mathscr{A},W^\alpha_\mathscr{A}|_{t=0}=w^\alpha_\mathscr{A}),
\end{split}
\end{equation}
 and with the regularity

\begin{align*}
\label{regularityApprox}
 &\sum_{j=0}^3(||\Phi_0||_{C^{j}([0,T],H^{3-j})}+||A^i_0||_{C^{j}([0,T],H^{3-j})}+ ||A^0_0||_{C^{j}([0,T],H^{3-j}_{\delta_0+j})})\leq c_0, \\
 &\max_{\mathscr{A}\in\mathbb{A}}(\sum_{j=0}^3(||\Psi_{\mathscr{A}}||_{C^{j}([0,T],H^{3-j})}+||W_{\mathscr{A}}|||_{C^{j}([0,T],H^{3-j})})\leq c_0,\\
 &\max_{\mathscr{A}\in\mathbb{A}}\sum_{j=0}^5||u_{\mathscr{A}}||_{C^{j}([0,T],H^{5-j}_{{\delta_1}+j})}\leq c_0,
\end{align*}
where  $c_0$ depends only on $c'_0$. Moreover, 
\begin{align*}
Supp(A^i_0,\Phi_0,\Psi_\mathscr{A},W^\alpha_\mathscr{A})\subset\mathscr{I}(T,B_{S'})\subset B_S,
\end{align*}
for $S=S'+T$. Only the phases $(u_{\mathscr{A}})_{\mathscr{A}\in\mathbb{A}}$ and $A^0_0$ are not compactly supported and need not be in $L^2$.\\
\item \label{item:approxit2Approx}
The background satisfies the \textbf{polarization}  
\begin{equation}
\label{eq:polarApprox}
\partial_{\beta}u_{\mathscr{A}}W^\beta_\mathscr{A}=0,
\end{equation}
and the \textbf{Lorenz gauge} condition 
\begin{equation}
\label{eq:lorenzgaugebgApprox}
\partial_\alpha A^\alpha_0=0.
\end{equation}
\item \label{item:approxit3Approx}
The first-order expansion 
 \begin{equation}    
 \label{eq:firstorderexpApprox}A_{1\lambda}^\alpha=A_0^\alpha+\lambda^{1/2}\sum_{\mathscr{A}\in\mathbb{A}}\left(P_\mathscr{A}^\alpha \cos\left(\frac{u_\mathscr{A}}{\lambda}\right)+Q_\mathscr{A}^\alpha \sin\left(\frac{u_\mathscr{A}}{\lambda}\right)\right), \;\;\;\;          \Phi_{1\lambda}=\Phi_0+\lambda^{1/2}\sum_{\mathscr{A}\in\mathbb{A}}\left(\Psi_\mathscr{A}e^{i\frac{u_\mathscr{A}}{\lambda}}\right)
 \end{equation}
defines an almost approximate solution of order one to KGM in Lorenz gauge, in the sense of Definition \ref{defi:almostapproxresults}.\\
\end{enumerate}
\end{propal}

\begin{rem}
\label{rem:structsysbgApprox}
    The system \eqref{eq:firstapproxsysApprox} has a triangular structure. The phase equations are decoupled. Then, the first three equations form a coupled subsystem. Taken together, these four equations are equivalent to KGML null-transport \eqref{eq:KGMnullgeneKGM}, presented in Section \ref{subsection:geneKGM}, with velocities given by the phase gradient $(\textbf{d}u^\#_\mathscr{A})_{\mathscr{A}\in\mathbb{A}}$ and the corresponding charge density $(|\Psi_\mathscr{A}|^2)_{\mathscr{A}\in\mathbb{A}}$. Indeed, the following transport equation holds
    \begin{align*}
        \mathscr{L}_{\mathscr{A}}|\Psi_{\mathscr{A}}|^2=\Psi_{\mathscr{A}}(\mathscr{L}_{\mathscr{A}}\overline{\Psi_{\mathscr{A}}})+\overline{\Psi}_{\mathscr{A}}(\mathscr{L}_{\mathscr{A}}\Psi_{\mathscr{A}})=|\Psi_{\mathscr{A}}|^2(2iA^\alpha_0\partial_\alpha u_\mathscr{A})-|\Psi_{\mathscr{A}}|^2(2iA^\alpha_0\partial_\alpha u_\mathscr{A})=0.
    \end{align*}
    The interaction terms produce a backreaction, an $O(1)$ non-oscillating term that we interpret as an effective charge flux. 
   Finally, the equation for $W^\beta_\mathscr{A}$ is decoupled.\\
\end{rem} 
\begin{rem}
\label{rem:nochargeApprox}
      In the case where $\Psi_{\mathscr{A}}|_{t=0}=0$ for all $\mathscr{A}\in\mathbb{A}$, then they remain identically zero (propagation of charge density) and $(A_0,\Phi_0)$ is solution of KGM in Lorenz gauge. \\
\end{rem}
\begin{rem}
\label{rem:polarnobackApprox}
 Due to the polarization condition, there is no backreaction in the Klein-Gordon equation.\\   
\end{rem}
\begin{nota}
\label{nota:cst0Approx}
      We denote with $C_0=C_{0}(c_0,T,N,\frac{1}{\eta_0})$ any constants that are a polynomial in $c_0$, $T$,  $N=|\mathbb{A}|$ and $\frac{1}{\eta_0}$, where $c_0$ and $T$ are from Proposition \ref{propal:firstapproxApprox}, $N$ is the number of phases from Definition \ref{defi:initialphasesetsetphase} and $\frac{1}{\eta_0}$ is from Proposition \ref{propal:smallnesscontrolsetphase} with the compact set given by $\Omega=B_S$ from Proposition \ref{propal:firstapproxApprox}.\\
\end{nota} 
\begin{nota}
\label{nota:b0background}
       We have the schematic notation $\textbf{B}_0'$ for all background terms $(A_0,\Phi_0,w_\mathscr{A},\psi_\mathscr{A},\textbf{d} u_{\mathscr{A}})$ and $\textbf{B}_0$ for all background terms excluding the phases, that is, $(A_0,\Phi_0,w_\mathscr{A},\psi_\mathscr{A})$.\\
\end{nota} 
\begin{proof}[Proof of point \ref{item:approxit1Approx} of Proposition \ref{propal:firstapproxApprox}]
Following Remark \ref{rem:structsysbgApprox}, we first determine the phases, relying on the geodesic equation and the method of characteristics, as explained in Proposition \ref{propal:initialphasegenephase}. In particular, the regularity is sufficient (see Definition \ref{defi:admissiblebginitansatz}) for Proposition \ref{propal:initialphasegenephase} and Sobolev regularity is also propagated. Once the phases are determined, classical local existence results based on energy estimates apply to the remaining wave/transport system, which here consists of the first three equations and the last equation of \eqref{eq:firstapproxsysApprox}. Regularity propagates on the interval of existence and the compact support property is preserved by finite speed of propagation.\\
\end{proof}
\begin{proof}[Proof of point \ref{item:approxit2Approx} of Proposition \ref{propal:firstapproxApprox}]
The initial data are admissible in the sense of Definition \ref{defi:admissiblebginitansatz}. This ensures that the polarization condition and the Lorenz gauge condition are satisfied at $t=0$. To recover the propagation of the polarization condition, one can easily show that $\mathscr{L}_{\mathscr{A}}\left(\partial_{\beta}u_{\mathscr{A}}W^\beta_\mathscr{A}\right)=0$. The propagation of the gauge follows from the argument of Section \ref{subsubsection:propaggauegeneKGM}.\\
\end{proof}
\begin{proof}[Proof of point \ref{item:approxit3Approx} of Proposition \ref{propal:firstapproxApprox}]
The polarization condition \eqref{eq:polarApprox} together with the Lorenz gauge \eqref{eq:lorenzgaugeapproxApprox} for $A_0$ imply that 
 \begin{align*}
        \partial_\alpha A_{1\lambda}^\alpha=\lambda^{1/2}\sum_{\mathscr{A}\in\mathbb{A}}\left(\partial_\alpha P_\mathscr{A}^\alpha \cos\left(\frac{u_\mathscr{A}}{\lambda}\right)+\partial_\alpha Q_\mathscr{A}^\alpha \sin\left(\frac{u_\mathscr{A}}{\lambda}\right)\right).  
 \end{align*}
We just proved that
 \begin{equation}
     \label{eq:lorenzgaugeapproxApprox}
\partial_\alpha A^\alpha_{1\lambda}=O(\lambda^{1/2}).
 \end{equation}

Substituting the truncated WKB expansions \eqref{eq:firstorderexpApprox} into KGML \eqref{eq:KGMLgeneKGM}, we get
\begin{equation}
\label{eq:plugKGMLMaxApprox}
\begin{split}
    \Box A_{1\lambda}^{\beta}+\Im(\Phi_{1\lambda}\overline{(\partial^{\beta}+iA^{\beta}_{1\lambda})\Phi_{1\lambda}})&=\underbrace{\Box A_{0}^{\beta}+\Im(\Phi_0\overline{\partial^{\beta}\Phi_0})-A_0^{\beta}|\Phi|^2-\sum_{\mathscr{A}\in\mathbb{A}}|\Psi_\mathscr{A}|^2\partial^{\beta}u_\mathscr{A}}_\text{$=0$}\\
    &-\lambda^{-1/2}\sum_{\mathscr{A}\in\mathbb{A}}\sin(\frac{u_\mathscr{A}}{\lambda})\underbrace{(\Box u_\mathscr{A}P_\mathscr{A}^\beta+2\partial^{\alpha}u_\mathscr{A}\partial_{\alpha}P_\mathscr{A}^{\beta}+\partial^{\beta}u_\mathscr{A}\Im(\Phi_0\overline{\Psi_\mathscr{A}}))}_\text{=0}\\
    &+\lambda^{-1/2}\sum_{\mathscr{A}\in\mathbb{A}}\cos(\frac{u_\mathscr{A}}{\lambda})\underbrace{(\Box u_\mathscr{A}Q_\mathscr{A}^\beta+2\partial^{\alpha}u_\mathscr{A}\partial_{\alpha}Q_\mathscr{A}^{\beta}-\partial^{\beta}u_\mathscr{A}\Re(\Phi_0\overline{\Psi_\mathscr{A}}))}_\text{=0}\\
    &-\lambda^{-3/2}\sum_{\mathscr{A}\in\mathbb{A}}\underbrace{\partial^\alpha u_\mathscr{A}\partial_\alpha u_\mathscr{A}}_\text{=0}\left(\cos\left(\frac{u_\mathscr{A}}{\lambda}\right)P^\beta_\mathscr{A}+\sin\left(\frac{u_\mathscr{A}}{\lambda}\right)Q^\beta_\mathscr{A}\right)\\
    &-\underbrace{\sum_{\mathscr{A},\mathscr{B}\in\mathbb{A},\mathscr{A}\neq\mathscr{B}}\partial^\beta u_{\mathscr{B}}\Im(e^{i\frac{u_\mathscr{A}}{\lambda}}\Psi_\mathscr{A}\overline{(ie^{i\frac{u_\mathscr{B}}{\lambda}}\Psi_\mathscr{B})})}_\text{$=\tilde{\Xi}_{A\lambda}$}+\lambda^{1/2}\Xi_{A\lambda}^\beta     
\end{split}
\end{equation}
and 
\begin{equation}
\label{eq:plugKGMLKleinApprox}
\begin{split}
    (\partial^{\alpha}+iA^{\alpha}_{1\lambda})(\partial_{\alpha}+i(A_{1\lambda})_\alpha )\Phi_{1\lambda}&=\underbrace{\Box \Phi_0+2iA_0^{\alpha}\partial_{\alpha}\Phi_0+A^{\alpha}_0A_{\alpha0}\Phi_0}_\text{$=0$}\\
    &+\sum_{\mathscr{A}\in\mathbb{A}}\lambda^{-1/2}ie^{i\frac{u_{\mathscr{A}}}{\lambda}}\underbrace{(\Box u_{\mathscr{A}}\Psi_{\mathscr{A}}+2\partial^{\alpha}u_{\mathscr{A}}\partial_{\alpha}\Psi_{\mathscr{A}}+2iA_0^{\alpha}\partial_{\alpha}u_{\mathscr{A}}\Psi_{\mathscr{A}})}_\text{=0}\\
     &-\lambda^{-3/2}\sum_{\mathscr{A}\in\mathbb{A}}\underbrace{\partial^\alpha u_\mathscr{A}\partial_\alpha u_\mathscr{A}}_\text{=0}e^{i\frac{u_\mathscr{A}}{\lambda}}\Psi_\mathscr{A}\\
&+\underbrace{\sum_{\mathscr{A},\mathscr{B}\in\mathbb{A},\mathscr{A}\neq\mathscr{B}}i2\left(\cos\left(\frac{u_{\mathscr{A}}}{\lambda}\right)P^\alpha_{\mathscr{A}}+\sin\left(\frac{u_{\mathscr{A}}}{\lambda}\right)Q^\alpha_{\mathscr{A}}\right)\partial_\alpha u_{\mathscr{B}}ie^{i\frac{u_{\mathscr{B}}}{\lambda}}\Psi_\mathscr{B}}_\text{$=\tilde{\Xi}_{\Phi\lambda}$}+\lambda^{1/2}\Xi_{\Phi\lambda}.\\
\end{split}
\end{equation}
The first two equations of \eqref{eq:firstapproxsysApprox} cancel the $O(1)$ non-oscillating terms, including the backreaction.\\ 
The third and fifth equations cancel the $O(\lambda^{-1/2})$ terms in the Klein-Gordon and Maxwell equations, respectively. \\
Then, the eikonal equation eliminates the $O(\lambda^{-3/2})$ terms
arising from the d'Alembertian. The remaining contributions are either of order $O(1)$, composed of the high-frequency interactions, or of order $O(\lambda^{1/2})$. The latter are precisely $\Xi_{\Phi\lambda}$ and $\Xi_{A\lambda}$ and we show in the next Proposition that the former correspond exactly to the $\tilde{\Xi}_{\Phi\lambda}$ and $\tilde{\Xi}_{A\lambda}$ as introduced in Definition \ref{defi:almostapproxresults}. \\
\end{proof}
\begin{nota}
\label{nota:xixi}
We use the  schematic notation $\tilde{\boldsymbol{\Xi}}_\lambda$ for both $\tilde{\Xi}_{A\lambda}$ and $\tilde{\Xi}_{\Phi\lambda}$, and $\boldsymbol{\Xi}_\lambda$ for both $\Xi_{A\lambda}$ and $\Xi_{\Phi\lambda}$.
\end{nota}
\begin{propal}
\label{propal:KandSApprox}  
       The $O(1)$ high-frequency interaction terms $\tilde{\Xi}_{A\lambda}$ of \eqref{eq:plugKGMLMaxApprox} and $\tilde{\Xi}_{\Phi\lambda}$ of \eqref{eq:plugKGMLKleinApprox} can be written and separated as follows:
\begin{equation}
\label{eq:XitildeMApprox}  
\tilde{\Xi}_{A\lambda}^\beta=\sum_{(\mathscr{A},\mathscr{B})\in\mathscr{C}}\Re\left(\overline{K^\beta_{A(\mathscr{A}\pm\mathscr{B})}(\textbf{B}_0')}e^{i\frac{u_\mathscr{A}\pm u_\mathscr{B}}{\lambda}}\right)+\sum_{(\mathscr{A},\mathscr{B})\in\mathscr{S}}\Re\left(\overline{S^\beta_{A(\mathscr{A}\pm\mathscr{B})}(\textbf{B}_0')}e^{i\frac{u_\mathscr{A}\pm u_\mathscr{B}}{\lambda}}\right),\\
\end{equation}
for the Maxwell equation and 
\begin{equation}
\label{eq:XitildeMKGMApprox}  
\tilde{\Xi}_{\Phi\lambda}=\sum_{(\mathscr{A},\mathscr{B})\in\mathscr{C}}e^{i\frac{u_\mathscr{A}\pm u_\mathscr{B}}{\lambda}}K_{\phi(\mathscr{A}\pm\mathscr{B})}(\textbf{B}_0' )+\sum_{(\mathscr{A},\mathscr{B})\in\mathscr{S}}e^{i\frac{u_\mathscr{A}\pm u_\mathscr{B}}{\lambda}}S_{\Phi(\mathscr{A}\pm\mathscr{B})}(\textbf{B}_0'),
\end{equation}
for the Klein-Gordon equation, with $\textbf{B}_0'$ from Notation \ref{nota:b0background}. In particular, for $(\mathscr{A},\mathscr{B})\in\mathscr{S}$,
\begin{align*}
 &S^\beta_{A(\mathscr{A}+\mathscr{B})}=0, &S^\beta_{A(\mathscr{A}-\mathscr{B})}=-\partial^\beta u_{\mathscr{B}}\overline{\Psi_\mathscr{A}}\Psi_\mathscr{B},\\
 &S_{\Phi(\mathscr{A}+\mathscr{B})}=\Psi_\mathscr{A}\partial_\alpha u_\mathscr{A}(P^\alpha_\mathscr{B}-iQ^\alpha_\mathscr{B}), &S_{\Phi(\mathscr{A}-\mathscr{B})}=\Psi_\mathscr{A}\partial_\alpha u_\mathscr{A}(P^\alpha_\mathscr{B}+iQ^\alpha_\mathscr{B}),
 \end{align*}
 and for $(\mathscr{A},\mathscr{B})\in\mathscr{C}$
 \begin{align*}
 &K^\beta_{A(\mathscr{A}+\mathscr{B})}=0, &K^\beta_{A(\mathscr{A}-\mathscr{B})}=-\partial^\beta u_{\mathscr{B}}\overline{\Psi_\mathscr{A}}\Psi_\mathscr{B},\\
 &K_{\Phi(\mathscr{A}+\mathscr{B})}=\Psi_\mathscr{A}\partial_\alpha u_\mathscr{A}(P^\alpha_\mathscr{B}-iQ^\alpha_\mathscr{B}), &K_{\Phi(\mathscr{A}-\mathscr{B})}=\Psi_\mathscr{A}\partial_\alpha u_\mathscr{A}(P^\alpha_\mathscr{B}+iQ^\alpha_\mathscr{B}).
 \end{align*}
\end{propal}
 \begin{nota}
 \label{nota:schemanotApprox}
    The term $\tilde{\boldsymbol{\Xi}}_\lambda$, representing schematically $\tilde{\Xi}_{A\lambda}$ and $\tilde{\Xi}_{\Phi\lambda}$, has the form
 \begin{align*}
\tilde{\boldsymbol{\Xi}}_\lambda=\sum_{(\mathscr{A},\mathscr{B})\in\mathscr{C}}e^{i\frac{u_\mathscr{B}\pm u_\mathscr{B}}{\lambda}}\boldsymbol{K}_{(\mathscr{A}\pm\mathscr{B})}(\textbf{B}_0' )+\sum_{(\mathscr{A},\mathscr{B})\in\mathscr{S}}e^{i\frac{u_\mathscr{B}\pm u_\mathscr{B}}{\lambda}}\textbf{S}_{(\mathscr{A}\pm\mathscr{B})}(\textbf{B}_0').
\end{align*}
 \end{nota}
 \begin{rem}
  \label{rem:KandSform}
 The $K$-terms correspond to contributions arising from resonant interactions,
whereas the $S$-terms to those emerging from non-resonant interactions.
 \end{rem}
\begin{proof}[Proof of Proposition \ref{propal:KandSApprox}]
The result follow from direct calculation using trigonometric formulas. We thus recover for $\tilde{\Xi}_{A\lambda}$ and $\tilde{\Xi}_{\Phi\lambda}$ the structure of Definition \ref{defi:almostapproxresults}, i.e., high-frequency terms whose phases are either always characteristic or never characteristic and which enjoy the appropriate regularity. This concludes the proof of Propositions \ref{propal:KandSApprox} and \ref{propal:firstapproxApprox}.
\end{proof}
We now state the properties of the interaction terms.
\begin{propal}
\label{propal:KandSpropertiesApprox}
For $\textbf{K}_{\mathscr{A}\pm\mathscr{B}}$ and $\textbf{S}_{\mathscr{A}\pm\mathscr{B}}$ defined as in Proposition \ref{propal:KandSApprox} and for $C_0$ given by Notation \ref{nota:cst0Approx}: 
\begin{enumerate}
       \item \label{item:itm1KandSpropertiesApprox} 
        The following estimates hold
       \begin{align*}
            &\forall t\in[0,T],\;\sum_{k\leq 3}||\textbf{K}_{\mathscr{A}\pm\mathscr{B}}(t)||_{H^{k}}dt\leq C_0,\\
             &\forall t\in[0,T],\;\sum_{k\leq 3}||\textbf{S}_{\mathscr{A}\pm\mathscr{B}}(t)||_{H^{k}}dt\leq C_0.
       \end{align*}
\item \label{item:itm2KandSpropertiesApprox} 
For $(\mathscr{A},\mathscr{B})\in\mathscr{C}$, the vector $K^\beta_{A(\mathscr{A}\pm\mathscr{B})}$ is orthogonal to $\partial^\beta u_{\mathscr{A}}\pm\partial^\beta u_{\mathscr{B}}$.
   \end{enumerate}

\end{propal}
\begin{proof}[Proof of point \ref{item:itm1KandSpropertiesApprox} of Proposition \ref{propal:KandSpropertiesApprox}]
This follows from direct estimates, as $H^3(\mathbb{R}^3)$ is an algebra, and all the terms involved
belong to $H^3(\mathbb{R}^3)$ uniformly on $[0,T]$. 
\end{proof}
\begin{proof}[Proof of point \ref{item:itm2KandSpropertiesApprox} of Proposition \ref{propal:KandSpropertiesApprox}]
By Proposition \ref{propal:KandSApprox}, one has $K^\beta_{A(\mathscr{A}-\mathscr{B})}=-\partial^\beta u_{\mathscr{B}}\overline{\Psi_\mathscr{A}}\Psi_\mathscr{B}$ and $K^\beta_{A(\mathscr{A}+\mathscr{B})}=0$. Consequently, for $(\mathscr{A},\mathscr{B})\in\mathscr{C}$, one has $(\partial^\alpha u_\mathscr{A}-\partial^\alpha u_\mathscr{B})\partial_\alpha u_\mathscr{A}=(\partial^\alpha u_\mathscr{A}-\partial^\alpha u_\mathscr{B})\partial_\alpha u_\mathscr{B}=0$, so that $K^\beta_{A(\mathscr{A}-\mathscr{B})}(\partial_\beta u_{\mathscr{A}}-\partial_\beta u_{\mathscr{B}})=0$. Then, the identity $K^\beta_{A(\mathscr{A}+\mathscr{B})}(\partial_\beta u_{\mathscr{A}}+\partial_\beta u_{\mathscr{B}})=0$ holds trivially.
\end{proof}

\section{Initial data for exact solutions}
\label{section:Init}
Here we provide a detailed description of the requirements imposed on the initial data for the error term in order to prove Theorems \ref{unTheorem:mainth1results} and \ref{unTheorem:mainth2results}.\\\\
More precisely, we aim to construct initial data for the precise error term \eqref{eq:schemparaerrorIdea}. As explained in Section \ref{subsection:Idea}, the evolution of the error term is governed by a system of coupled equations \eqref{eq:schemfullsystemerrterm} for distinct components. In order to obtain the precise structure required to close the argument, one must prescribe initial data for each \textbf{error component} under suitable \textbf{admissibility conditions} given in Section \ref{subsection:criteria}.\\\\
In Sections \ref{subsectionLap} and \ref{subsection:consinit}, we construct for any $\lambda>0$ a \textbf{generic set} of initial data fulfilling these requirements, see Proposition \ref{propal:specsetconsinit}. The construction proceeds by first solving the constraint equations for the KGM system in Lorenz gauge \eqref{eq:initconstrainterrorInit} in a small-data regime, and then separating the resulting solutions into the initial data associated with the different error components, taking care of the corresponding required properties.
Since the number of variables exceeds the number of constraint equations, the construction leaves a certain freedom of choice. In what follows, we exploit this freedom as little as possible.\\\\
Before that, we enunciate the definition of the constraint equations and the admissibility criteria.
\subsection{Equations of constraint for KGM in Lorenz gauge}
\label{subsection:Initconstr}
The \textbf{initial data set for Klein-Gordon-Maxwell} consists of
$(a^\alpha_\lambda,\dot{a}^\alpha_\lambda,\phi_\lambda,\dot{\phi}_\lambda)$. \\
In the high-frequency regime, we expand these quantities as the \textbf{initial parametrices}
\begin{equation} 
\label{eq:initialfulldefInit}
\begin{split}
&a^\alpha_\lambda=a^\alpha_{1\lambda}+z^\alpha_\lambda,
\;\;\;\;\;\;\dot{a}^\alpha_\lambda=\dot{a}^\alpha_{1\lambda}+\dot{z}_\lambda^\alpha,\\
&\phi_\lambda=\phi_{1\lambda}+\zeta_\lambda,
\;\;\;\;\;\;\dot{\phi}_\lambda=\dot{\phi}_{1\lambda}+\dot{\zeta}_\lambda,
\end{split}
\end{equation}
where $(z_\lambda,\zeta_\lambda,\dot{z}_\lambda,\dot{\zeta}_\lambda)$ is the
\textbf{error initial data} associated with the \textbf{error term} $(Z_\lambda,\mathcal{Z}_\lambda)$. 
\begin{defi}
\label{defi:initialbgdefInit}
For a given admissible background initial data set $(a^\alpha_0,\dot{a}^\alpha_0,\phi_0,\dot{\phi}_0,v_{\mathscr{A}},\dot{v}_\mathscr{A},\psi_\mathscr{A},w^\alpha_\mathscr{A})$, we define the initial ansatz $(a^\alpha_{1\lambda},\phi_{1\lambda})$ directly from the background data as in Definition \ref{defi:initansatzinitansatz}.\\
Then, we also define $(\dot{a}^\alpha_{1\lambda},\dot{\phi}_{1\lambda})$ as $((\partial_tA^\alpha_{1\lambda})|_{t=0},(\partial_t\Phi_{1\lambda})|_{t=0})$ where $(A^\alpha_{1\lambda},\Phi_{1\lambda})$
are almost approximate solutions to KGM given using Proposition \ref{propal:firstapproxApprox}. The time derivatives $((\partial_tA^\alpha_{1\lambda})|_{t=0},(\partial_t\Phi_{1\lambda})|_{t=0})$ can be expressed only in terms of the background initial data and their space derivatives.\\
\end{defi}    
 \begin{rem}
 \label{defi:constraintfullInit}
   Recall that the constraints for $(a^\alpha_\lambda,\dot{a}^\alpha_\lambda,\phi_\lambda,\dot{\phi}_\lambda)$ associated with KGM in Lorenz gauge  are 
\begin{equation}
\begin{cases}
 \label{eq:initconstraintfullInit}
    \dot{a}^0_\lambda=-\partial_ia^i_\lambda , \\
        -\Delta a_\lambda^0-\partial_i(\dot{a}_\lambda^i)=-\Im(\phi_\lambda\overline{\dot{\phi_\lambda}})-a^0_\lambda|\phi_\lambda|^2.
\end{cases}
    \end{equation}   
 \end{rem}
\begin{rem}
\label{rem:remainfreedomInit}
Once $(a^\alpha_{1\lambda},\phi_{1\lambda},\dot{a}^\alpha_{1\lambda},\dot{\phi}_{1\lambda})$ is fixed according to Definition \ref{defi:initialbgdefInit}, it remains to determine the \textbf{error initial data} $(z_\lambda,\zeta_\lambda,\dot{z}_\lambda,\dot{\zeta}_\lambda)$. This corresponds to $10$ unknowns (as $z_\lambda$ and $\dot{z}_\lambda$ are four-vectors), while only two constraint equations are imposed. Hence, $8$ degrees of freedom remain.
 \end{rem}
The following Proposition specifies the constraint equations and the regularity requirement for the \textbf{error initial data}. Their resolution is postponed to Section \ref{subsectionLap}.
\begin{propal}
 \label{propal:initconstrainterrorInit}
For a given admissible background initial data set $(a^\alpha_0,\dot{a}^\alpha_0,\phi_0,\dot{\phi}_0,v_{\mathscr{A}},\dot{v}_\mathscr{A},\psi_\mathscr{A},w^\alpha_\mathscr{A})$ from Definition \ref{defi:admissiblebginitansatz}, let $(a^\alpha_{1\lambda},\dot{a}^\alpha_{1\lambda},\phi_{1\lambda},\dot{\phi}_{1\lambda})$ be given as in Definition \ref{defi:initialbgdefInit} for some $\lambda>0$ and let the error initial data set $(z_\lambda^0,z_\lambda^i,\zeta_\lambda,\dot{z}_\lambda^0,\dot{z}_\lambda^i,\dot{\zeta}_\lambda)\in H^2_\delta\times H^2\times H^2\times H^1_{\delta}\times H^1\times H^1$, for some $-3/2<\delta<-1/2$ with $\delta\leq\delta_0$, satisfy
    \begin{equation}   
    \begin{cases} 
     \label{eq:initconstrainterrorInit}
   \dot{z}_\lambda^0+\partial_jz^j_\lambda=-\lambda^{1/2}\sum_{\mathscr{A}\in\mathbb{A}}\Re\left(e^{i\frac{u_\mathscr{A}}{\lambda}}(\overline{\partial_\alpha W^\alpha_\mathscr{A}})|_{t=0}\right),\\\\ 
-\Delta z_\lambda^0+|\phi_{1\lambda}+\zeta_\lambda|^2z^0_\lambda-\partial_j\dot{z}_\lambda^j=
\lambda^{1/2}\sum_{\mathscr{A}\in\mathbb{A}}\left(\partial_t\Re\left(e^{i\frac{u_\mathscr{A}}{\lambda}}\overline{\partial_\alpha W^\alpha_\mathscr{A}}\right)\right)|_{t=0}+\lambda^{1/2}\xi_{A\lambda}^0(\textbf{B}_0')+\tilde{\xi}_{A\lambda}^0(\textbf{B}_0'),\\\\
-\Im(\phi_{1\lambda}\overline{\dot{\zeta}_\lambda})-\Im(\zeta_\lambda\overline{\dot{\phi}_{1\lambda}})-\Im(\zeta_\lambda\overline{\dot{\zeta}_\lambda})-|\zeta_\lambda|^2a_{1\lambda}^0-(\zeta_\lambda\overline{\phi_{1\lambda}}+\overline{\zeta_\lambda}\phi_{1\lambda})a_{1\lambda}^0.
\end{cases}
\end{equation}
Here, $\lambda^{1/2}\xi_{A\lambda}^0(\textbf{B}_0')=\lambda^{1/2}\Xi_{A\lambda}^0(\textbf{B}_0')|_{t=0}$ and $\tilde{\xi}_{A\lambda}^0(\textbf{B}_0')=\tilde{\Xi}_{A\lambda}^0(\textbf{B}_0')|_{t=0}$ are remainders terms defined in Proposition \ref{propal:KandSApprox}, with Notation \ref{nota:b0background} for the background terms.
Then, with $(a^\alpha_\lambda,\dot{a}^\alpha_\lambda,\phi_\lambda,\dot{\phi}_\lambda)$ defined by \eqref{eq:initialfulldefInit}, the constraints \eqref{eq:initconstraintfullInit} are satisfied with sufficient regularity for our purposes.
\end{propal}
\begin{proof}
    We first consider the Lorenz gauge condition. We require  
\begin{align*}  
&\dot{a}^0_\lambda+\partial_ia^i_\lambda=\dot{a}^0_{1\lambda}+\partial_ia^i_{1\lambda}+\dot{z}_\lambda^0+\partial_iz^i_\lambda=0,
\intertext{and we know that}
&\dot{a}^0_{1\lambda}+\partial_ia^i_{1\lambda}=\lambda^{1/2}\sum_{\mathscr{A}\in\mathbb{A}}\Re\left(e^{i\frac{u_\mathscr{A}}{\lambda}}(\overline{\partial_\alpha W^\alpha_\mathscr{A}}\right)|_{t=0}),
\end{align*}
from Definition \ref{defi:constraintbginitansatz}. This yields the first equation in
\eqref{eq:initconstrainterrorInit}.\\\\
For the second constraint, we compute
\begin{align*}
&-\Delta\left(\lambda^{1/2}\Re\left(e^{i\frac{u_\mathscr{A}}{\lambda}}\overline{W_\mathscr{A}^{0}}\right)\right)-\partial^2_{jt}\left(\lambda^{1/2} \Re\left(e^{i\frac{u_\mathscr{A}}{\lambda}}\overline{W_\mathscr{A}^{j}}\right)\right)\\
&=-\lambda^{-1/2}\Re\left(ie^{i\frac{u_\mathscr{A}}{\lambda}}(\Delta u_\mathscr{A}\overline{W_\mathscr{A}^{0}}+2\partial_ju_\mathscr{A}\partial_j\overline{W_\mathscr{A}^{0}}+\partial_t(\partial_ju_\mathscr{A}\overline{W_\mathscr{A}^{j}})+\partial_tu_\mathscr{A}\partial_j\overline{W_\mathscr{A}^{j}})\right)\\
&-\lambda^{1/2}\Re\left(e^{i\frac{u_\mathscr{A}}{\lambda}}(\Delta W_\mathscr{A}^{0}+\partial^2_{tj}\overline{W_\mathscr{A}^{j}})\right)+\lambda^{-3/2}\Re\left(e^{i\frac{u_\mathscr{A}}{\lambda}}(\partial_ju_\mathscr{A}\partial_ju_\mathscr{A}\overline{W_\mathscr{A}^0}+\partial_ju_\mathscr{A}\partial_tu_\mathscr{A}\overline{W_\mathscr{A}^j})\right)\\
&=-\lambda^{-1/2}\Re\left(ie^{i\frac{u_\mathscr{A}}{\lambda}}(\Delta u_\mathscr{A}\overline{W_\mathscr{A}^{0}}-\partial^2_{tt}u_\mathscr{A}\overline{W_\mathscr{A}^{0}}-\partial_{t}u_\mathscr{A}\partial_{t}\overline{W_\mathscr{A}^{0}}+2\partial_ju_\mathscr{A}\partial_j\overline{W_\mathscr{A}^{0}}-\partial_tu_\mathscr{A}\partial_t\overline{W_\mathscr{A}^{0}})\right)\\
&-\lambda^{1/2}\Re\left(e^{i\frac{u_\mathscr{A}}{\lambda}}(\Delta \overline{W_\mathscr{A}^0}-\partial^2_{tt}\overline{W_\mathscr{A}^{0}})\right)-\partial_t\left(\Re\left(e^{i\frac{u_\mathscr{A}}{\lambda}}\lambda^{1/2}\partial_\alpha \overline{W_\mathscr{A}^\alpha}\right)\right)\\
&=-\Box\left(\lambda^{1/2}\Re\left(e^{i\frac{u_\mathscr{A}}{\lambda}}\overline{W_\mathscr{A}^{0}}\right)\right)-\partial_t\left(\Re\left(e^{i\frac{u_\mathscr{A}}{\lambda}}\lambda^{1/2}\partial_\alpha \overline{W_\mathscr{A}^\alpha}\right)\right),
\end{align*}
where we use the polarization condition $\partial_\alpha u_\mathscr{A}W_\mathscr{A}^\alpha=0$. Since the background initial data are admissible,
they satisfy the constraint equation \eqref{eq:constraintbginitansatz}, which implies that  
\begin{equation}
        -\Delta a_{1\lambda}^0-\partial_i(\dot{a}_{1\lambda}^i)+\Im(\phi_{1\lambda}\overline{\dot{\phi_{1\lambda}}})+a^0_{1\lambda}|\phi_{1\lambda}|^2=-\lambda^{1/2}\sum_{\mathscr{A}\in\mathbb{A}}\partial_t\left(\Re\left(e^{i\frac{u_\mathscr{A}}{\lambda}}\overline{\partial_\alpha W_\mathscr{A}^\alpha}\right)\right)|_{t=0}+\lambda^{1/2}\xi_{A\lambda}^0(\textbf{B}_0')+\tilde{\xi}_{A\lambda}^0(\textbf{B}_0').
\end{equation}   
All remaining terms are obtained by direct calculations.\\
\end{proof}
\subsection{Definition of the criteria for admissible error initial data}
\label{subsection:criteria}
As explained in Section \ref{subsection:Idea}, the error term is decomposed into several components to recover the necessary structure for the bootstrap argument of Section \ref{subsubsection:evosys}. Further details are provided in Section \ref{subsection:errterm}. The \textbf{error components} are given by $((E^{evo})^{\alpha}_\lambda,\mathcal{E}^{evo}_\lambda,W^{+\alpha}_{\lambda\mathscr{A}},\breve{W}^{+\alpha}_{\mathscr{A}\pm\mathscr{B}},\Psi^+_{\lambda\mathscr{A}},\breve{\Psi}^+_{\mathscr{A}\pm\mathscr{B}},G^{+\alpha}_{W^+_{\lambda\mathscr{A}}},G^+_{\Psi^+_{\lambda\mathscr{A}}},(E^{ell}_\lambda)^\alpha,\mathcal{E}^{ell}_\lambda)$ and their corresponding \textbf{error components initial data} by
 \begin{align*}
((e^{evo}_\lambda)^{\alpha},\epsilon^{evo}_\lambda,(\dot{e}^{evo}_\lambda)^{\alpha},\dot{\epsilon}^{evo}_\lambda,w^{+\alpha}_{\lambda\mathscr{A}},\breve{w}^{+\alpha}_{\mathscr{A}\pm\mathscr{B}},\psi^+_{\lambda\mathscr{A}},\breve{\psi}^+_{\mathscr{A}\pm\mathscr{B}},g^{+\alpha}_{W^+_{\lambda\mathscr{A}}},g^+_{\Psi^+_{\lambda\mathscr{A}}},(e^{ell}_\lambda)^{\alpha},\epsilon^{ell}_\lambda,(\dot{e}^{ell}_\lambda)^{\alpha},\dot{\epsilon}^{ell}_\lambda).   
 \end{align*}
 In our construction, the quantities $((e^{evo}_\lambda)^{\alpha},\epsilon^{evo}_\lambda,(\dot{e}_\lambda^{evo})^{\alpha},\dot{\epsilon}^{evo},w^{+\alpha}_{\lambda\mathscr{A}},\psi^+_{\lambda\mathscr{A}},g^{+\alpha}_{W^+_{\lambda\mathscr{A}}},g^+_{\Psi^+_{\lambda\mathscr{A}}})$ serve as initial data for the error system \eqref{eq:schemfullsystemerrterm}, the component $(\breve{w}^{+\alpha}_{\mathscr{A}\pm\mathscr{B}},\breve{\psi}^+_{\mathscr{A}\pm\mathscr{B}})$ corresponds to initial data for the \textbf{decoupled equations} \eqref{eq:schemtpfabinfofab} and $((e^{ell}_\lambda)^{\alpha},\epsilon^{ell}_\lambda,(\dot{e}_\lambda^{ell})^{\alpha},\dot{\epsilon}^{ell}_\lambda)$ are prescribed \textbf{directly by the background}. Details on these equations are given in Section \ref{subsection:errterm}.\\ 
The next definition explains how the \textbf{error initial data} in the sense of \eqref{eq:initialfulldefInit} are reconstructed from \textbf{error components initial data}.

\begin{defi}
\label{defi:glueerrorcomponentcriteria}
For a given \textbf{admissible background initial data} and for \textbf{error components initial data}, we define $(a^\alpha_\lambda,\dot{a}^\alpha_\lambda,\phi_\lambda,\dot{\phi}_\lambda)$, the initial data for KGM, by    
    \begin{align*} 
&a^\alpha_\lambda=a^\alpha_{1\lambda}+z^\alpha_\lambda,
&\dot{a}^\alpha_\lambda=\dot{a}^\alpha_{1\lambda}+\dot{z}_\lambda^\alpha,\\
&\phi_\lambda=\phi_{1\lambda}+\zeta_\lambda,
&\dot{\phi}_\lambda=\dot{\phi}_{1\lambda}+\dot{\zeta}_\lambda,\\
\end{align*} where $(a^\alpha_{1\lambda},\dot{a}^\alpha_{1\lambda},\phi_{1\lambda},\dot{\phi}_{1\lambda})$ is expressed only in terms of the \textbf{background initial data} as in Definition \ref{defi:initialbgdefInit}, while the \textbf{precise error term initial data} $(z^\alpha_\lambda,\dot{z}_\lambda^\alpha,\zeta_\lambda,\dot{\zeta}_\lambda)$ is expressed in terms of the \textbf{error components initial data} as
\begin{equation}
\label{eq:glueerrorcomponentbigerrorcriteria}
\begin{cases}
\!\begin{aligned}
&z^\alpha_\lambda=\sum_{\mathscr{A}\in\mathbb{A}}\lambda^{1}\Re\left(e^{i\frac{v_\mathscr{A}}{\lambda}}\overline{w^{+\alpha}_{\lambda\mathscr{A}}}\right)+\sum_{(\mathscr{A},\mathscr{B})\in\mathscr{C}}\lambda^{1}\Re\left(e^{i\frac{v_{\mathscr{A}}\pm v_{\mathscr{B}}}{\lambda}}\overline{\breve{w}^{+\alpha}_{\mathscr{A}\pm\mathscr{B}}}\right)+(e^{evo}_\lambda)^{\alpha}+(e^{ell}_\lambda)^{\alpha}, \\
&\zeta_\lambda=\sum_{\mathscr{A}\in\mathbb{A}}\lambda^{1}\psi^+_{\lambda\mathscr{A}}e^{i\frac{v_\mathscr{A}}{\lambda}}+\sum_{(\mathscr{A},\mathscr{B})\in\mathscr{C}}\lambda^{1}\breve{\psi}^+_{\mathscr{A}\pm\mathscr{B}}e^{i\frac{v_{\mathscr{A}}\pm v_{\mathscr{B}}}{\lambda}}+\epsilon^{evo}_\lambda+\epsilon^{ell}_\lambda,\\
&\dot{z}^\alpha_A=\sum_{\mathscr{A}\in\mathbb{A}}\lambda^1\Re\left(e^{i\frac{v_\mathscr{A}}{\lambda}}(\overline{\partial_tW^{+\alpha}_{\lambda\mathscr{A}}})|_{t=0}\right)+\sum_{\mathscr{A}\in\mathbb{A}}\dot{v}_\mathscr{A}\Re\left(ie^{i\frac{v_\mathscr{A}}{\lambda}}\overline{w^{+\alpha}_{\lambda\mathscr{A}}}\right)+\sum_{(\mathscr{A},\mathscr{B})\in\mathscr{C}}\lambda^1 \Re\left(e^{i\frac{v_{\mathscr{A}}\pm v_{\mathscr{B}}}{\lambda}}(\overline{\partial_t \breve{W}_{\mathscr{A}\pm\mathscr{B}}^{+\alpha}})|_{t=0}\right)\\
&+\sum_{(\mathscr{A},\mathscr{B})\in\mathscr{C}}(\dot{v}_{\mathscr{A}}\pm \dot{v}_{\mathscr{B}}) \Re\left(e^{i\frac{v_{\mathscr{A}}\pm v_{\mathscr{B}}}{\lambda}}\overline{ \breve{w}_{\mathscr{A}\pm\mathscr{B}}^{+\alpha}}\right)+(\dot{e}_\lambda^{evo})^{\alpha}+(\dot{e}_\lambda^{ell})^{\alpha},\\
 & \dot{\zeta}_\lambda=\sum_{\mathscr{A}\in\mathbb{A}}\lambda^1  (\partial_t\Psi^+_{\lambda\mathscr{A}})|_{t=0}e^{i\frac{v_\mathscr{A}}{\lambda}}+\sum_{\mathscr{A}\in\mathbb{A}}  \dot{v}_\mathscr{A}\psi^+_{\lambda\mathscr{A}}ie^{i\frac{v_\mathscr{A}}{\lambda}}+\sum_{(\mathscr{A},\mathscr{B})\in\mathscr{C}}\lambda^1\partial_t(\breve{\Psi}^+_{\mathscr{A}\pm\mathscr{B}})|_{t=0}e^{i\frac{v_{\mathscr{A}}\pm v_{\mathscr{B}}}{\lambda}}\\
&+\sum_{(\mathscr{A},\mathscr{B})\in\mathscr{C}}(\dot{v}_{\mathscr{A}}\pm \dot{v}_{\mathscr{B}})\breve{\psi}^+_{\mathscr{A}\pm\mathscr{B}}ie^{i\frac{v_{\mathscr{A}}\pm v_{\mathscr{B}}}{\lambda}}+\dot{\epsilon}^{evo}_\lambda+\dot{\epsilon}^{ell}_\lambda.      
\end{aligned}
\end{cases}
\end{equation} 
\end{defi}
\begin{rem}
\label{rem:noGpluscriteria}
    We observe that the quantities $(g^{+\alpha}_{W^+_{\lambda\mathscr{A}}},g^+_{\Psi^+_{\lambda\mathscr{A}}})$ do not appear explicitly in these parametrices. Indeed, they correspond to the initial data for \textbf{auxiliary functions} introduced to handle regularity loss in view of local well-posedness, see Section \ref{subsubsection:estiGplus}. 
\end{rem}
\begin{rem}
\label{rem:timederivcriteria}
    The quantities $(\partial_tW^{+\alpha}_{\lambda\mathscr{A}})_{t=0}$ and $(\partial_t\Psi^{+}_{\lambda\mathscr{A}})_{t=0}$ are computed directly using the evolution equations \eqref{eq:nonschemtpFpluserrterm} and depend on $((e^{evo}_\lambda)^{\alpha},\epsilon^{evo}_\lambda,w^{+\alpha}_{\lambda\mathscr{A}},\psi^+_{\lambda\mathscr{A}},\nabla w^{+\alpha}_{\lambda\mathscr{A}},\nabla\psi^+_{\lambda\mathscr{A}})$ (a part of the \textbf{initial data} for the \textbf{error components}) and the background. The same remark applies to $(\partial_t\breve{\Psi}^+_{\mathscr{A}\pm\mathscr{B}})|_{t=0}$ and $(\partial_t \breve{W}_{\mathscr{A}\pm\mathscr{B}}^{+\alpha})|_{t=0}$.
\end{rem}
We now introduce the three criteria required for the initial data in order to prove Theorems \ref{unTheorem:mainth1results} and \ref{unTheorem:mainth2results}. First, a specific definition is given for the elliptic component $((e^{ell}_\lambda)^{\alpha},\epsilon^{ell}_\lambda,(\dot{e}^{ell}_\lambda)^{\alpha},\dot{\epsilon}^{ell}_\lambda)$.\\
\begin{defi}
 \label{defi:defiEellinitcriteria}
     We define $((e^{ell}_\lambda)^{\alpha},\epsilon^{ell}_\lambda,(\dot{e}^{ell}_\lambda)^{\alpha},\dot{\epsilon}^{ell}_\lambda)$ as
\begin{align*}
&(e^{ell}_\lambda)^{\alpha}=((E^{ell}_\lambda)^\alpha(t))|_{t=0},  &\epsilon^{ell}_\lambda=(\mathcal{E}^{ell}_\lambda(t))||_{t=0},\\
&(\dot{e}^{ell}_\lambda)^{\alpha}=\partial_t(E^{ell}_\lambda)^\alpha(t))|_{t=0}, &\dot{\epsilon}^{ell}_\lambda=(\partial_t\mathcal{E}^{ell}_\lambda(t))|_{t=0}, 
\end{align*}
where $((E^{ell}_\lambda)^\alpha,\mathcal{E}^{ell}_\lambda)$ are given by Definition \ref{defi:defiEellinfoEell}. The time derivatives are computed with the transport equations of Proposition \ref{propal:firstapproxApprox}. To be more precise, starting from the initial ansatz we construct a background solution to \eqref{eq:firstapproxsysApprox} and evaluate the time derivatives of the corresponding background quantities at time $t=0$. These derivatives depend only on the \textbf{background initial data} and their spatial derivatives.\\
\end{defi}
\begin{propal}
For $((e^{ell}_\lambda)^{\alpha},\epsilon^{ell}_\lambda,(\dot{e}^{ell}_\lambda)^{\alpha},\dot{\epsilon}^{ell}_\lambda)$ given by Definition \ref{defi:defiEellinitcriteria} and for $C_0$ given by Notation \ref{nota:cst0Approx}, the following regularity and smallness estimates hold:
\begin{equation}
 \label{eq:estimatesEellinitcriteria}
\sum_{k\leq 2}\lambda^k(||e^{ell}_\lambda||_{H^{k}}+||\epsilon^{evo}_\lambda||_{H^{k}})+\sum_{k\leq 1}\lambda^{k+1}(||\dot{e}^{ell}_\lambda||_{H^{k}}+||\dot{\epsilon}^{ell}_\lambda||_{H^{k}})\leq \lambda^{2} C_0.\\
\end{equation}
\end{propal}
\begin{proof}
The estimates follow directly from the results of Section \ref{subsubsection:infoEell}.
\end{proof}
\begin{defi}
\label{defi:admissibleKGMLcriteria}
     We say that the \textbf{error components initial data} are \textbf{admissible for KGML} with respect to an \textbf{admissible background initial data} (in the sense of Definition \ref{defi:admissiblebginitansatz}) if, for some $-3/2<\delta<-1/2$ with $\delta\leq\delta_0$, we have
 \begin{align*}
     &g^{+\alpha}_{W^+_{\lambda\mathscr{A}}}=(\Box W^{+\alpha}_{\lambda\mathscr{A}})|_{t=0}, 
     & g^+_{\Psi^+_{\lambda\mathscr{A}}}=(\Box \Psi^{+}_{\lambda\mathscr{A}})|_{t=0},
 \end{align*}
and 
 \begin{align*}
&\forall\mathscr{A}\in\mathbb{A},\;
w^{+\alpha}_{\lambda\mathscr{A}}\in H^2,
&\forall(\mathscr{A},\mathscr{B})\in\mathscr{C},\;\breve{w}^{+\alpha}_{\mathscr{A}\pm\mathscr{B}}\in H^3,\\
&\forall\mathscr{A}\in\mathbb{A},\;
\psi^+_{\lambda\mathscr{A}}\in H^2,
&\forall(\mathscr{A},\mathscr{B})\in\mathscr{C},\;\breve{\psi}^+_{\mathscr{A}\pm\mathscr{B}}\in H^3,\\
&\forall\mathscr{A}\in\mathbb{A},\;g^{+\alpha}_{W^+_{\lambda\mathscr{A}}}\in H^1,
     &\forall\mathscr{A}\in\mathbb{A},\; g^+_{\Psi^+_{\lambda\mathscr{A}}}\in H^1,\\
     &(e^{evo}_\lambda)^i\in H^2,
&(\dot{e}^{evo}_\lambda)^i\in H^1, \\
 &(e^{evo}_\lambda)^0\in H^2_\delta, &(\dot{e}^{evo}_\lambda)^0\in H^1_{\delta+1},\\
 &\epsilon^{evo}_\lambda\in H^2,
  &\dot{\epsilon}^{evo}_\lambda\in H^1,
    \end{align*}
    and if $((e^{ell}_\lambda)^{\alpha},\epsilon^{ell}_\lambda,(\dot{e}^{ell}_\lambda)^{\alpha},\dot{\epsilon}^{ell}_\lambda)$ are given by Definition \ref{defi:defiEellinitcriteria}. In addition, we require that the functions \\$(g^\alpha_{W^+_{\lambda\mathscr{A}}},g^\alpha_{\Psi^+_{\lambda\mathscr{A}}},w^{+\alpha}_{\lambda\mathscr{A}},\psi^{+}_{\lambda\mathscr{A}},\breve{w}^{+\alpha}_{\mathscr{A}\pm\mathscr{B}},\breve{\psi}^{+}_{\mathscr{A}\pm\mathscr{B}},(e^{evo}_\lambda)^i,
(\dot{e}^{evo}_\lambda)^i,\epsilon^{evo}_\lambda,
  \dot{\epsilon}^{evo}_\lambda)$ are compactly supported\footnote{We choose the same $S'$ as for the background in Proposition \ref{propal:firstapproxApprox} without loss of generality.} in $B_{S'}$.
 \end{defi}
 \begin{rem}
 \label{rem:timetimederivcriteria}
 The second derivatives, $(\partial^2_{tt}W^{+\alpha}_{\lambda\mathscr{A}})_{t=0}$ and $(\partial^2_{tt}\Psi^{+}_{\lambda\mathscr{A}})_{t=0}$, are obtained by differentiating the transport equations once more. They are expressed with respect to $((e^{evo}_\lambda)^\alpha,\epsilon^{evo}_\lambda,(\dot{e}^{evo}_\lambda)^\alpha,\dot{\epsilon}^{evo}_\lambda,w^{+\alpha}_{\lambda\mathscr{A}},\psi^+_{\lambda\mathscr{A}},\nabla w^{+\alpha}_{\lambda\mathscr{A}},\nabla\psi^+_{\lambda\mathscr{A}})$, together with $((\partial_tW^{+\alpha}_{\lambda\mathscr{A}})_{t=0},(\partial_t\Psi^{+}_{\lambda\mathscr{A}})_{t=0},\nabla(\partial_tW^{+\alpha}_{\lambda\mathscr{A}})_{t=0},\nabla(\partial_t\Psi^{+}_{\lambda\mathscr{A}})_{t=0})$ and the background. Taking into account Remark \ref{rem:timederivcriteria}, it follows that the quantities $(\Box \Psi^{+}_{\lambda\mathscr{A}})|_{t=0}$ and $(\Box W^{+\alpha}_{\lambda\mathscr{A}})|_{t=0}$ can be expressed solely in terms of the \textbf{error components initial data} and the background.
\end{rem} 
 \begin{defi}
 \label{defi:admissibleKGMgaugecriteria}
     We say that the \textbf{error components initial data set} is \textbf{admissible for KGM in Lorenz gauge} with respect to a \textbf{background initial data} if the corresponding quantities $(z^\alpha_\lambda,\dot{z}^\alpha_\lambda,\zeta_\lambda,\dot{\zeta}_\lambda)$ defined by \eqref{eq:glueerrorcomponentbigerrorcriteria} solve the constraint equations \eqref{eq:initconstrainterrorInit}. 
 \end{defi}
\begin{defi}
 \label{defi:reqsmallnesscriteria}
    Let $C_0$ be defined as in Notation \ref{nota:cst0Approx}. The \textbf{required smallness} conditions for the coupled components are   
    \begin{align*}
 &\sum_{k\leq 1} \lambda^k(||\epsilon^{evo}_\lambda||_{H^{k+1}}+||\dot{\epsilon}^{evo}_\lambda||_{H^{k}}+||(e^{evo}_\lambda)^i||_{H^{k+1}}+||(\dot{e}^{evo}_\lambda)^i||_{H^{k}}+||(e^{evo}_\lambda)^0||_{H_\delta^{k+1}}+||(\dot{e}^{evo}_\lambda)^0||_{H^{k}_{\delta+1}})\leq \lambda^{1/2}C_0, \\  
  &\max_{\mathscr{A}\in\mathbb{A}}(\sum_{k\leq 1} \lambda^k(||w^{+\alpha}_{\lambda\mathscr{A}}||_{H^{k+1}}+||g^{+\alpha}_{W^{+}_{\lambda\mathscr{A}}}||_{H^{k}}+||\psi^{+}_{\lambda\mathscr{A}}||_{H^{k+1}}+||g^+_{\Psi^{+}_{\lambda\mathscr{A}}}||_{H^{k}}))\leq C_0,  \\
   \end{align*}
   and for the decoupled components\footnote{The quantities $(\breve {w}^{+\alpha}_{\mathscr{A}\pm\mathscr{B}},\breve {\phi}^{+}_{\mathscr{A}\pm\mathscr{B}})$ are treated as background terms in the evolution scheme, since the evolution equation  \eqref{eq:schemtpfabinfofab} is decoupled. See Section \ref{subsubsection:infofab}.}
   \begin{align*}
&\max_{(\mathscr{A},\mathscr{B})\in\mathscr{C}}(||\breve {w}^{+\alpha}_{\mathscr{A}\pm\mathscr{B}}||_{H^{3}}+||\breve {\phi}^{+}_{\mathscr{A}\pm\mathscr{B}}||_{H^{3}})\leq C_0. 
\end{align*}
\end{defi}
\begin{rem}
  \label{rem:admissiblefullcriteria}  
    Initial data satisfying the three criteria of Definitions \ref{defi:admissibleKGMLcriteria}, \ref{defi:admissibleKGMgaugecriteria} and \ref{defi:reqsmallnesscriteria} are said to be admissible. They provide sufficient conditions to construct exact solutions, close the bootstrap argument of Section~\ref{subsubsection:bootstrap}, and prove Theorems~\ref{unTheorem:mainth1results} and~\ref{unTheorem:mainth2results}.
\end{rem}
The next two Sections are devoted to the construction of a \textbf{generic set} of \textbf{error components initial data} satisfying these criteria. 
\subsection{Inversion of the perturbed Laplacian and smallness for the error term}
\label{subsectionLap}
This Section provides the basis for the construction of the \textbf{generic set} of admissible initial data (see Proposition \ref{propal:specsetconsinit}). In particular, it ensures that the \textbf{precise error initial data} \eqref{eq:glueerrorcomponentbigerrorcriteria} satisfy the constraint system \eqref{eq:initconstrainterrorInit}. These results will be used crucially in the next Section. In addition, since we aim to construct solutions whose error term is small with respect to the parameter~$\lambda$, we must impose suitable smallness assumptions on the corresponding initial data. Thus, in the Proposition below, we assume that the free component of the \textbf{error initial data} $(z_\lambda^i,\zeta_\lambda,\dot{z}_\lambda^i,\dot{\zeta}_\lambda)$ satisfy appropriate smallness conditions, and we determine the remaining component $(z_\lambda^0,\dot{z}_\lambda^0)$ through the equations. Moreover, an additional smallness assumption $\partial_i\dot{z}_\lambda^i$ is required to recover the smallness of $z^0_\lambda$. This is possible thanks to the freedom in the choice of initial data, although it makes the resulting precise error initial data slightly less generic.\\\\
 During this entire Section, we denote with $C_0$ any constant that only depends on the background data. See Notation \ref{nota:cst0Approx}.
\begin{propal}
\label{propal:invertsystinitLap}
Let $(z^i_\lambda,\zeta_\lambda,\dot{z}_\lambda^i,\dot{\zeta}_\lambda)\in H^2\times H^2\times H^1\times H^1$ be compactly supported\footnote{We pick the same $S'$ as for the background in Proposition \ref{propal:firstapproxApprox} without loss of generality.} in $B_{S'}$ and satisfy the following smallness conditions
\begin{equation}
\label{eq:smallinitfreeerrorLap}
 \sum_{k\leq 1} \lambda^k||\zeta_\lambda||_{H^{k+1}}+\lambda^k||\dot{\zeta}_\lambda||_{H^{k}}+\lambda^k||z^i_\lambda||_{H^{k+1}}+\lambda^k||\dot{z}_\lambda^i||_{H^{k}}\leq \lambda^{1/2}C_0,   
\end{equation}
together with
\begin{align}
\label{eq:extrasmallLap}
    ||\partial_i\dot{z}_\lambda^i||_{L^2}\leq \lambda^{1/2}C_0.
\end{align}
Then, for any $-3/2<\delta<-1/2$ there exists a solution $(z^0_\lambda,\dot{z}_\lambda^0)\in H_\delta^{k+1}\times H_{\delta+1}^{k}$ to \eqref{eq:initconstrainterrorInit} such that 
\begin{align}
\label{eq:smallinitconstrerrorLap}
 \sum_{k\leq 1} \lambda^k||z^0_\lambda||_{H_\delta^{k+1}}+\lambda^k||\dot{z}_\lambda^0||_{H_{\delta+1}^{k}}\leq \lambda^{1/2}C_0,   
\end{align}
where the weighted Sobolev norms are introduced in Definition \ref{defi:weightedSobgenenot}.
\end{propal}
\begin{proof}
   For $\dot{z}_\lambda^0$, we use the first equation of \eqref{eq:initconstrainterrorInit}. The desired estimate follows directly from the assumed smallness of the free data.
\begin{rem}
\label{rem:polarkillLap}
     The polarization condition \eqref{eq:initpolarinitansatz} is essential here, as it cancels potentially harmful terms of size $O(\lambda^{-1/2})$ in the first constraint equation.   
 \end{rem}
 For $z^0_\lambda$, we use at the second equation of \eqref{eq:initconstrainterrorInit}, involving the operator $-\Delta+|\phi_{1\lambda}+\zeta_\lambda|^2$.\\
If one were to invert this perturbed Laplacian directly\footnote{The perturbed Laplacian operator is only invertible in some weighted Sobolev spaces.} with Theorem \ref{unTheorem:laplcianAx}, one would generally obtain an estimate of the form\footnote{The constant in Lemma \ref{lem:laplcianpluscompactAx} only depends on $L^2$ norms of $C_0$ size initial data.} $||z^0_\lambda||_{H^2_\delta}\sim \lambda^{-1/2}$ due to the presence of several problematic source terms. This includes the derivative of the divergence of $W_\mathscr{A}$ (see Remark \ref{rem:defectremApprox}), $O(1)$ high-frequency interactions and the term  $\Im(\zeta_\lambda\overline{\dot{\phi}_{1\lambda}})$. The latter contains contributions of the form $-\sum_{\mathscr{A}\in\mathbb{A}}\Im(i\zeta_\lambda\lambda^{-1/2}\dot{v}_\mathscr{A}\overline{\psi}_{\mathscr{A}}e^{-i\frac{v_\mathscr{A}}{\lambda}})$, which corresponds precisely to the obstruction described in \ref{item:obst2Idea} of Section \ref{subsection:Idea}. This difficulty is addressed at the level of the evolution equations by introducing the unknown $\textbf{F}^+_{\lambda\mathscr{A}}$.\\
At the level of the constraints, these obstructions are all treated by explicitly inverting by hand the $\lambda^{-1/2}$ and $\lambda^0$ terms. 
We set 
\begin{align*}
\label{zbisLap}
z^{0BIS}_\lambda&=\lambda^{3/2}\sum_{\mathscr{A}\in\mathbb{A}}[\frac{\dot{v}_\mathscr{A}\Re\left(ie^{i\frac{v_\mathscr{A}}{\lambda}}(\overline{\partial_\alpha W_\mathscr{A}^\alpha})|_{t=0}\right)}{\partial^iv_\mathscr{A}\partial_iv_\mathscr{A}}]-\lambda^{3/2}\sum_{\mathscr{A}\in\mathbb{A}}[\frac{\Im\left(i\zeta_\lambda\dot{v}_\mathscr{A}\overline{\psi}_{\mathscr{A}}e^{-i\frac{v_\mathscr{A}}{\lambda}}\right)}{\partial^iv_\mathscr{A}\partial_iv_\mathscr{A}}],\\   &+\lambda^{2}\sum_{(\mathscr{A},\mathscr{B})\in\mathscr{C}}\frac{\Re\left(\overline{K^\beta_{A(\mathscr{A}\pm\mathscr{B})}(\textbf{B}_0')}e^{i\frac{v_\mathscr{A}\pm v_\mathscr{B}}{\lambda}}\right)}{\partial^i(v_\mathscr{A}\pm v_\mathscr{B})\partial_i(v_\mathscr{A}\pm v_\mathscr{B})}+\lambda^2\sum_{(\mathscr{A},\mathscr{B})\in\mathscr{S}}\frac{\Re\left(\overline{S^\beta_{A(\mathscr{A}\pm\mathscr{B})}(\textbf{B}_0')}e^{i\frac{u_\mathscr{A}\pm u_\mathscr{B}}{\lambda}}\right)}{\partial^i(v_\mathscr{A}\pm v_\mathscr{B})\partial_i(v_\mathscr{A}\pm v_\mathscr{B})},
\end{align*}
where the $K^\beta_{A(\mathscr{A}\pm\mathscr{B})}(\textbf{B}_0')$'s and $S^\beta_{A(\mathscr{A}\pm\mathscr{B})}(\textbf{B}_0')$'s terms are defined in Proposition \ref{propal:KandSApprox} (here with respect to the background initial data) and arise from the term $\tilde{\xi}^0_{A\lambda}(\textbf{B}_0')$ of the second equation of \eqref{eq:initconstrainterrorInit}. In particular, $K^\beta_{A(\mathscr{A}+\mathscr{B})}(\textbf{B}_0')=S^\beta_{A(\mathscr{A}+\mathscr{B})}(\textbf{B}_0')=0$.
 By construction, $z^{0BIS}_\lambda$ possesses the sufficient smallness and is compactly supported. More precisely, using Proposition \ref{propal:smallnesscontrolsetphase} and estimates \eqref{eq:eta5}, \eqref{eq:eta3}, \eqref{eq:eta4} to control the denominators, we obtain
 \begin{align}
 \sum_{k\leq 1} \lambda^k||z^{0BIS}_\lambda||_{H^{k+1}}\leq \lambda^{1/2}C_0.
\end{align}
 Then, we search for $z^{0TER}_\lambda$ as a solution to 
 \begin{equation}
 \label{eq:ztereqLap}
 \begin{split}
     (-\Delta+|\phi_{1\lambda}+\zeta_\lambda|^2) z_\lambda^{0TER}-\partial_i\dot{z}_\lambda^i&=
(\Delta-|\phi_{1\lambda}+\zeta_\lambda|^2) z_\lambda^{0BIS}+\lambda^{1/2}\xi_{A\lambda}^0(\textbf{B}_0)+\tilde{\xi}_{A\lambda}^0(\textbf{B}_0)\\
&-\Im(\phi_{1\lambda}\overline{\dot{\zeta}_\lambda})-\Im(\zeta_\lambda\overline{\dot{\phi}_{1\lambda}})-\Im(\zeta_\lambda\overline{\dot{\zeta}_\lambda})-|\zeta_\lambda|^2a_{1\lambda}^0-(\zeta_\lambda\overline{\phi_{1\lambda}}+\overline{\zeta_\lambda}\phi_{1\lambda})a_{1\lambda}^0\\
&+\lambda^{1/2}\sum_{\mathscr{A}\in\mathbb{A}}\left(\partial_t\Re\left(e^{i\frac{u_\mathscr{A}}{\lambda}}(\overline{\partial_\alpha W_\mathscr{A}^\alpha})\right)\right)|_{t=0}.\\
 \end{split}
 \end{equation}
The RHS of \eqref{eq:ztereqLap} is now of order $O(\lambda^{1/2})$ and is well defined in terms of regularity. In particular, the Laplacian of $\zeta_\lambda$ appears and belongs to $L^2$ because $\zeta_\lambda\in H^2$. The remaining terms only involve background terms which are more regular. Moreover, the quantity $|\phi_{1\lambda}+\zeta_\lambda|^2$ belongs to $H^2$ and is compactly supported. This allows us to use the Lemma \ref{lem:laplcianpluscompactAx} (as $-3/2<\delta<-1/2$) to get 
\begin{align*}
\label{estimzterLap}
    ||z_\lambda^{0TER}||_{H^2_\delta}\leq \lambda^{1/2}C_0.
\end{align*}
In particular, the term $\Im(\zeta_\lambda\overline{\dot{\zeta}\lambda})$, which is problematic in the evolution equations (see item \ref{item:obst4Idea} in Section \ref{subsection:Idea}), is harmless at the level of the constraints due to the decoupled nature of the constraints. Indeed, 
\begin{align*}
||\Im(\zeta_\lambda\overline{\dot{\zeta}_\lambda})||_{L^2}\leq||\zeta_\lambda||_{H^1}||\dot{\zeta}_\lambda||_{H^{1/2}}\leq\lambda^{1/2} C_0,
\end{align*}
where we used the product estimate \eqref{eq:holdsobAx} and the interpolation inequality \eqref{eq:interpolsobAx}. All other terms are easily controlled in $L^2$, using in particular the additional smallness \eqref{eq:extrasmallLap} for $\partial_i\dot{z}^i_\lambda$. Finally, defining $z^{0}_\lambda=z^{0BIS}_\lambda+z^{0TER}_\lambda$, we obain 
\begin{align*}
    \sum_{k\leq 1} \lambda^k||z^0_\lambda||_{H_\delta^{k+1}}\leq \lambda^{1/2}C_0,
\end{align*} 
which ends the proof.
\end{proof}
 \subsection{Construction of the precise error initial data}
 \label{subsection:consinit}
 In the previous Section, we constructed small \textbf{error initial data} $(z_\lambda,\zeta_\lambda,\dot{z}_\lambda,\dot{\zeta}_\lambda)$ that solve the constraints \eqref{eq:initconstrainterrorInit}. However, prescribing these initial data directly as initial conditions for the error evolution system (see \eqref{eq:schemeqerrorerrterm}) does \textbf{not} provide the appropriate structure to close the bootstrap. Therefore, the purpose of this Section is to construct a generic set of \textbf{initial data} for the \textbf{error components} that are \textbf{admissible} in the sense of Remark \ref{rem:admissiblefullcriteria} and that are based on the error initial data obtained in Proposition \ref{propal:invertsystinitLap}.
\begin{propal}
\label{propal:specsetconsinit}
    For a given admissible background initial data set $(a^\alpha_0,\dot{a}^\alpha_0,\phi_0,\dot{\phi}_0,v_{\mathscr{A}},\dot{v}_\mathscr{A},\psi_\mathscr{A},w^\alpha_\mathscr{A})$ and for $\lambda>0$, Let $(z_\lambda^0,z_\lambda^i,\zeta_\lambda,\dot{z}_\lambda^0,\dot{z}_\lambda^i,\dot{\zeta}_\lambda)$ be given by Proposition \ref{propal:invertsystinitLap} and $(e^{ell}_\lambda,\epsilon^{ell}_\lambda,\dot{e}^{ell}_\lambda,\dot{\epsilon}^{ell}_\lambda)$ by Definition \ref{defi:defiEellinitcriteria}. Then, for 
    \begin{align*}  
&(e^{evo}_\lambda)^\alpha=z_\lambda^\alpha-(e^{ell}_\lambda)^\alpha,
&\epsilon^{evo}_\lambda=\zeta_\lambda-\epsilon^{ell}_\lambda,\\
&\forall\mathscr{A}\in\mathbb{A},\;
w^{+\alpha}_{\lambda\mathscr{A}}=0,
&\forall(\mathscr{A},\mathscr{B})\in\mathscr{C},\;\breve{w}^{+\alpha}_{\mathscr{A}\pm\mathscr{B}}=0,\\
&\forall\mathscr{A}\in\mathbb{A},\;
\psi^+_{\lambda\mathscr{A}}=0,
&\forall(\mathscr{A},\mathscr{B})\in\mathscr{C},\;\breve{\psi}^+_{\mathscr{A}\pm\mathscr{B}}=0,\\
&\forall\mathscr{A}\in\mathbb{A},\;g^{+\alpha}_{W^+_{\lambda\mathscr{A}}}=(\Box W^{+\alpha}_{\lambda\mathscr{A}})|_{t=0}, 
     &\forall\mathscr{A}\in\mathbb{A},\; g^+_{\Psi^+_{\lambda\mathscr{A}}}=(\Box \Psi^{+}_{\lambda\mathscr{A}})|_{t=0},\\
    \end{align*}
\begin{align*}
&(\dot{e}^{evo}_\lambda)^{\alpha}=\dot{z}^\alpha_\lambda-\lambda\sum_{\mathscr{A}\in\mathbb{A}} \Re\left(e^{\frac{v_\mathscr{A}}{\lambda}}(\overline{\partial_tW_{\lambda\mathscr{A}}^{+\alpha}})|_{t=0}\right)-\sum_{(\mathscr{A},\mathscr{B})\in\mathscr{C}}\lambda \Re\left(e^{\frac{v_{\mathscr{A}}\pm v_{\mathscr{B}}}{\lambda}}(\overline{\partial_t \breve{W}_{\mathscr{A}\pm\mathscr{B}}^{+\alpha}})|_{t=0}\right) -(\dot{e}^{ell}_\lambda)^{\alpha},\\ 
  &\dot{\epsilon}^{evo}_\lambda=\dot{\zeta}_\lambda-\lambda\sum_{\mathscr{A}\in\mathbb{A}}  (\partial_t\Psi^+_{\lambda\mathscr{A}})|_{t=0}e^{i\frac{v_\mathscr{A}}{\lambda}}-\sum_{(\mathscr{A},\mathscr{B})\in\mathscr{C}}\lambda\partial_t(\breve{\Psi}^+_{\mathscr{A}\pm\mathscr{B}})|_{t=0}e^{i\frac{v_{\mathscr{A}}\pm v_{\mathscr{B}}}{\lambda}}-\dot{\epsilon}^{ell}_\lambda, 
\end{align*}
the initial data for the \textbf{error components }are \textbf{admissible for KGML}, \textbf{admissible for KGM in Lorenz gauge} and satisfy the \textbf{required smallness}.
\end{propal}
\begin{rem}
\label{rem:specsetremconsinit}
This initial data set is generic in the sense that no additional freedom is used on $(z_\lambda^0,z_\lambda^i,\zeta_\lambda,\dot{z}_\lambda^0,\dot{z}_\lambda^i,\dot{\zeta}_\lambda)$, aside from the convenient smallness and regularity conditions, together with the technical smallness of $\partial_i\dot{z}_\lambda^i$. As already observed from estimates \eqref{eq:smallinitfreeerrorLap} and
\eqref{eq:smallinitconstrerrorLap}, the error initial data $(z_\lambda^0,z_\lambda^i,\zeta_\lambda,\dot{z}_\lambda^0,\dot{z}_\lambda^i,\dot{\zeta}_\lambda)$ enjoy better
smallness properties than the full error terms $(Z_\lambda,\mathcal{Z}_\lambda)$ (see point \ref{item:item2th2results} of Theorem \ref{unTheorem:mainth2results}) in terms of smallness. This improvement is due to the choice $(w^{+\alpha}_{\lambda\mathscr{A}},
\psi^+_{\lambda\mathscr{A}},\breve{w}^{+\alpha}_{\mathscr{A}\pm\mathscr{B}},\breve{\psi}^{+}_{\mathscr{A}\pm\mathscr{B}})=0$, which yields additional smallness at time $t=0$ and ensures sufficient regularity to $(g^{+\alpha}_{W^+_{\lambda\mathscr{A}}},g^+_{\Psi^+_{\lambda\mathscr{A}}})$. At a later time, the extra smallness of the error initial data does not propagate, as there is no reason that $(W^{+\alpha}_{\lambda\mathscr{A}},
\Psi^+_{\lambda\mathscr{A}},\breve{W}^{+\alpha}_{\mathscr{A}\pm\mathscr{B}},\breve{\Psi}^{+}_{\mathscr{A}\pm\mathscr{B}})=0$, whereas the regularity of $(g^{+\alpha}_{W^+_{\lambda\mathscr{A}}},g^+_{\Psi^+_{\lambda\mathscr{A}}})$ does. \\
Overall, the freedom of choice is used only on the auxiliary components, while the $10$ real initial data are only
subject to $2$ constraints, retaining their $8$ real degrees of freedom.
\end{rem}

 \begin{proof}[Proof of Proposition \ref{propal:specsetconsinit}]
 
For the KGML admissibility, it suffices to verify regularity and smallness. The estimates for $(z_\lambda,\zeta_\lambda,\dot{z}_\lambda,\dot{\zeta}_\lambda)$
and $(e^{ell}_\lambda,\epsilon^{ell}_\lambda,\dot{e}^{ell}_\lambda,\dot{\epsilon}^{ell}_\lambda)$ follow directly from Proposition \ref{propal:invertsystinitLap} together with \eqref{eq:smallinitfreeerrorLap}, \eqref{eq:smallinitconstrerrorLap}, and \eqref{eq:estimatesEellinitcriteria}. The only delicate part concerns the time derivatives of $W^{+\alpha}_{\lambda\mathscr{A}},\breve{W}^{+\alpha}_{\mathscr{A}\pm\mathscr{B}},\Psi^+_{\lambda\mathscr{A}}$ and $\breve{\Psi}^+_{\mathscr{A}\pm\mathscr{B}}$. In particular, for $g^{+\alpha}_{W^+_{\lambda\mathscr{A}}}$ and $ g^+_{\Psi^+_{\lambda\mathscr{A}}}$, Remark \ref{rem:timetimederivcriteria} and the fact that $w^{+\alpha}_{\lambda\mathscr{A}}=0$ and $\psi^+_{\lambda\mathscr{A}}=0$ by choice imply that, although $\Box$ terms appear to involve two derivatives, it is not the case. Indeed, the two spatial derivatives vanish, and the two time derivatives only depend on quantities that involve one derivative with sufficient regularity. We have
\begin{align*}
    &(\partial_t W^{+\alpha}_{\lambda\mathscr{A}})|_{t=0}=\lambda^{-1/2}f_w^\alpha(\textbf{B}'_0)\epsilon^{evo}_\lambda, \\
    &(\partial_t \Psi^{+}_{\lambda\mathscr{A}})|_{t=0}=\lambda^{-1/2}f_{\psi\alpha}(\textbf{B}'_0)(e^{evo}_\lambda)^{\alpha},\\
    &(\partial^2_{tt} W^{+\alpha}_{\lambda\mathscr{A}})|_{t=0}=L^\alpha_w(\nabla(\partial_t W^{+\alpha}_{\lambda\mathscr{A}})|_{t=0},\lambda^{-1/2}\dot{\epsilon}^{evo}_\lambda,\textbf{B}'_0)+Q^\alpha_w(\epsilon^{evo}_\lambda,e^{evo}_\lambda)+[\ldots],\\
    &(\partial^2_{tt} \Psi^{+}_{\lambda\mathscr{A}})|_{t=0}=L^\alpha_\psi(\nabla(\partial_t \Phi^{+\alpha}_{\mathscr{A}})|_{t=0},\lambda^{-1/2}\dot{e}^{evo}_\lambda,\textbf{B}'_0)+Q^\alpha_\psi(e^{evo}_\lambda,e^{evo}_\lambda)+[\ldots],
\end{align*}
where the $[\ldots]$ involve lower-order more regular terms. The $f_w$ and $f_\psi$ terms involve no derivatives of $\textbf{B}'_0$, $L_w$ and $L_\psi$ depends linearly on their first two parameters, and $Q_w$ and $Q_\psi$ contain a product with respect to their two parameters. All these contributions are controlled using \eqref{eq:eta1} and the estimates of Proposition \ref{propal:invertsystinitLap} and Definition \ref{defi:admissiblebginitansatz}. To estimate $\dot{e}^{evo}_\lambda$ and $\dot{\epsilon}^{evo}_\lambda$, we use
\begin{align*}
    &\sum_{k\leq 1}\lambda^k||\lambda e^{i\frac{v_\mathscr{A}}{\lambda}}(\partial_t W^{+}_{\lambda\mathscr{A}})|_{t=0}||_{H^k}\leq \lambda C_0, &\sum_{k\leq 1}\lambda^k||\lambda e^{i\frac{v_\mathscr{A}}{\lambda}}(\partial_t \Phi^{+}_{\lambda\mathscr{A}})|_{t=0}||_{H^k}\leq \lambda C_0, 
\end{align*}
where $C_0$ is as in Notation \ref{nota:cst0Approx}, other direct estimates. For $g^{+\alpha}_{W^+_{\lambda\mathscr{A}}}$ and $g^+_{\Psi^+_{\lambda\mathscr{A}}}$, we use $C_0$ 
\begin{align*}
    &\sum_{k\leq 1}\lambda^k||(\partial^2_{tt} W^{+}_{\lambda\mathscr{A}})|_{t=0}||_{H^k}\leq C_0, &\sum_{k\leq 1}\lambda^k||(\partial^2_{tt} \Psi^{+}_{\lambda\mathscr{A}})|_{t=0}||_{H^k}\leq  C_0,
\end{align*}
so that we have the required smallness and the required regularity of Definition \ref{defi:reqsmallnesscriteria}. In particular, the quantities $g^{+\alpha}_{W^+_{\lambda\mathscr{A}}}$ and $g^+_{\Phi^+_{\lambda\mathscr{A}}}$ are more regular than two standard derivatives of $W^+_{\lambda\mathscr{A}}$ and $\Psi^+_{\lambda\mathscr{A}}$ are expected to be.\\\\
Finally, the KGM in Lorenz gauge admissibility follows because our \textbf{precise error initial data} are constructed on the \textbf{error initial data} from Definition \ref{defi:glueerrorcomponentcriteria}, which directly implies that $(a^\alpha_\lambda,\dot{a}^\alpha_\lambda,\phi_\lambda,\dot{\phi}_\lambda)$ is solution to the constraint equations \eqref{eq:initconstraintfullInit}.
 \end{proof} 

\section{Exact multiphase high-frequency solutions}
\label{section:exactsolut}
\subsection{Generalities on the error term and evolution equations}
\label{subsection:errterm}
Recall that we seek exact solutions $(A_\lambda, \Phi_\lambda)_{0<\lambda}$ to KGM under the form:
\begin{equation}
\label{eq:fullparametrixerrterm}
        A^\alpha_\lambda=A^\alpha_{1\lambda}+Z_\lambda^\alpha, \;\;\;\;
         \Phi_\lambda=\Phi_{1\lambda}+\mathcal{Z}_\lambda,         
 \end{equation}
which satisfy the Lorenz gauge condition $\partial_\alpha A^\alpha_\lambda=0$. In this Section, we assume that 
\begin{align}
\label{exporder1evosys}
&A^\alpha_{1\lambda}=A^\alpha_0+\lambda^{1/2}\sum_{\mathscr{A}\in\mathbb{A}}\cos\left(\frac{u_\mathscr{A}}{\lambda}\right)P^\alpha_\mathscr{A}+\sin\left(\frac{u_\mathscr{A}}{\lambda}\right)Q^\alpha_\mathscr{A},
&\Phi_{1\lambda}=\Phi_0+\lambda^{1/2}\sum_{\mathscr{A}\in\mathbb{A}}e^{i\frac{u_\mathscr{A}}{\lambda}}\Psi_\mathscr{A},
\end{align} 
are \textbf{first-order almost approximate solutions} to KGM in Lorenz gauge, in the sense of Definition \ref{defi:almostapproxresults}. These solutions are constructed via Proposition \ref{propal:firstapproxApprox}, and more generally Section \ref{section:Approx}, and are based on an admissible initial ansatz, in the sense of Definition \ref{defi:admissiblebginitansatz}. In other words, we assume that step \ref{item:item1Idea} of Section \ref{subsection:Idea} is completed. We further assume that step \ref{item:item2Idea} is completed, namely that the initial data for \eqref{eq:fullparametrixerrterm} solve the constraint of KGM in Lorenz gauge \eqref{defi:constraintfullInit}, and that the initial data for the error term possesses both the appropriate structure and smallness properties. The construction of such initial data is detailed in Section \ref{section:Init}, and in particular in Proposition \ref{propal:specsetconsinit}. The present Section corresponds to steps \ref{item:item3Idea}, \ref{item:item4Idea} and  \ref{item:item5Idea}. 
\begin{nota}
\label{nota:schemnotaparametrixerrterm}
   Since
    $\cos\left(\frac{u_\mathscr{A}}{\lambda}\right)P^\alpha_\mathscr{A}+\sin\left(\frac{u_\mathscr{A}}{\lambda}\right)Q^\alpha_\mathscr{A}$ may be shorten as $\Re\left(e^{i\frac{u_\mathscr{A}}{\lambda}}\overline{W^\alpha_\mathscr{A}}\right)$ for $W^\alpha_\mathscr{A}=P^\alpha_\mathscr{A}+iQ^\alpha_\mathscr{A}$, we introduce the schematic notation 
    \begin{equation}
        \label{eq:schemfullparametrixerrterm}      
        \textbf{F}_\lambda=\underbrace{\textbf{F}_0+\lambda^{1/2}\sum_{\mathscr{A}\in\mathbb{A}}e^{i\frac{u_\mathscr{A}}{\lambda}}\textbf{F}_\mathscr{A}}_\text{$\textbf{F}_{1\lambda}$}+\textbf{Z}_\lambda.
    \end{equation}
    Throughout the paper, bold symbols, such as $\textbf{Z}_\lambda$, denote both the potential and the wave function unknowns at a schematic level. Recall also Notation \ref{nota:b0background},:
    $\textbf{B}_0'$ denotes the background terms $(\textbf{F}_0, \textbf{F}_{\mathscr{A}},\textbf{d}u_\mathscr{A})$, while $\textbf{B}_0$ denotes the background terms excluding the phases $(\textbf{d}u_\mathscr{A})_{\mathscr{A}\in\mathbb{A}}$. 
\end{nota}
With Notation \ref{nota:schemnotaparametrixerrterm}, the KGM system \eqref{eq:KGMintro} in Lorenz gauge, namely the KGML system \eqref{eq:KGMLgeneKGM}, reduces to the schematic equation
\begin{equation}
\label{eq:schemeqError}
    \Box \textbf{F}_\lambda=\textbf{F}_\lambda\partial\textbf{F}_\lambda+(\textbf{F}_\lambda)^3.
\end{equation}
The gauge propagation is addressed in Section \ref{subsubsection:gpropag}.
The error term $\textbf{Z}_\lambda$ (denoting $(Z_\lambda,\mathcal{Z}_\lambda)$) must therefore satisfy schematically 
\begin{equation}
\label{eq:schemeqerrorerrterm}
    \Box \textbf{Z}_\lambda=\textbf{Z}_\lambda\partial\textbf{Z}_\lambda+\mathscr{L}_1\left(\textbf{Z}_\lambda,\partial \textbf{B}_0,\textbf{B}'_0,\frac{\mathscr{U}_\mathbb{A}}{\lambda}\right)+\mathscr{L}_2\left(\partial\textbf{Z}_\lambda,\textbf{B}_0,\frac{\mathscr{U}_\mathbb{A}}{\lambda}\right)+\mathscr{H}\left(\textbf{Z}_\lambda,\textbf{B}_0,\frac{\mathscr{U}_\mathbb{A}}{\lambda}\right)+\lambda^{1/2}\boldsymbol{\Xi}_\lambda(\textbf{B}_0)+\tilde{\boldsymbol{\Xi}}_\lambda(\textbf{B}_0),
\end{equation}
so that the parametrix \eqref{eq:schemfullparametrixerrterm} solves \eqref{eq:schemeqError}, provided that $\textbf{F}_{1\lambda}$ is an almost approximate solution to \eqref{eq:schemeqError}.
Here, $\mathscr{L}_1$  (resp. $\mathscr{L}_2$) is linear in its first three (resp. first two) arguments and periodic in its fourth (resp. third) argument, while $\mathscr{H}$ is at least quadratic and at most cubic in its first argument and periodic in its third argument. The nonlinear term $\textbf{Z}_\lambda\partial\textbf{Z}_\lambda$ is explicit. The terms $\boldsymbol{\Xi}_\lambda(\textbf{B}_0)$ and $\tilde{\boldsymbol{\Xi}}_\lambda(\textbf{B}_0)$ denote the remainders arising from the approximation procedure introduced in Definition \ref{defi:almostapproxresults}, see Notation \ref{nota:xixi}.\\ \\
The system \eqref{eq:schemeqerrorerrterm} (represented here by a single equation) is clearly locally well-posed for $\lambda>0$ and for initial data $(\textbf{Z}_\lambda,\dot{\textbf{Z}}_\lambda)\in H^2\times H^1$, although it is far from obvious that the time of existence is uniform in $\lambda$.
To obtain such a uniform time of existence, we refine the structure of the error term by writing it as
\begin{equation}  
\label{eq:schemfullerrorparametrixerrterm}
\textbf{Z}_\lambda=\sum_{\mathscr{A}\in\mathbb{A}}\lambda^{1}\textbf{F}^+_{\lambda\mathscr{A}}e^{i\frac{u_\mathscr{A}}{\lambda}}+\sum_{(\mathscr{A},\mathscr{B})\in\mathscr{C}}\lambda^{1}\breve{\textbf{F}}^+_{\mathscr{A}\pm\mathscr{B}}e^{i\frac{u_{\mathscr{A}}\pm u_{\mathscr{B}}}{\lambda}}+\boldsymbol{\mathcal{E}}_\lambda^{ell}+\boldsymbol{\mathcal{E}}_\lambda^{evo}.
\end{equation} 
The components $\breve{\textbf{F}}^+_{\mathscr{A}\pm\mathscr{B}}$ and $\boldsymbol{\mathcal{E}}_\lambda^{ell}$ are the \textbf{decoupled error components} and solve
\begin{equation}
\label{eq:schemFABerrterm}
 \forall(\mathscr{A},\mathscr{B})\in\mathscr{C},\;\mathscr{L}_{\mathscr{A}\pm\mathscr{B}} \breve{\textbf{F}}^+_{\mathscr{A}\pm\mathscr{B}} =\textbf{K}_{(\mathscr{A}\pm\mathscr{B})}(\textbf{B}_0')+(\partial u_{\mathscr{A}}\pm \partial u_{\mathscr{B}})\textbf{B}_0 \breve{\textbf{F}}_{\mathscr{A}\pm\mathscr{B}}^+,
\end{equation}
\begin{equation}
\label{eq:schemellipticerrorerrterm}
    \boldsymbol{\mathcal{E}}_\lambda^{ell}=-\lambda^2\sum_{(\mathscr{A},\mathscr{B})\in\mathscr{S}}e^{i\frac{u_\mathscr{B}\pm u_\mathscr{B}}{\lambda}}\frac{\textbf{S}_{(\mathscr{A}\pm\mathscr{B})}(\textbf{B}_0 ')}{\partial_\alpha(u_\mathscr{A}\pm u_\mathscr{B})\partial^\alpha(u_\mathscr{A}\pm u_\mathscr{B})},
\end{equation}
for $\textbf{S}_{(\mathscr{A}\pm\mathscr{B})}(\textbf{B}_0 ')$ and $\textbf{K}_{(\mathscr{A}\pm\mathscr{B})}(\textbf{B}_0 ')$ given by Proposition \ref{propal:KandSApprox} and Notations \ref{nota:schemanotApprox}.\\\\
The components $(\boldsymbol{\mathcal{E}}_\lambda^{evo}, \textbf{F}^+_{\lambda\mathscr{A}})$, together with the auxiliary function $\textbf{G}^+_{\lambda\mathscr{A}}$ (introduced to handle the derivative loss discussed in point \ref{item:obst3Idea} of Section \ref{subsection:Idea}) are the \textbf{coupled error components}. They solve the coupled system of
\begin{equation}
\begin{cases}
\label{eq:schemfullsystemerrterm}
&\Box\boldsymbol{\mathcal{E}}_\lambda^{evo}=\boldsymbol{\mathcal{E}}_\lambda^{evo}\partial\boldsymbol{\mathcal{E}}_\lambda^{evo}+\sum_{\mathscr{A}\in\mathbb{A}}e^{i\frac{u_\mathscr{A}}{\lambda}}(i\partial u_\mathscr{A}\Pi_{\kappa,+}(\boldsymbol{\mathcal{E}}_\lambda^{evo}\textbf{F}^+_{\lambda\mathscr{A}})-\lambda^{1}\textbf{G}^+_{\lambda\mathscr{A}})+[\ldots]_{\boldsymbol{\mathcal{E}}_\lambda^{evo}},\\\\
&\mathscr{L}_\mathscr{A} \textbf{F}^+_{\lambda\mathscr{A}} =\lambda^{-1/2}\partial u_\mathscr{A}\boldsymbol{\mathcal{E}}_\lambda^{evo}\textbf{B}_0   +\partial u_\mathscr{A}\Pi_{\kappa,-}(\boldsymbol{\mathcal{E}}_\lambda^{evo}\textbf{F}^+_{\lambda\mathscr{A}})+\partial u_\mathscr{A}\textbf{F}_{\lambda\mathscr{A}}^+\textbf{B}_0 , \\\\
    &\mathscr{L}_\mathscr{A}\textbf{G}^+_{\lambda\mathscr{A}}=[\mathscr{L}_A,\Box]\textbf{F}^+_{\lambda\mathscr{A}}+\lambda^{-1/2}\partial u_\mathscr{A}\Box\boldsymbol{\mathcal{E}}_\lambda^{evo}\textbf{B}_0+\Box(\partial u_\mathscr{A}\Pi_{\kappa,-}(\boldsymbol{\mathcal{E}}_\lambda^{evo}\textbf{F}^+_{\lambda\mathscr{A}}))\\
    &+\partial u_\mathscr{A}\Pi_{\kappa,-}(\boldsymbol{\mathcal{E}}_\lambda^{evo}\textbf{G}^+_{\lambda\mathscr{A}})-\partial u_\mathscr{A}\Pi_{\kappa,-}(\boldsymbol{\mathcal{E}}_\lambda^{evo}\Box\textbf{F}^+_{\lambda\mathscr{A}})+\partial u_\mathscr{A}\textbf{G}^+_{\lambda\mathscr{A}}\textbf{B}_0+[\ldots]_{\textbf{G}^+_{\lambda\mathscr{A}}},\\
\end{cases}
\end{equation}
where the $[\ldots]_{\boldsymbol{\mathcal{E}}_\lambda^{evo}}$ and the $[\ldots]_{\textbf{G}^+_{\lambda\mathscr{A}}}$ depend on $\boldsymbol{\mathcal{E}}_\lambda^{ell}$, $\breve{\textbf{F}}^+_{\mathscr{A}\pm\mathscr{B}}$, the background $\textbf{B}_0$ or lower-order terms. We refer to Sections \ref{subsubsection:estiFplus}, \ref{subsubsection:estiGplus} and \ref{subsubsection:estiEevo} for their precise form. The projectors $\Pi_{\kappa,+}$ and $\Pi_{\kappa,-}$ are introduced in Definition \ref{defi:projfreq} and depend on a parameter $\kappa>0$, which is fixed in Section \ref{subsubsection:estiGplus}.\\\\
The decoupled error components obey linear transport equations or depend only on the background, and are therefore well-defined on the entire interval of existence $[0,T]$ of the background solution. Consequently, the decoupled components are not part of the bootstrap argument. By contrast, the coupled error components satisfy nonlinear wave and transport equations that depend on both the background and the decoupled components, and therefore require a bootstrap argument to extend their existence interval up to $[0,T]$. \\\\
At the level of the error term, Steps \ref{item:item3Idea}, \ref{item:item4Idea}, and \ref{item:item5Idea} are summarized by the following three propositions.
\begin{propal}[Exact solutions to KGML]
\label{propal:exaKGMLmainpropal}
Let $(\textbf{f}^+_{\lambda\mathscr{A}},\textbf{g}^+_{\lambda\mathscr{A}},\breve{\textbf{f}}^+_{\mathscr{A}\pm\mathscr{B}},\textbf{e}^{evo}_\lambda,\dot{\textbf{e}}^{evo}_\lambda,\textbf{e}_\lambda^{ell},\dot{\textbf{e}}_\lambda^{ell})$ be initial data for the error components that are \textbf{admissible for KGML} in the sense of Definition \ref{defi:admissibleKGMLcriteria}. Then, there exist error components $(\textbf{F}^+_{\lambda\mathscr{A}},\textbf{G}^+_{\lambda\mathscr{A}},\breve{\textbf{F}}^+_{\mathscr{A}\pm\mathscr{B}},\boldsymbol{\mathcal{E}}^{evo}_\lambda,\boldsymbol{\mathcal{E}}^{ell}_\lambda)$ satisfying \eqref{eq:schemFABerrterm}, \eqref{eq:schemellipticerrorerrterm} and \eqref{eq:schemfullsystemerrterm}. The associated error term $\textbf{Z}_\lambda$ defined by \eqref{eq:schemfullerrorparametrixerrterm}, when added to the approximate solution $\textbf{F}_{1\lambda}$, yields an exact solution $\textbf{F}_\lambda=\textbf{F}_{1\lambda}+\textbf{Z}_\lambda$ to KGML \eqref{eq:schemeqError}, defined on some interval $[0,t_\lambda]$. 
\end{propal} 
\begin{rem}
\label{rem:nonschemparaerrterm}
    In a non-schematic form, the parametrix $\textbf{F}_\lambda$ reads
\begin{equation} 
 \label{eq:fullparametrixAerrterm} 
\begin{split}
&A^\alpha_\lambda=A^\alpha_0+\lambda^{1/2}\sum_{\mathscr{A}\in\mathbb{A}}\Re\left(e^{i\frac{u_\mathscr{A}}{\lambda}}\overline{W^\alpha_\mathscr{A}}\right)+\sum_{\mathscr{A}\in\mathbb{A}}\lambda^{1}\Re\left(e^{i\frac{u_\mathscr{A}}{\lambda}}\overline{W^{+\alpha}_{\lambda\mathscr{A}}}\right)+\sum_{(\mathscr{A},\mathscr{B})\in\mathscr{C}}\lambda^{1}\Re\left(e^{i\frac{u_{\mathscr{A}}\pm u_{\mathscr{B}}}{\lambda}}\overline{\breve{W}^{+\alpha}_{\mathscr{A}\pm\mathscr{B}}}\right)\\
&+(E_\lambda^{evo})^{\alpha}+(E_\lambda^{ell})^{\alpha}, 
\end{split}
 \end{equation}
\begin{equation}
\label{eq:fullparametrixPhierrterm} 
\Phi_\lambda=\Phi_0+\lambda^{1/2}\sum_{\mathscr{A}\in\mathbb{A}}e^{i\frac{u_\mathscr{A}}{\lambda}}\Psi_\mathscr{A}+\sum_{\mathscr{A}\in\mathbb{A}}\lambda^{1}\Psi^+_{\lambda\mathscr{A}}e^{i\frac{u_\mathscr{A}}{\lambda}}+\sum_{(\mathscr{A},\mathscr{B})\in\mathscr{C}}\lambda^{1}\breve{\Psi}^+_{\mathscr{A}\pm\mathscr{B}}e^{i\frac{u_{\mathscr{A}}\pm u_{\mathscr{B}}}{\lambda}}+\mathcal{E}_\lambda^{evo}+\mathcal{E}_\lambda^{ell}.
    \end{equation} 
\end{rem}

\begin{propal}[Uniform existence time in $\lambda$]
\label{propal:exasmallKGMmainpropal}
Let $(\textbf{F}^+_{\lambda\mathscr{A}},\textbf{G}^+_{\lambda\mathscr{A}},\boldsymbol{\mathcal{E}}_\lambda^{evo},\boldsymbol{\mathcal{E}}_\lambda^{ell},\breve{\textbf{F}}^+_{\mathscr{A}\pm\mathscr{B}})$ be given by the Proposition \ref{propal:exaKGMLmainpropal}. If, in addition, the initial data satisfy \textbf{the required smallness} in the sense of Definition \ref{defi:reqsmallnesscriteria}, then there exists $\lambda_0>0$, such that the family of solutions $(\textbf{F}_\lambda)_{0<\lambda<\lambda_0}$ is defined on the full interval $[0,T]$.
\end{propal} 

\begin{propal}[Exact solutions to KGM in Lorenz gauge]
\label{propal:exaKGMmainpropal}
Let $(\textbf{F}^+_{\lambda\mathscr{A}},\textbf{G}^+_{\lambda\mathscr{A}},\boldsymbol{\mathcal{E}}_\lambda^{evo},\boldsymbol{\mathcal{E}}_\lambda^{ell},\breve{\textbf{F}}^+_{\mathscr{A}\pm\mathscr{B}})$ be given by the Proposition \ref{propal:exaKGMLmainpropal}. If the initial data are also \textbf{admissible for KGM in Lorenz gauge} in the sense of  Definition \ref{defi:admissibleKGMgaugecriteria}, then the parametrix $\textbf{F}_\lambda=\textbf{F}_{1\lambda}+\textbf{Z}_\lambda$ is solution to KGM \eqref{eq:KGMintro} in Lorenz gauge, i.e., the gauge propagates.
\end{propal} 
\subsection{Decoupled error components}
\label{subsection:decouplederrorcompo}
\subsubsection{Information on $\breve{\textbf{F}}^+_{\mathscr{A}\pm\mathscr{B}}$}
\label{subsubsection:infofab}
\begin{defi}
\label{defi:defifabinfofab}
For a given admissible background initial data, let $\textbf{B}'_0$ be the associated background provided by Proposition \ref{propal:firstapproxApprox}. We define the transport equation for $(\breve{\textbf{F}}^+_{\mathscr{A}\pm\mathscr{B}})_{(\mathscr{A},\mathscr{B})\in\mathscr{C}}$ by
\begin{equation}
\label{eq:schemtpfabinfofab}
   \forall (\mathscr{A},\mathscr{B})\in\mathscr{C},\;\; \mathscr{L}_{\mathscr{A}\pm\mathscr{B}} \breve{\textbf{F}}^+_{\mathscr{A}\pm\mathscr{B}} =\textbf{K}_{(\mathscr{A}\pm\mathscr{B})}(\textbf{B}_0')+(\partial u_{\mathscr{A}}\pm \partial u_{\mathscr{B}})\textbf{B}_0 \breve{\textbf{F}}_{\mathscr{A}\pm\mathscr{B}}^+. 
\end{equation}
\end{defi}
\begin{rem}
The error components $(\breve{\textbf{F}}^+_{\mathscr{A}\pm\mathscr{B}})_{(\mathscr{A},\mathscr{B})\in\mathscr{C}}$ schematically represent $(\breve{W}_{\mathscr{A}\pm\mathscr{B}},\breve{\Psi}^+_{\mathscr{A}\pm\mathscr{B}})_{(\mathscr{A},\mathscr{B})\in\mathscr{C}}$, see the parametrices \eqref{eq:fullparametrixAerrterm}. The transport equation \eqref{eq:schemtpfabinfofab} therefore corresponds, for all $(\mathscr{A},\mathscr{B})\in\mathscr{C}$, to 
\begin{equation}
\begin{cases}
\label{eq:nonschemtpFABerrterm}
\mathscr{L}_{\mathscr{A}\pm\mathscr{B}}\breve{W}_{\mathscr{A}\pm\mathscr{B}}^{+\beta}=i(\partial^\beta u_\mathscr{A}\pm\partial^\beta u_\mathscr{B})\overline{\breve{\Psi}^+_{\mathscr{A}\pm\mathscr{B}}}\Phi_0+iK^\beta_{A(\mathscr{A}\pm\mathscr{B})}(\textbf{B}_0'), \\
\mathscr{L}_{\mathscr{A}\pm\mathscr{B}}\breve{\Psi}^+_{\mathscr{A}\pm\mathscr{B}}=-i2A^\alpha_0(\partial^\alpha u_\mathscr{A}\pm\partial_\alpha u_\mathscr{B})\breve{\Psi}^+_{\mathscr{A}\pm\mathscr{B}}-iK_{\Phi(\mathscr{A}\pm\mathscr{B})}(\textbf{B}_0'),  \\ 
\end{cases}
\end{equation}
where, in particular, the RHS of the equation for $\breve{W}_{\mathscr{A}\pm\mathscr{B}}^{+\beta}$ is orthogonal to $(\partial_{\beta} u_{\mathscr{A}}\pm \partial_{\beta} u_{\mathscr{B}})$, see Proposition \ref{propal:KandSpropertiesApprox}.
\end{rem}
\begin{rem}
\label{rem:rolefabinfofab}
    The components $\breve{\textbf{F}}^+_{\mathscr{A}\pm\mathscr{B}}$ are introduced to absorb the term $\sum_{(\mathscr{A},\mathscr{B})\in\mathscr{C}}e^{i\frac{u_\mathscr{B}\pm u_\mathscr{B}}{\lambda}}\textbf{K}_{(\mathscr{A}\pm\mathscr{B})}(B_0 ')$, i.e., the $O(1)$ high-frequency terms oscillating along characteristic phases emerging from resonant interaction contained in $\tilde{\boldsymbol{\Xi}}_\lambda$. The equation \eqref{eq:schemtpfabinfofab} is \textbf{decoupled} from the system \eqref{eq:schemfullsystemerrterm}. For simplicity, and exploiting the freedom in the choice of these parameters, we impose $\breve{\textbf{F}}^+_{\mathscr{A}\pm\mathscr{B}}|_{t=0}=0$ in Proposition \ref{propal:specsetconsinit}. Further details concerning the initial data are given in Section \ref{section:Init}.
\end{rem}
\begin{rem}
\label{rem:RHSfabinfofab}
    The RHS term
     \begin{align}
&(\partial u_{\mathscr{A}}\pm \partial u_{\mathscr{B}})\textbf{B}_0 \breve{\textbf{F}}_{\mathscr{A}\pm\mathscr{B}}^+
\end{align} 
originates from $\mathscr{L}_2(\partial\textbf{Z},\textbf{B}_0,\frac{\mathscr{U}_\mathbb{A}}{\lambda})$ in \eqref{eq:schemeqerrorerrterm}, while  
the term 
\begin{align}
&\textbf{K}_{(\mathscr{A}\pm\mathscr{B})}(\textbf{B}_0')
\end{align}
arises from the decomposition of $\tilde{\boldsymbol{\Xi}}_\lambda$ given in Proposition \ref{propal:KandSApprox}.\\
\end{rem}
The next proposition establishes the smallness and regularity properties of  $(\breve{\textbf{F}}^+_{\mathscr{A}\pm\mathscr{B}}
)_{(\mathscr{A},\mathscr{B})\in\mathscr{C}}$.
\begin{propal}
\label{propal:propertiesfabinfofab}
Let $\breve{\textbf{F}}^+_{\mathscr{A}\pm\mathscr{B}}\in H^3$ with $Supp(\breve{\textbf{F}}^+_{\mathscr{A}\pm\mathscr{B}})\subset\Omega$ for $\Omega$ a compact set of $\mathbb{R}^3$.
There exists a unique local solution $\breve{\textbf{F}}^+_{\mathscr{A}\pm\mathscr{B}}\in \cap^3_{j=0}(C^j([0,T],H^{3-j})$ defined on the time interval $[0,T]$ of existence of the background, and satisfying $\breve{\textbf{F}}^+_{\mathscr{A}\pm\mathscr{B}}|_{t=0}=\breve{\textbf{F}}^+_{\mathscr{A}\pm\mathscr{B}}$. Moreover,  $Supp(\breve{\textbf{F}}^+_{\mathscr{A}\pm\mathscr{B}})\subset\mathscr{I}(T,\Omega)$ and, for $C_0$ from Notation \ref{nota:cst0Approx}, the following estimate holds
 \begin{equation}
     \label{eq:estimatesfabinfofab}
     \max_{(\mathscr{A},\mathscr{B})\in\mathscr{C}}\sum_{k\leq 3}(||\partial^{k}_t\breve{\textbf{F}}^+_{\mathscr{A}\pm\mathscr{B}}||_{L^\infty([0,T],H^{3-k})})\leq C_0.\\
 \end{equation}
\end{propal}
\begin{proof}
  Existence and uniqueness follow from the standard energy method. Finite speed of propagation yields the support property.  For \eqref{eq:estimatesfabinfofab}, we refer to Section \ref{subsection:estimatesappendix}. For \eqref{eq:estimatesfabinfofab}, we refer to Section \ref{subsection:estimatesappendix}. In particular, we apply Proposition\footnote{The linear term on the RHS of \eqref{eq:estienergytp1AX} is harmless and can be controlled in the same way as the other linear term appearing in the evolution equation.} \ref{propal:estienergytpAX} to obtain control of the $H^3$ norm, using that $\breve{\textbf{f}}^+_{\mathscr{A}\pm\mathscr{B}} \in H^3$. Then, we use Proposition \ref{propal:estienergytpderivtempsAX} to control the time derivatives.\\
\end{proof}
\begin{rem}
\label{rem:moreregfabinfofab}
    We note that $\breve{\textbf{F}}^+_{\mathscr{A}\pm\mathscr{B}}$ belongs to $H^3$ for $t\in[0,T]$, just as the background does, and is more regular than the coupled components $\textbf{F}^+_{\lambda\mathscr{A}}$ and $\boldsymbol{\mathcal{E}}_\lambda^{evo}$. This additional regularity is useful since the RHS of \eqref{eq:schemwaveEevosestiEevo}
contains the term $\Box\breve{\textbf{F}}^+_{\mathscr{A}\pm\mathscr{B}}$. 
\end{rem}
\subsubsection{Information on $\boldsymbol{\mathcal{E}}_\lambda^{ell}$}
\label{subsubsection:infoEell}
\begin{defi}
 \label{defi:defiEellinfoEell}
For a given admissible background initial data, let $\textbf{B}'_0$ denote the background provided by Proposition \ref{propal:firstapproxApprox}. We set 
\begin{equation}
\label{eq:eqEellinfoEell}
 \boldsymbol{\mathcal{E}}_\lambda^{ell}=-\lambda^2\sum_{(\mathscr{A},\mathscr{B})\in\mathscr{S}}e^{i\frac{u_\mathscr{A}\pm u_\mathscr{B}}{\lambda}}\frac{\textbf{S}_{(\mathscr{A}\pm\mathscr{B})}(\textbf{B}_0 ')}{\partial_\alpha(u_\mathscr{A}\pm u_\mathscr{B})\partial^\alpha(u_\mathscr{A}\pm u_\mathscr{B})}.   
\end{equation}
\end{defi}
\begin{rem}
\label{rem:nonschematEll}
Non-schematically, \eqref{eq:eqEellinfoEell} reads
\begin{equation} 
    \label{eq:schemEellAerrterm}
        (E^{ell}_\lambda)^{\alpha}=-\lambda^2\sum_{(\mathscr{A},\mathscr{B})\in\mathscr{S}}\frac{\Re\left(\overline{S^{\alpha}_{A(\mathscr{A}\pm\mathscr{B})}(\textbf{B}_0 ')}e^{i\frac{u_\mathscr{A}\pm u_\mathscr{B}}{\lambda}}\right)}{\partial_\alpha(u_\mathscr{A}\pm u_\mathscr{B})\partial^\alpha(u_\mathscr{A}\pm u_\mathscr{B})},   \\
\end{equation}
\begin{equation} 
    \label{eq:schemEellPhierrterm}
    \mathcal{E}^{ell}_\lambda=-\lambda^2\sum_{(\mathscr{A},\mathscr{B})\in\mathscr{S}}\frac{S_{\Phi(\mathscr{A}\pm\mathscr{B})}(\textbf{B}_0 ')e^{i\frac{u_\mathscr{A}\pm u_\mathscr{B}}{\lambda}}}{\partial_\alpha(u_\mathscr{A}\pm u_\mathscr{B})\partial^\alpha(u_\mathscr{A}\pm u_\mathscr{B})}.   
\end{equation}
We also refer to the non-schematic parametrices \eqref{eq:fullparametrixAerrterm}. These expressions are well defined since the denominator is uniformly bounded away from zero, see Proposition \ref{propal:smallnesscontrolsetphase}. 
\end{rem}
\begin{rem}
\label{rem:roleEellinfoEell}
The component $\boldsymbol{\mathcal{E}}_\lambda^{ell}$ is introduced to absorb the $O(1)$ high-frequency terms oscillating along non-characteristic phases emerging from non-resonant interactions contained in $\tilde{\boldsymbol{\Xi}}_\lambda$. More precisely, 
 \begin{equation}
\label{eq:RHSEellinfoEell}
       \Box\boldsymbol{\mathcal{E}}_\lambda^{ell}-\sum_{(\mathscr{A},\mathscr{B})\in\mathscr{S}}e^{i\frac{u_\mathscr{A}\pm u_\mathscr{B}}{\lambda}}\textbf{S}_{(\mathscr{A}\pm\mathscr{B})}(\textbf{B}_0 ')=O(\lambda^{1})f(\textbf{B}'_0), 
 \end{equation} 
while, at the same time, $\boldsymbol{\mathcal{E}}_\lambda^{ell}$ remains sufficiently small as an error component, as quantified in the following proposition.
\end{rem}

\begin{propal} 
\label{propal:propertiesEellinfoEell}
We have $Supp(\boldsymbol{\mathcal{E}}_\lambda^{ell})\subset B_{S}$ and, for $C_0$ from Notation \ref{nota:cst0Approx}, the estimates
\begin{equation}
 \label{eq:estimatesEellinfoEell}
\sum_{k\leq 2}\lambda^k||\boldsymbol{\mathcal{E}}_\lambda^{ell}||_{L^\infty([0,T],H^{k})}+\sum_{k\leq 1}\lambda^{k+1}||\partial_t\boldsymbol{\mathcal{E}}_\lambda^{ell}||_{L^\infty([0,T],H^{k})}+\lambda^2||\partial^2_{tt}\boldsymbol{\mathcal{E}}_\lambda^{ell}||_{L^\infty([0,T],L^2)}\leq \lambda^{2} C_0,\\
\end{equation}
and 
\begin{equation}
    \label{eq:estimmatesEelldefectinfoEell}
    \sum_{k\leq 1}\lambda^k||\Box\boldsymbol{\mathcal{E}}_\lambda^{ell}-\sum_{(\mathscr{A},\mathscr{B})\in\mathscr{S}}e^{i\frac{u_\mathscr{A}\pm u_\mathscr{B}}{\lambda}}\textbf{S}_{(\mathscr{A}\pm\mathscr{B})}(\textbf{B}_0 '))||_{L^\infty([0,T],H^{k})}\leq \lambda^{1} C_0.\\
\end{equation}
\end{propal}
\begin{proof}
    Using the Definition \ref{defi:defiEellinfoEell} together with Remark \ref{rem:roleEellinfoEell}, the compact support property follows immediately. The estimates are obtained by combining \eqref{eq:eta2}, which provides uniform control on the smallness of the denominator, and point \ref{item:approxit1Approx} of Proposition \ref{propal:firstapproxApprox}, which provides bounds on the background $\textbf{B}_0'$.
\end{proof} 

\subsection{Coupled error components}
\label{subsection:couplederrorcompo}
\subsubsection{System of evolution}
\label{subsubsection:evosys}
The coupled component $(\boldsymbol{\mathcal{E}}_\lambda^{evo}, \textbf{F}^+_{\lambda\mathscr{A}})$, together with the auxiliary unknown $(\textbf{G}^+_{\lambda\mathscr{A}})$, solves the system 
\begin{equation}
\begin{cases}
\label{eq:schemfullsystemevosys}
&\Box\boldsymbol{\mathcal{E}}_\lambda^{evo}=\boldsymbol{\mathcal{E}}_\lambda^{evo}\partial\boldsymbol{\mathcal{E}}_\lambda^{evo}+\sum_{\mathscr{A}\in\mathbb{A}}e^{i\frac{u_\mathscr{A}}{\lambda}}(i\partial u_\mathscr{A}\Pi_{\kappa,+}(\boldsymbol{\mathcal{E}}_\lambda^{evo}\textbf{F}^+_{\lambda\mathscr{A}})-\lambda^{1}\textbf{G}^+_{\lambda\mathscr{A}})+[\ldots]_{\boldsymbol{\mathcal{E}}_\lambda^{evo}},\\\\
&\mathscr{L}_\mathscr{A} \textbf{F}^+_{\lambda\mathscr{A}} =\lambda^{-1/2}\partial u_\mathscr{A}\boldsymbol{\mathcal{E}}_\lambda^{evo}\textbf{B}_0   +\partial u_\mathscr{A}\Pi_{\kappa,-}(\boldsymbol{\mathcal{E}}_\lambda^{evo}\textbf{F}^+_{\lambda\mathscr{A}})+\partial u_\mathscr{A}\textbf{F}_{\lambda\mathscr{A}}^+\textbf{B}_0,  \\\\
    &\mathscr{L}_\mathscr{A}\textbf{G}^+_{\lambda\mathscr{A}}=[\mathscr{L}_A,\Box]\textbf{F}^+_{\lambda\mathscr{A}}+\lambda^{-1/2}\partial u_\mathscr{A}\Box\boldsymbol{\mathcal{E}}_\lambda^{evo}\textbf{B}_0+\Box(\partial u_\mathscr{A}\Pi_{\kappa,-}(\boldsymbol{\mathcal{E}}_\lambda^{evo}\textbf{F}^+_{\lambda\mathscr{A}})),\\
    &+\partial u_\mathscr{A}\Pi_{\kappa,-}(\boldsymbol{\mathcal{E}}_\lambda^{evo}\textbf{G}^+_{\lambda\mathscr{A}})-\partial u_\mathscr{A}\Pi_{\kappa,-}(\boldsymbol{\mathcal{E}}_\lambda^{evo}\Box\textbf{F}^+_{\lambda\mathscr{A}})+\partial u_\mathscr{A}\textbf{G}^+_{\lambda\mathscr{A}}\textbf{B}_0+[\ldots]_{\textbf{G}^+_{\lambda\mathscr{A}}}.\\
\end{cases}
\end{equation}
All terms appearing in \eqref{eq:schemfullsystemevosys} are detailed in the following Sections. In Section \eqref{subsubsection:WP}, we first establish the local well-posedness of \eqref{eq:schemfullsystemevosys}. Then, in Section \eqref{subsubsection:ass}, we state suitable a priori bounds for the coupled error components. Through Sections \ref{subsubsection:estiFplus}, \ref{subsubsection:estiFplus} and \ref{subsubsection:estiEevo}, we derive energy estimates for $\textbf{F}^+_{\lambda\mathscr{A}}$, $\textbf{G}^+_{\lambda\mathscr{A}}$ and $\boldsymbol{\mathcal{E}}_\lambda^{evo}$, respectively. Finally, in Sections \ref{subsubsection:bootstrap}, we exhibit a certain $0<\lambda_0$ such that the latter energy estimates allow one to improve the a priori bounds for $0<\lambda<\lambda_0$, thereby closing a bootstrap argument.
\begin{nota}
\label{nota:notaepsilonunderass}
     We denote by $\underline{\boldsymbol{\mathcal{E}}^{evo}_\lambda}$ all the components of $\boldsymbol{\mathcal{E}}_\lambda^{evo}$ except $(E_\lambda^{evo})^{0}$ in the parametrices introduced in Remark \ref{rem:nonschemparaerrterm}. The quantity $(E_\lambda^{evo})^{0}$ corresponds to the part of the error term associated with $A^0_\lambda$. Its initial data are obtained by solving the constraint equation \eqref{eq:initconstraintfullInit}. To invert the Laplacian, we resort to weighted Sobolev norms, which implies that $(E_\lambda^{evo})^{0}$ is in $H^{2}_{\delta}$, for some $-3/2<\delta<-1/2$, and therefore need not be in $L^2$.
\end{nota}
\subsubsection{Well-posedness}
\label{subsubsection:WP}
In this Section, we establish the well-posedness of \eqref{eq:schemfullsystemevosys}.
\begin{propal}
\label{propal:wellposfullsystWP}
For a given admissible background initial data set in the sense of Definition \ref{defi:admissiblebginitansatz}, let $\textbf{B}'_0$ be the background given by Proposition \ref{propal:firstapproxApprox}, let $\lambda>0$ and let $\kappa>0$ be the constant appearing in the Definition \ref{defi:projfreq} of the projectors $\Pi_{\kappa,-}$ and $\Pi_{\kappa,+}$. Let $\boldsymbol{\breve{F}}^+_{\mathscr{A}\pm\mathscr{B}}$ be given by Proposition \ref{propal:propertiesfabinfofab} and let $\boldsymbol{\mathcal{E}}_\lambda^{ell}$ be given by Definition \ref{defi:defiEellinfoEell} from initial data $(\textbf{f}^+_{\lambda\mathscr{A}},\textbf{g}^+_{\lambda\mathscr{A}},\breve{\textbf{f}}^+_{\mathscr{A}\pm\mathscr{B}},\textbf{e}^{evo},\dot{\textbf{e}}^{evo},\textbf{e}^{ell},\dot{\textbf{e}}^{ell})$ which are \textbf{admissible for KGML} in the sense of Definition \ref{defi:admissibleKGMLcriteria}. Then, the Cauchy problem associated with \eqref{eq:schemfullsystemevosys} is locally well-posed for the initial data $(\textbf{f}^+_{\lambda\mathscr{A}},\textbf{g}^+_{\lambda\mathscr{A}}\textbf{e}^{evo},\dot{\textbf{e}}^{evo})$. There exists a time $t_\lambda$ such that there exists a solution $(\textbf{F}^+_{\lambda\mathscr{A}},\textbf{G}^+_{\lambda\mathscr{A}},\boldsymbol{\mathcal{E}}_\lambda^{evo})$ with the following regularity
    \begin{align*}    
    &\forall\mathscr{A}\in\mathbb{A},\;\textbf{F}^+_{\lambda\mathscr{A}}\in\bigcap_{j=0}^2 C^{j}([0,t_\lambda],H^{2-j}),\;\forall\mathscr{A}\in\mathbb{A},\;\textbf{G}^+_{\lambda\mathscr{A}}\in\bigcap_{j=0}^1 C^{j}([0,t_\lambda],H^{1-j}),\\
    &\underline{\boldsymbol{\mathcal{E}}^{evo}_\lambda}\in\bigcap_{j=0}^2 C^{j}([0,t_\lambda],H^{2-j}),\;(E^{evo})^{0}\in\bigcap_{j=0}^2 C^{j}([0,t_\lambda],H^{2-j}_{\delta+j}).
    \end{align*}
Moreover, $Supp(\textbf{F}^+_{\lambda\mathscr{A}},\textbf{G}^+_{\lambda\mathscr{A}},\underline{\boldsymbol{\mathcal{E}}^{evo}_\lambda})\subset B_S$.
\end{propal}
\begin{proof}
    In terms of derivatives, we observe that
\begin{align*} 
&\partial\boldsymbol{\mathcal{E}}_\lambda^{evo}\sim \textbf{G}^+_{\lambda\mathscr{A}},\; \partial\textbf{F}^+_{\lambda\mathscr{A}},\\
&\textbf{F}^+_{\lambda\mathscr{A}}\sim  \boldsymbol{\mathcal{E}}_\lambda^{evo},\\
&\textbf{G}^+_{\lambda\mathscr{A}}\sim\partial\textbf{F}^+_{\lambda\mathscr{A}},\;\partial\boldsymbol{\mathcal{E}}_\lambda^{evo}.
\end{align*}
More precisely, in the equation for $\textbf{G}^+_{\lambda\mathscr{A}}$, we use the evolution equation to replace $\Box \boldsymbol{\mathcal{E}}_\lambda^{evo}$ by its RHS. This produces terms that are at most linear in $\partial\boldsymbol{\mathcal{E}}_\lambda^{evo}$, $\partial\textbf{F}^+_{\lambda\mathscr{A}}$ or $\textbf{G}^+_{\lambda\mathscr{A}}$, and other contributions that belong to $H^2$ and thus to $L^\infty$. We also use the commutator estimates of Section \ref{subsection:commut}. The nonlinearity $\Pi_{\kappa,-}(\partial\textbf{F}^+_{\lambda\mathscr{A}}\partial\boldsymbol{\mathcal{E}}_\lambda^{evo})$ is handled using the projector on low-frequency, which yields arbitrary regularity gains for any $\kappa>0$. All the remaining nonlinearities are compatible with the regularity stated in the Proposition.
In fact, several of these bounds are derived in detail in the bootstrap Sections \ref{subsubsection:estiFplus}, \ref{subsubsection:estiGplus}, and \ref{subsubsection:estiEevo}. Finally, for the support property, admissibility for KGML ensures
\\$Supp(\textbf{f}^+_{\lambda\mathscr{A}},\textbf{g}^+_{\lambda\mathscr{A}},\underline{\textbf{e}}^{evo},\dot{\underline{\textbf{e}}}^{evo})\subset B_{S'}$ so that, by finite speed of propagation, we recover the stated support property.
\end{proof}

\begin{rem}
\label{rem:supportestimWP}
Combining Propositions \ref{propal:propertiesfabinfofab}, \ref{propal:propertiesEellinfoEell} and \ref{propal:wellposfullsystWP} we obtain 
$Supp(\textbf{F}^+_{\lambda\mathscr{A}},\textbf{G}^+_{\lambda\mathscr{A}},\breve{\textbf{F}}^+_{\mathscr{A}\pm\mathscr{B}},\underline{\boldsymbol{\mathcal{E}}_\lambda}^{evo},\boldsymbol{\mathcal{E}}_\lambda^{ell})\subset B_S$. The quantity $(E^{evo}_\lambda)^{0}$ may not be compactly supported but belongs to $H^2_{\delta}$. Nevertheless, every RHS term of the evolution equations \eqref{eq:schemfullsystemevosys} is compactly supported, since $A^0$ and $(E^{evo}_\lambda)^0$ always appear in products with compactly supported terms. Consequently, in the estimates performed in Sections \ref{subsubsection:estiFplus}, \ref{subsubsection:estiGplus} and \ref{subsubsection:estiEevo}, the
$||.||_{L^p}$ norm is effectively taken over $B_S$, thus equivalent to $||.||_{L^p(B_S)}$. 
 Moreover, for a function $f$ in $H^m_{\delta}$, we also have 
$$||f||_{H^m(B_S)}\leq C_S||f||_{H^m_{\delta}}.$$
Hence, we abuse the notation to replace all the weighted norms with classical Sobolev norms in the following estimates.\\
\end{rem}
\begin{rem}
\label{rem:strichsupportpbWP}
    When applying Strichartz estimates of Theorem \ref{unTheorem:StrichartzAx} or the Lemmas \ref{lem:lemma1Ax} and \ref{lem:lemma2Ax} to $(E^{evo}_\lambda)^0$, one requires bounds on $||\partial_t(E^{evo}_\lambda)^0(0)||_{L^2}$ and $||(E^{evo}_\lambda)^0(0)||_{\dot{H}^1}$. Because $(E^{evo}_\lambda)^0|_{t=0}$ (resp. $\partial_t(E^{evo}_\lambda)^0|_{t=0}$) is only controlled in $H^2_\delta$ (resp.  $H^1_{\delta+1}$) for $-3/2<\delta<-1/2$, we have no information on the previous norms. However, as noted in Remark \ref{rem:supportestimWP}, the term $(E^{evo}_\lambda)^0$ always appears in products with compactly supported terms. Therefore, the equations only involve $(E^{evo}_\lambda)^0|_{B_S}$. The latter coincides with $E^{'0}_\lambda$ given by  
    \begin{equation}
        \begin{cases}
            \Box E^{'0}_\lambda=\Box (E^{evo}_\lambda)^0,\\
            E^{'0}_\lambda(t=0)=(E^{evo}_\lambda)^0(t=0)|_{\mathscr{I}(T,B_S)},\\    \partial_tE^{'0}_\lambda(t=0)=(\partial_t(E^{evo}_\lambda)^0)(t=0)|_{\mathscr{I}(T,B_S)},
        \end{cases}
    \end{equation}
    where $\mathscr{I}(T,B_S)$ denotes the propagation of $B_S$ defined in Notation \ref{nota:supportpropaggenenot}. Thus, Strichartz estimates can be applied to $E^{'0}_\lambda$, which satisfies the required norm bounds thanks to its compact support.
\end{rem}
  \subsubsection{Assumptions}
\label{subsubsection:ass}
To extend the time of existence $0<t_\lambda$ obtained in Proposition \ref{propal:wellposfullsystWP} to the time of existence of the background $T$ (for $\lambda$ small enough), we use a bootstrap argument. We therefore assume a priori bounds on the coupled evolution components and later show that these bounds can be improved in the following Sections. 
\begin{ass}
\label{ass:evoass}
  There exist $c_1>0$ and $c_2>0$ such that for all $\tau\in[0,T]$
  \begin{equation}
  \label{eq:assumptions}
\begin{split}
&\max_{\mathscr{A}\in\mathbb{A}}(\sum_{k\leq 1}\lambda^k( ||\textbf{F}_{\lambda\mathscr{A}}^+(\tau)||_{H^{k+1}}+||\partial_t\textbf{F}_{\lambda\mathscr{A}}^+(\tau)||_{H^k})+\lambda||\partial^2_{tt}\textbf{F}_{\lambda\mathscr{A}}^+(\tau)||_{L^2})\leq c_1e^{\tau c_2},\\
 &\sum_{k\leq 1}\lambda^k(||\underline{\boldsymbol{\mathcal{E}}^{evo}_\lambda}(\tau)||_{H^{k+1}}+||\partial_t\underline{\boldsymbol{\mathcal{E}}^{evo}_\lambda}(\tau)||_{H^k})+\lambda||\partial^2_{tt}\underline{\boldsymbol{\mathcal{E}}^{evo}_\lambda}(\tau)||_{L^2}\leq \lambda^{1/2} c_1e^{\tau c_2},\\
 &\sum_{k\leq 1} \lambda^k(||(E_\lambda^{evo})^{0}(\tau)||_{H^{k+1}_{\delta}}+||\partial_t(E_\lambda^{evo})^{0}(\tau)||_{H^k_{\delta+1}})+\lambda||\partial^2_{tt}(E_\lambda^{evo})^{0}(\tau)||_{L^2_{\delta+2}}\leq \lambda^{1/2} c_1e^{\tau c_2},\\
  &\max_{\mathscr{A}\in\mathbb{A}}(\sum_{k\leq 1}( \lambda^k||\textbf{G}^+_{\lambda\mathscr{A}}(\tau)||_{H^{k}})+\lambda||\partial_t\textbf{G}^+_{\lambda\mathscr{A}}(\tau)||_{L^2})\leq c_1e^{\tau c_2}.  
\end{split}
\end{equation}
\end{ass}
\begin{nota}
\label{nota:cst12ass}
    We denote by $C_{0,1,2,\tau}=C_{0,1,2,\tau}(c_0,\frac{1}{\eta_0},N,T,c_1,c_2,\tau)$ any constant which depends polynomially on $c_1e^{\tau c_2}$, $\frac{c_1}{c_2}e^{\tau c_2}$, and the background constants $c_0$, $\frac{1}{\eta_0}$, $N=|\mathbb{A}|$ and $T$, where $c_1$ and $c_2$ are from Assumptions \ref{ass:evoass}. We also recall that $c_0$ and $T$ are provided in Proposition \ref{propal:firstapproxApprox}, $N$ denotes the number of phases introduced in Definition \ref{defi:initialphasesetsetphase}, and $\frac{1}{\eta_0}$ is given by Proposition \ref{propal:smallnesscontrolsetphase} for the compact set $\Omega=B_S$ from Proposition \ref{propal:firstapproxApprox}.\\
\end{nota} 
We recall that $C_0$ designates any constant depending only on the background data, see Notation \ref{nota:cst0Approx}. In general, the constants satisfy the hierarchy:
\begin{align}
\lambda\ll c_0\sim C_0<c_1\ll c_2\ll C_{0,1,2,\tau}.
\end{align}
Only the background constants $c_0$ and $C_0$ are fixed. The bootstrap constants $c_1$ and $c_2$ are free parameters at this stage and will be determined later in Lemma \ref{lem:improvmentbootstrap} in order to improve the bootstrap Assumptions \ref{ass:evoass}. Increasing $c_1$ and $c_2$ alone is not sufficient to control the nonlinearities. Indeed, nonlinear interactions introduce constants of the form $C_{0,1,2,\tau}$, which grow exponentially with respect to $c_2$, in the estimates. In fact, in every occurrence of $C_{0,1,2,\tau}$ within the estimates, an additional small factor $\lambda^\varepsilon$ (for some $\varepsilon>0$) is present. This extra smallness allows the bootstrap bounds to be improved for $\lambda$ sufficiently small. We highlight that this mechanism is compatible with the error term being constructed after a first-order approximation.
Now, we give a useful notation before we sum up the discussion on the role of the constants in the bootstrap.
\begin{nota}
\label{nota:lambdaplus}
    For $c\in\mathbb{R}$, the notation $\lambda^{c+}$ stands for $\lambda^{c+\varepsilon}$ for some $\varepsilon>0$.\\
\end{nota} 
To improve the bounds of Assumptions \ref{ass:evoass}, we rely on one or a combination of the following mechanisms, applied term by term:
\begin{enumerate}[label=\alph*)]
    \item \label{item:item1improveass} The estimates depend in fact only on background norms that remain uniformly bounded on $[0,T]$. ($\leq C_0$)
    \item \label{item:item2improveass} The estimates depend on the bootstrap quantities but with a strictly higher power of $\lambda$, so that, for $\lambda$ small enough, they can be improved. ($\leq \lambda^{0+}C_{0,1,2,\tau}$)
    \item \label{item:item3improveass} The estimates depend only linearly on the bootstrap quantities integrated in time. ($C_0\frac{c_1}{c_2}e^{\tau c_2}$) 
\end{enumerate}   
\subsubsection{Estimates for $\textbf{F}^+_{\lambda\mathscr{A}}$}
\label{subsubsection:estiFplus}
Throughout this Section, we denote with $C_0$ any constant that only depends on the background data and with $c_1$, $c_2$ and $C_{0,1,2,\tau}$ the constants of the bootstrap. See Notations \ref{nota:cst0Approx} and \ref{nota:cst12ass}. We also use the schematic Notation \ref{nota:schemnotaparametrixerrterm}, where background terms are represented by $\textbf{B}_0$ and $\textbf{B}_0'$. The constant $\kappa>0$ designates the constant of the projectors $\Pi_{\kappa,-}$ and $\Pi_{\kappa,+}$ introduced in Definition \ref{defi:projfreq}.\\\\
The transport equation for $(\textbf{F}^+_{\lambda\mathscr{A}})_{\mathscr{A}\in\mathbb{A}}$ reads
\begin{equation}
\label{eq:schemtpFplusestiFplus}
   \forall\mathscr{A}\in\mathbb{A},\;\mathscr{L}_\mathscr{A} \textbf{F}^+_{\lambda\mathscr{A}} =\lambda^{-1/2}\partial u_\mathscr{A}\boldsymbol{\mathcal{E}}_\lambda^{evo}\textbf{B}_0  +\partial u_\mathscr{A}\Pi_{\kappa,-}(\boldsymbol{\mathcal{E}}_\lambda^{evo}\textbf{F}^+_{\lambda\mathscr{A}})+\partial u_\mathscr{A}\textbf{F}_{\lambda\mathscr{A}}^+\textbf{B}_0.  \\ 
\end{equation}
\begin{rem}
\label{rem:nonschemtperrterm}
The error components $(\textbf{F}^+_{\lambda\mathscr{A}})_{\mathscr{A}\in\mathbb{A}}$ denote schematically $(W^{+}_{\lambda\mathscr{A}},\Psi^+_{\lambda\mathscr{A}})_{\mathscr{A}\in\mathbb{A}}$ in the parametrices \eqref{eq:fullparametrixAerrterm}. The transport equation \eqref{eq:schemtpFplusestiFplus} represents schematically for every $\mathscr{A}\in\mathbb{A}$, 
\begin{equation}
\begin{cases}
\label{eq:nonschemtpFpluserrterm}
\mathscr{L}_{\mathscr{A}}W^{+\beta}_{\lambda\mathscr{A}}=i\partial^\beta u_\mathscr{A}\overline{\Psi^+_{\lambda\mathscr{A}}}\Phi_0+\lambda^{-1/2}i\partial^\beta u_\mathscr{A}\overline{\Psi_\mathscr{A}}\mathcal{E}_\lambda^{evo}+i\partial^\beta u_\mathscr{A}\Pi_ {-}(\overline{\Psi^+_{\lambda\mathscr{A}}}\mathcal{E}_\lambda^{evo}),\\
\mathscr{L}_{\mathscr{A}}\Psi^+_{\lambda\mathscr{A}}=-i2A^\alpha_0\partial_\alpha u_\mathscr{A}\Psi^+_{\lambda\mathscr{A}}-\lambda^{-1/2}i2(E^{evo}_\lambda)^{\alpha}\partial_\alpha u_\mathscr{A}\Psi_\mathscr{A}-i2\partial_\alpha u_\mathscr{A}\Pi_{\kappa,-}((E^{evo}_\lambda)^{\alpha}\Psi^+_{\lambda\mathscr{A}}), \\ 
\end{cases}
\end{equation}
where the RHS of the equation for $W^{+\beta}_{\lambda\mathscr{A}}$ is orthogonal to $\partial_{\beta} u_{\mathscr{A}}$.
\end{rem}
\begin{rem}
\label{rem:roleFplusestiFplus}
The components $(\textbf{F}^+_{\lambda\mathscr{A}})_{\mathscr{A}\in\mathbb{A}}$ are introduced specifically to absorb the contribution $\lambda^{-1/2}\partial u_\mathscr{A}\boldsymbol{\mathcal{E}}_\lambda^{evo}\textbf{B}_0$, which corresponds to high-amplitude, high-frequency terms oscillating in characteristic directions and arising from error-background interactions. The transport \eqref{eq:schemtpFplusestiFplus} is \textbf{coupled} with \eqref{eq:schemtpGplusoriginestiGplus} and \eqref{eq:schemwaveEevosestiEevo}.\\
  \end{rem}
We now detail the RHS terms and establish the corresponding estimates. 
  \begin{propal}
  \label{propal:RHSFplusestiFplus}
 The RHS terms
\begin{align}
 \lambda^{-1/2}\partial u_\mathscr{A}\boldsymbol{\mathcal{E}}_\lambda^{evo}\textbf{B}_0 ,\;\;
 \partial u_\mathscr{A}\Pi_{\kappa,-}(\boldsymbol{\mathcal{E}}_\lambda^{evo}\textbf{F}^+_{\lambda\mathscr{A}})\;\text{and }
\partial u_\mathscr{A}\textbf{F}_{\lambda\mathscr{A}}^+\textbf{B}_0 
  \end{align} 
    respectively originate from  $\mathscr{L}_1\left(\textbf{Z}_\lambda,\partial \textbf{B}_0,\textbf{B}'_0,\frac{\mathscr{U}_{\mathbb{A}}}{\lambda}\right)$,  $\textbf{Z}_\lambda\partial\textbf{Z}$ and $\mathscr{L}_2\left(\partial\textbf{Z}_\lambda,\textbf{B}_0,\frac{\mathscr{U}_{\mathbb{A}}}{\lambda}\right)$ in \eqref{eq:schemeqerrorerrterm}.
Under the Assumptions \ref{ass:evoass}, for any $\tau\in[0,T]$ and for any $\kappa>0$, the following estimates hold
\begin{align}
&I:=\sum_{k\leq 1}\int^{\tau}_0\lambda^k||\partial u_\mathscr{A}\textbf{F}_{\lambda\mathscr{A}}^+\textbf{B}_0 (t)||_{H^{k+1}}dt\leq C_0\frac{c_1}{c_2}e^{\tau c_2}\label{eq:estimIestiFplus},\\
&II:=\sum_{k\leq 1}\int^{\tau}_0\lambda^k||\lambda^{-1/2}\partial u_\mathscr{A}\boldsymbol{\mathcal{E}}_\lambda^{evo} (t)||_{H^{k+1}}dt\leq C_0\frac{c_1}{c_2}e^{\tau c_2}\label{eq:estimIIestiFplus},\\
&III:=\sum_{k\leq 1}\int^{\tau}_0\lambda^k||\partial u_\mathscr{A}\Pi_{\kappa,-}(\boldsymbol{\mathcal{E}}_\lambda^{evo}\textbf{F}^+_{\lambda\mathscr{A}}) (t)||_{H^{k+1}}dt\leq C_{0,1,2,\tau}\lambda^{0+}. \label{eq:estimIIIestiFplus}
 \end{align}
\end{propal}
\begin{proof} 
Estimates for I and II follow from direct applications of Sobolev inequalities of Section \ref{subsection:soboineq}, the product estimate \eqref{eq:holdsobAx} of Section \ref{subsection:prodesti} and the bounds of the Assumptions \ref{ass:evoass}. These terms fall into the case \ref{item:item3improveass}, since they are linear in the coupled error components. To prove estimate III, we emphasize that the projector $\Pi_{\kappa,-}$ is not required to produce additional smallness (see case \ref{item:item2improveass}). We begin with the $H^1$ estimate 
\begin{align*}
&\int^{\tau}_0||\partial u_\mathscr{A}\Pi_{\kappa,-}(\boldsymbol{\mathcal{E}}_\lambda^{evo}\textbf{F}^+_{\lambda\mathscr{A}}) (t)||_{H^{1}}dt\\
&\leq ||\partial u_\mathscr{A}||_{L^\infty( [0,\tau],W^{1,\infty})}\int^{\tau}_0||\boldsymbol{\mathcal{E}}_\lambda^{evo}\textbf{F}^+_{\lambda\mathscr{A}} (t)||_{L^2}+||\nabla\boldsymbol{\mathcal{E}}_\lambda^{evo}\textbf{F}^+_{\lambda\mathscr{A}} (t)||_{L^2}+||\boldsymbol{\mathcal{E}}_\lambda^{evo}\nabla\textbf{F}^+_{\lambda\mathscr{A}} (t)||_{L^2}dt\\
&\leq C_0(\underbrace{||\boldsymbol{\mathcal{E}}_\lambda^{evo}\textbf{F}^+_{\lambda\mathscr{A}} ||_{L^1([0,\tau],L^2)}}_\text{III.1}+\underbrace{||\nabla\boldsymbol{\mathcal{E}}_\lambda^{evo}\textbf{F}^+_{\lambda\mathscr{A}} ||_{L^1([0,\tau],L^2)}}_\text{III.2}+\underbrace{||\boldsymbol{\mathcal{E}}_\lambda^{evo}\nabla\textbf{F}^+_{\lambda\mathscr{A}} ||_{L^1([0,\tau],L^2)}}_\text{III.3}).
\end{align*}
For III.1, we use the product inequality \eqref{eq:holdsobAx} to obtain 
\begin{align*}
    III.1\leq C_{0,1,2,\tau}\lambda^{1/2}.
\end{align*}
Applying Lemma \ref{lem:lemma2Ax}, we obtain that, for $\theta\in(0,1/2)$ and for $p=2(1-\theta)$, $p'=\frac{2(1-\theta)}{1-2\theta}$, 
\begin{align*}
    III.2&\leq  C_0||\nabla\boldsymbol{\mathcal{E}}_\lambda^{evo}||_{L^p([0,\tau],L^2)}(||\textbf{F}^+_{\lambda\mathscr{A}}||_{L^{p'}([0,\tau],H^2)})^\theta(||\Box \textbf{F}^+_{\lambda\mathscr{A}}||_{L^1([0,\tau],L^2)}+||\textbf{F}^+_{\lambda\mathscr{A}}(0)||_{\dot{H}^1}+||\partial_t\textbf{F}^+_{\lambda\mathscr{A}}(0)||_{L^2})^{1-\theta}\\\\
    &\leq C_{0,1,2,\tau}\lambda^{1/2}\lambda^{-\theta}(|| \textbf{G}^+_{\lambda\mathscr{A}} ||_{L^1([0,\tau],L^2)}+C_0)^{1-\theta}\leq C_{0,1,2,\tau}\lambda^{1/2-\theta}\leq C_{0,1,2,\tau}\lambda^{0+},\\\\
III.3&\leq  C_0||\nabla\textbf{F}^+_{\lambda\mathscr{A}}||_{L^p([0,\tau],L^2)}(||\boldsymbol{\mathcal{E}}_\lambda^{evo}||_{L^{p'}([0,\tau],H^2)})^\theta(||\Box \boldsymbol{\mathcal{E}}_\lambda^{evo}||_{L^1([0,\tau],L^2)}+||\boldsymbol{\mathcal{E}}_\lambda^{evo}(0)||_{\dot{H}^1}+||\partial_t\boldsymbol{\mathcal{E}}_\lambda^{evo}(0)||_{L^2})^{1-\theta}\\\\
    &\leq C_{0,1,2,\tau}\lambda^{-\theta/2}(|| \Box \boldsymbol{\mathcal{E}}_\lambda^{evo}||_{L^1([0,\tau],L^2)}+C_0\lambda^{1/2})^{1-\theta}\leq C_{0,1,2,\tau}\lambda^{1/2-\theta}\leq C_{0,1,2,\tau}\lambda^{0+}.
\end{align*}
The bound $||\Box \boldsymbol{\mathcal{E}}_\lambda^{evo} ||_{L^1([0,\tau],L^2)}\leq C_{0,1,2,\tau}\lambda^{1/2}$ follows from the naive estimates of Proposition \ref{propal:naiveestiestiEevo}. Secondly, we treat the $H^2$ estimate
\begin{align*}
&\lambda\int^{\tau}_0||\partial u_\mathscr{A}\Pi_{\kappa,-}(\boldsymbol{\mathcal{E}}_\lambda^{evo}\textbf{F}^+_{\lambda\mathscr{A}}) (t)||_{H^{2}}dt\leq \lambda||\partial u_\mathscr{A}||_{L^\infty( [0,\tau],L^{\infty})}\int^{\tau}_0||\nabla\nabla(\boldsymbol{\mathcal{E}}_\lambda^{evo}\textbf{F}^+_{\lambda\mathscr{A}} )(t)||_{L^2}dt\\
&+\underbrace{\lambda\int^{\tau}_0||\partial u_\mathscr{A}(t)||_{W^{1,\infty}}||\nabla(\boldsymbol{\mathcal{E}}_\lambda^{evo}\textbf{F}^+_{\lambda\mathscr{A}} )(t)||_{L^2}+||\nabla\nabla\partial u_\mathscr{A}(t)||_{L^2}||\boldsymbol{\mathcal{E}}_\lambda^{evo}(t)||_{L^\infty}||\textbf{F}^+_{\lambda\mathscr{A}}(t)||_{L^\infty}+||\partial u_\mathscr{A}\boldsymbol{\mathcal{E}}_\lambda^{evo}\textbf{F}^+_{\lambda\mathscr{A}}(t)||_{H^1}dt}_\text{III.4}\\
&\leq III.4+C_0\underbrace{\lambda||\nabla\boldsymbol{\mathcal{E}}_\lambda^{evo}\nabla\textbf{F}^+_{\lambda\mathscr{A}} ||_{L^1([0,\tau],L^2)}}_\text{III.5}+C_0\underbrace{\lambda||\nabla\nabla\boldsymbol{\mathcal{E}}_\lambda^{evo}\textbf{F}^+_{\lambda\mathscr{A}} ||_{L^1([0,\tau],L^2)}}_\text{III.6}+C_0\underbrace{\lambda||\boldsymbol{\mathcal{E}}_\lambda^{evo}\nabla\nabla\textbf{F}^+_{\lambda\mathscr{A}}||_{L^1([0,\tau],L^2)}}_\text{III.7}.
\end{align*}
For III.4, we combine the previous calculations, the Sobolev embeddings of Proposition \ref{propal:embedsoboAx} and the product estimate of Proposition \ref{propal:holdsobAx} to 
get, for a small $0<\varepsilon<1/2$,
\begin{align*}
    III.4\leq C_{0,1,2,\tau}\lambda^{1/2-\varepsilon}.
\end{align*}
For III.6 and III.7, we apply Lemma \ref{lem:lemma2Ax} as for III.2 and III.3, losing one $\lambda$ in the operation. To estimate III.5, we apply Lemma \ref{lem:lemma1Ax} with $1/2>\nu>\varepsilon'$ and $p=\frac{1}{1-\nu}$, where $\varepsilon'$ is defined in the naive estimates of Proposition \ref{propal:naiveestiestiEevo}, to obtain
\begin{align*}
    III.5&\leq  \lambda C_0||\nabla \textbf{F}^+_{\lambda\mathscr{A}} ||_{L^p([0,\tau],H^{1/2-\nu})}(||\Box \nabla\boldsymbol{\mathcal{E}}_\lambda^{evo} ||_{L^1([0,\tau],L^2)}+||\nabla\boldsymbol{\mathcal{E}}_\lambda^{evo} (0)||_{\dot{H}^1}+||\partial_t\nabla\boldsymbol{\mathcal{E}}_\lambda^{evo}(0)||_{L^2})\\
    &\leq C_{0,1,2,\tau}\lambda^{1/2+\nu}(|| \Box\nabla \boldsymbol{\mathcal{E}}_\lambda^{evo} ||_{L^1([0,\tau],L^2)}+\lambda^{-1/2}C_0)\\
    &\leq C_{0,1,2,\tau}\lambda^{0+}.
\end{align*}
Collecting all contributions, we conclude
\begin{align*}
 III\leq C_0(III.1+III.2+III.3+III.4+III.5+III.6+III.7)\leq C_{0,1,2,\tau}\lambda^{0+}.
\end{align*}
\end{proof}
\subsubsection{Estimates for $\textbf{G}^+_{\lambda\mathscr{A}}$}
\label{subsubsection:estiGplus}
Throughout this Section, we denote with $C_0$ any constant that only depends on the background data and with $c_1$, $c_2$ and $C_{0,1,2,\tau}$ the constants of the bootstrap. See Notations \ref{nota:cst0Approx} and \ref{nota:cst12ass}. We also use the schematic Notation \ref{nota:schemnotaparametrixerrterm}, where background terms are represented by $\textbf{B}_0$ and $\textbf{B}_0'$. The constant $\kappa>0$ designates the constant of the projectors $\Pi_{\kappa,-}$ and $\Pi_{\kappa,+}$ introduced in Definition \ref{defi:projfreq}.\\\\
The transport equation for $\textbf{G}^+_{\lambda\mathscr{A}}$ reads 
\begin{align}
\label{eq:schemtpGplusestiGplus}
    &\mathscr{L}_\mathscr{A}\textbf{G}^+_{\lambda\mathscr{A}}=[(2\partial^\alpha u_\mathscr{A}\partial_\alpha+\Box u_\mathscr{A}),\Box]\textbf{F}^+_{\lambda\mathscr{A}}+\lambda^{-1/2}\partial u_\mathscr{A}\Box\boldsymbol{\mathcal{E}}_\lambda^{evo}\textbf{B}_0+\Box(\partial u_\mathscr{A}\Pi_{\kappa,-}(\boldsymbol{\mathcal{E}}_\lambda^{evo}\textbf{F}^+_{\lambda\mathscr{A}}))\\
    &+\partial u_\mathscr{A}\Pi_{\kappa,-}(\boldsymbol{\mathcal{E}}_\lambda^{evo}\textbf{G}^+_{\lambda\mathscr{A}})-\partial u_\mathscr{A}\Pi_{\kappa,-}(\boldsymbol{\mathcal{E}}_\lambda^{evo}\Box\textbf{F}^+_{\lambda\mathscr{A}})+\partial u_\mathscr{A}\textbf{G}^+_{\lambda\mathscr{A}}\textbf{B}_0+[\ldots]_{\textbf{G}^+_{\lambda\mathscr{A}}}.\nonumber
\end{align}

\begin{rem}
\label{rem:roleGplusestiGplus}
     The auxiliary component $\textbf{G}^+_{\lambda\mathscr{A}}$ is introduced to control $\Box\textbf{F}^+_{\lambda\mathscr{A}}$ and compensate for an apparent loss of derivatives as discussed in the introduction about the system \eqref{eq:schemsys1Idea}. The transport \eqref{eq:schemtpGplusestiGplus} is obtained by commuting the transport equation for $\textbf{F}^+_{\lambda\mathscr{A}}$ with the d'Alembertian operator, that is,
\begin{align}
\label{eq:schemtpGplusoriginestiGplus}
    &\mathscr{L}_\mathscr{A}\textbf{G}^+_{\lambda\mathscr{A}}=[(2\partial^\alpha u_\mathscr{A}\partial_\alpha+\Box u_\mathscr{A}),\Box]\textbf{F}^+_{\lambda\mathscr{A}}+\Box(\mathscr{L}_\mathscr{A}\textbf{F}_\mathscr{A}),
\end{align}
where every occurrence of $\Box \textbf{F}_{\lambda\mathscr{A}}^+$ is replaced by $\textbf{G}^+_{\lambda\mathscr{A}}$. The quantities $\Box \textbf{F}_{\lambda\mathscr{A}}^+$ and $\textbf{G}^+_{\lambda\mathscr{A}}$ share identical initial data, see Definition \ref{defi:admissibleKGMLcriteria}, and obey the same transport equation driven by the regular vector field $\textbf{d}u^\#_{\mathscr{A}}$. Consequently, they are equal. The equation \eqref{eq:schemtpGplusestiGplus} is coupled with \eqref{eq:schemtpFplusestiFplus} and \eqref{eq:schemwaveEevosestiEevo}. More explicitly, 
\begin{align}
 &[\ldots]_{\textbf{G}^+_{\lambda\mathscr{A}}}=\Box(\partial u_\mathscr{A}\textbf{F}_{\lambda\mathscr{A}}^+\textbf{B}_0 )- \partial u_\mathscr{A}\Box\textbf{F}_{\lambda\mathscr{A}}^+\textbf{B}_0+\Box(\lambda^{-1/2}\partial u_\mathscr{A}\boldsymbol{\mathcal{E}}_\lambda^{evo}\textbf{B}_0)-\lambda^{-1/2}\partial u_\mathscr{A}\Box\boldsymbol{\mathcal{E}}_\lambda^{evo}\textbf{B}_0.
  \end{align}
\end{rem}
We analyse the RHS terms and estimate them in the next Proposition.
\begin{propal}
\label{propal:RHSGplusestiGplus}
    Under the Assumptions \ref{ass:evoass}, there exists $\kappa>0$ sufficiently small such that for any $\tau\in[0,T]$, the terms $[\ldots]_{\textbf{G}^+_{\lambda\mathscr{A}}}$ satisfy
\begin{align}
&IV:=\sum_{k\leq 1}\int^{\tau}_0\lambda^k||[\ldots]_{\textbf{G}^+_{\lambda\mathscr{A}}}(t)||_{H^{k}}dt\leq C_0\frac{c_1}{c_2}e^{\tau c_2}\label{eq:estimIVestiGplus},
 \end{align}
the commutator satisfy
    \begin{align}
        &V:=\sum_{k\leq 1}\int^{\tau}_0\lambda^k||[(2\partial^\alpha u_\mathscr{A}\partial_\alpha+\Box u_\mathscr{A}),\Box]\textbf{F}_\mathscr{A}(t)||_{H^{k}}dt\leq C_0\frac{c_1}{c_2}e^{\tau c_2}+C_{0,1,2,\tau}\lambda^{0+}\label{eq:estimVestiGplus},
    \end{align}
    the remaining terms excluding the projector part satisfy
    \begin{align}
        &VI:=\sum_{k\leq 1}\int^{\tau}_0\lambda^k||\partial u_\mathscr{A}\textbf{G}^+_{\lambda\mathscr{A}}\textbf{B}_0(t)||_{H^{k}}dt\leq C_0\frac{c_1}{c_2}e^{\tau c_2}\label{eq:estimVIestiGplus},\\
        &VII:=\sum_{k\leq 1}\int^{\tau}_0\lambda^k||\lambda^{-1/2}\partial u_\mathscr{A}\Box\boldsymbol{\mathcal{E}}_\lambda^{evo}\textbf{B}_0(t)||_{H^{k}}dt\leq C_0\frac{c_1}{c_2}e^{\tau c_2}+C_0+C_{0,1,2,\tau}\lambda^{0+}\label{eq:estimVIIestiGplus},
    \end{align}
   and, finally, the terms involving the projector on low frequencies satisfy
    \begin{align}
        &VIII:=\sum_{k\leq 1}\int^{\tau}_0\lambda^k||\partial u_\mathscr{A}\Pi_{\kappa,-}(\boldsymbol{\mathcal{E}}_\lambda^{evo}\textbf{G}^+_{\lambda\mathscr{A}})(t)||_{H^{k}}dt\leq C_{0,1,2,\tau}\lambda^{0+}\label{eq:estimVIIIestiGplus},\\
          &IX:=\sum_{k\leq 1}\int^{\tau}_0\lambda^k||\Box(\partial u_\mathscr{A}\Pi_{\kappa,-}(\boldsymbol{\mathcal{E}}_\lambda^{evo}\textbf{F}^+_{\lambda\mathscr{A}}))(t)-\partial u_\mathscr{A}\Pi_{\kappa,-}(\boldsymbol{\mathcal{E}}_\lambda^{evo}\Box\textbf{F}^+_{\lambda\mathscr{A}})(t)||_{H^{k}}dt\leq C_{0,1,2,\tau}\lambda^{0+}.\label{eq:estimIXestiGplus}
    \end{align}
\end{propal}
\begin{proof}
The explicit expression of $[\ldots]_{\textbf{G}^+_{\lambda\mathscr{A}}}$ follows directly from Remark \ref{rem:roleGplusestiGplus}.
The estimate for IV is obtained using the bootstrap bounds from Assumption \ref{ass:evoass} and the higher regularity of the background. The corresponding terms are linear in $\boldsymbol{\mathcal{E}}_\lambda^{evo}$ and $\textbf{F}_{\lambda\mathscr{A}}^+$, involve at most one derivative, and therefore fall into case \ref{item:item3improveass}.\\
For V, using the commutator estimate \eqref{eq:esticommutcommutax}, we obtain 
\begin{align*}
     &\sum_{k\leq 1}\int^{\tau}_0\lambda^k ||[(2\partial^\alpha u\partial_\alpha+\Box u),\Box]\textbf{F}_\mathscr{A}(t)||_{H^k}dt\\
     &\leq C_0\sum_{k\leq 1}\int^{\tau}_0\lambda^k (||\textbf{G}^+_{\lambda\mathscr{A}}(t)||_{H^k}+||\partial (\partial^\alpha u_\mathscr{A}\partial_\alpha\textbf{F}_{\lambda\mathscr{A}}^+)(t)||_{H^k}+|| \textbf{F}_{\lambda\mathscr{A}}^+(t)||_{H^{k+1}})dt.
\end{align*}
The bounds on $\mathscr{L}_\mathscr{A}\partial\textbf{F}_{\lambda\mathscr{A}}^+$ provided by Proposition \ref{propal:RHSFplusestiFplus}, together with the Assumptions \ref{ass:evoass}, yield the desired inequality.\\
The estimate for VI follows directly from the Assumptions \ref{ass:evoass} and again corresponds to case \ref{item:item3improveass}. The term VII is controlled using the refined bound on $||\Box\boldsymbol{\mathcal{E}}_\lambda^{evo}||_{L^1([0,\tau],H^{k})}$ obtained in Section \ref{subsubsection:estiEevo}. For both VI and VII, the regularity of the background is exploited. In particular, the fact that one derivative of $\textbf{B}'_0$ is in $L^\infty$.\\  
We now estimate the low-frequency projector contributions VIII and IX. For $\kappa>0$ sufficiently small, the projector, together with the product inequality \eqref{eq:holdsobAx}, provides extra smallness (case \ref{item:item2improveass}).
  We begin with VIII:
\begin{align*}
VIII&\leq\sum_{k\leq 1}\lambda^k||\partial u_\mathscr{A}||_{L^\infty([0,\tau],W^{1,\infty})}||\Pi_{\kappa,-}(\textbf{G}^+_{\lambda\mathscr{A}}\boldsymbol{\mathcal{E}}_\lambda^{evo})||_{L^1([0,\tau],H^k)}\\&\leq C_0\sum_{k\leq 1}\lambda^k\lambda^{-\kappa/2}||\textbf{G}^+_{\lambda\mathscr{A}}\boldsymbol{\mathcal{E}}_\lambda^{evo}||_{L^1([0,\tau],H^{k-1/2})}\\
    &\leq C_0\sum_{k\leq 1}\lambda^{k-\kappa/2}||\boldsymbol{\mathcal{E}}_\lambda^{evo}||_{L^\infty([0,\tau],H^1)}||\textbf{G}^+_{\lambda\mathscr{A}}||_{L^1([0,\tau],H^k)}\\
    &\leq C_{0,1,2,\tau}\lambda^{1/2-\kappa/2}\leq C_{0,1,2,\tau}\lambda^{0+},
\end{align*}
 where we use the Assumptions \ref{ass:evoass} and the product inequality from Proposition \ref{propal:holdsobAx} for both $k=0$ and $k=1$.
Then, for IX, we first develop the term as follows:
 \begin{align*}
        IX&\leq\underbrace{\sum_{k\leq 1}\lambda^k||\Box\partial u_\mathscr{A}\Pi_{\kappa,-}(\boldsymbol{\mathcal{E}}_\lambda^{evo}\textbf{F}^+_{\lambda\mathscr{A}})||_{L^1([0,\tau],H^k)}}_\text{IX.1}+\underbrace{\sum_{k\leq 1}\lambda^k||\partial\partial u_\mathscr{A}\partial\Pi_{\kappa,-}(\boldsymbol{\mathcal{E}}_\lambda^{evo}\textbf{F}^+_{\lambda\mathscr{A}})||_{L^1([0,\tau],H^k)}}_\text{IX.2}\\
        &+\underbrace{\sum_{k\leq 1}\lambda^k||\partial u_\mathscr{A}\Pi_{\kappa,-}(\partial\boldsymbol{\mathcal{E}}_\lambda^{evo}\partial\textbf{F}^+_{\lambda\mathscr{A}})||_{L^1([0,\tau],H^k)}}_\text{IX.3}+\underbrace{\sum_{k\leq 1}\lambda^k||\partial u_\mathscr{A}\Pi_{\kappa,-}(\Box\boldsymbol{\mathcal{E}}_\lambda^{evo}\textbf{F}^+_{\lambda\mathscr{A}})||_{L^1([0,\tau],H^k)}}_\text{IX.4}.
    \end{align*}
        For IX.1 and IX.2 using the regularity of $u_\mathscr{A}$ and the smoothing effect of the projector, we obtain, for an arbitrarily small $\varepsilon>0$,
   \begin{align*}
        IX.1&\leq \sum_{k\leq 1}\lambda^k||\Box\partial u_\mathscr{A}||_{L^\infty([0,\tau],H^1)}||\Pi_{\kappa,-}(\boldsymbol{\mathcal{E}}_\lambda^{evo}\textbf{F}^+_{\lambda\mathscr{A}})||_{L^1([0,\tau],H^{k+1/2+\varepsilon})}\\
        &\leq C_0\sum_{k\leq 1}\lambda^{k-\kappa-\varepsilon\kappa}||\boldsymbol{\mathcal{E}}_\lambda^{evo}\textbf{F}^+_{\lambda\mathscr{A}}||_{L^1([0,\tau],H^{k-1/2})}\\
        &\leq C_0\sum_{k\leq 1}\lambda^{k-\kappa-\varepsilon\kappa}||\boldsymbol{\mathcal{E}}_\lambda^{evo}||_{L^2([0,\tau],H^{1})}||\textbf{F}^+_{\lambda\mathscr{A}}||_{L^2([0,\tau],H^{k})}\\
        &\leq C_{0,1,2,\tau}\lambda^{1/2-\kappa-\varepsilon\kappa},
    \end{align*}
    \begin{align*}
IX.2&\leq \sum_{k\leq 1}\lambda^k||\partial\partial u_\mathscr{A}||_{L^\infty([0,\tau],H^1)}||\Pi_{\kappa,-}(\partial(\boldsymbol{\mathcal{E}}_\lambda^{evo}\textbf{F}^+_{\lambda\mathscr{A}}))||_{L^1([0,\tau],H^{k+1/2+\varepsilon})}\\
&\leq C_0\sum_{k\leq 1}\lambda^{k-\kappa-\varepsilon\kappa}||\partial(\boldsymbol{\mathcal{E}}_\lambda^{evo}\textbf{F}^+_{\lambda\mathscr{A}})||_{L^1([0,\tau],H^{k-1/2})}\\
&\leq C_0\sum_{k\leq 1}\lambda^{k-\kappa-\varepsilon\kappa}(||\partial\boldsymbol{\mathcal{E}}_\lambda^{evo}||_{L^\infty([0,\tau],H^{k})}||\textbf{F}^+_{\lambda\mathscr{A}}||_{L^1([0,\tau],H^{1})}+||\partial\textbf{F}^+_{\lambda\mathscr{A}}||_{L^\infty([0,\tau],H^{k})}||\boldsymbol{\mathcal{E}}_\lambda^{evo}||_{L^1([0,\tau],H^{1})})\\
&\leq C_{0,1,2,\tau}\lambda^{1/2-\kappa-\varepsilon\kappa}.
    \end{align*}
For IX.3 and IX.4, using the smoothing effect of the projector and the naive estimates \ref{propal:naiveestiestiEevo} for $||\Box\boldsymbol{\mathcal{E}}_\lambda^{evo}||_{L^1([0,\tau],H^k)}$ with $k=0,1$, we obtain 
\begin{align*}
        IX.3&\leq\sum_{k\leq 1}\lambda^k||\partial u_\mathscr{A}(t)||_{L^\infty([0,\tau],W^{1,\infty})}||\Pi_{\kappa,-}(\partial\boldsymbol{\mathcal{E}}_\lambda^{evo}\partial\textbf{F}^+_{\lambda\mathscr{A}})||_{L^1([0,\tau],H^k)}\\
        &\leq C_0\sum_{k\leq 1}\lambda^{k-\kappa}||\partial\boldsymbol{\mathcal{E}}_\lambda^{evo}\partial\textbf{F}^+_{\lambda\mathscr{A}}||_{L^1([0,\tau],H^{k-1})}\\
        &\leq  C_0\sum_{k\leq 1}\lambda^{k-\kappa}||\partial\boldsymbol{\mathcal{E}}_\lambda^{evo}||_{L^2([0,\tau],H^{k})}||\partial\textbf{F}^+_{\lambda\mathscr{A}}(t)||_{L^2([0,\tau],H^{1/2})}\\
    &\leq C_0\lambda^{1/2-\kappa},
    \end{align*}
\begin{align*}
IX.4&\leq\sum_{k\leq 1}\lambda^k||\partial u_\mathscr{A}||_{L^\infty([0,\tau],W^{1,\infty})}||\Pi_{\kappa,-}(\Box\boldsymbol{\mathcal{E}}_\lambda^{evo}\textbf{F}^+_{\lambda\mathscr{A}})||_{L^1([0,\tau],H^k)}\\
&\leq C_0\sum_{k\leq 1}\lambda^k\lambda^{-\kappa/2}||\Box\boldsymbol{\mathcal{E}}_\lambda^{evo}\textbf{F}^+_{\lambda\mathscr{A}}||_{L^1([0,\tau],H^{k-1/2})}\\
    &\leq C_0\sum_{k\leq 1}\lambda^k\lambda^{-\kappa/2}||\textbf{F}^+_{\lambda\mathscr{A}}||_{L^\infty([0,\tau],H^1)}||\Box\boldsymbol{\mathcal{E}}_\lambda^{evo}||_{L^1([0,\tau],H^k)}\\
    &\leq C_{0,1,2,\tau}\lambda^{1/2-\kappa/2},
    \end{align*}
which is the desired result for $\kappa<1$. Combining all contributions yields
\begin{align*}
    IX\leq IX.1+IX.2+IX.3+IX.4\leq C_{0,1,2,\tau}\lambda^{1/2-\kappa-\varepsilon}\leq C_{0,1,2,\tau}\lambda^{0+},
\end{align*}
which ends the proof.
\end{proof}
\begin{rem}
\label{rem:projectornullformestiGplus}
    Among the terms appearing in the analysis, the only one that genuinely requires the low-frequency projector is the null form $\partial\boldsymbol{\mathcal{E}}_\lambda^{evo}\partial\textbf{F}^+_{\lambda\mathscr{A}}$.
\end{rem}

\subsubsection{Estimates for $\boldsymbol{\mathcal{E}}_\lambda^{evo}$}
\label{subsubsection:estiEevo}
Throughout this Section, we denote with $C_0$ any constant that only depends on the background data and with $c_1$, $c_2$ and $C_{0,1,2,\tau}$ the constants of the bootstrap. See Notations \ref{nota:cst0Approx} and \ref{nota:cst12ass}. We also use the schematic Notation \ref{nota:schemnotaparametrixerrterm}, where background terms are represented by $\textbf{B}_0$ and $\textbf{B}_0'$. The constant $\kappa>0$ designates the constant of the projectors $\Pi_{\kappa,-}$ and $\Pi_{\kappa,+}$ introduced in Definition \ref{defi:projfreq}.\\\\
The wave equation satisfied by $\boldsymbol{\mathcal{E}}_\lambda^{evo}$ is given by
\begin{equation}    
\label{eq:schemwaveEevosestiEevo}
\Box\boldsymbol{\mathcal{E}}_\lambda^{evo}=\boldsymbol{\mathcal{E}}_\lambda^{evo}\partial\boldsymbol{\mathcal{E}}_\lambda^{evo}+\sum_{\mathscr{A}\in\mathbb{A}}e^{i\frac{u_\mathscr{A}}{\lambda}}(i\partial u_\mathscr{A}\Pi_{\kappa,+}(\boldsymbol{\mathcal{E}}_\lambda^{evo}\textbf{F}^+_{\lambda\mathscr{A}})-\lambda^{1}\textbf{G}^+_{\lambda\mathscr{A}})+[\ldots]_{\boldsymbol{\mathcal{E}}_\lambda^{evo}}.\\
\end{equation}
\begin{rem}
\label{rem:roleEevosestiEevo}
    The component $\boldsymbol{\mathcal{E}}_\lambda^{evo}$ is introduced in order to absorb all remaining terms of the equation \eqref{eq:schemeqerrorerrterm}.  The equation \eqref{eq:schemwaveEevosestiEevo} is coupled with the transport equations \eqref{eq:schemtpFplusestiFplus} and \eqref{eq:schemtpGplusestiGplus}. More precisely, the term $[\ldots]_{\boldsymbol{\mathcal{E}}_\lambda^{evo}}$ reads  
    \begin{align*}
     [\ldots]_{\boldsymbol{\mathcal{E}}_\lambda^{evo}}&=\mathscr{H}\left(\textbf{Z}_\lambda,\textbf{B}_0,\frac{u}{\lambda}\right)+(\mathscr{L}_1\left(\textbf{Z}_\lambda,\partial \textbf{B}_0,\textbf{B}'_0,\frac{u}{\lambda}\right)-\underbrace{\lambda^{-1/2}\sum_{\mathscr{A}\in\mathbb{A}}i\partial u_\mathscr{A}\boldsymbol{\mathcal{E}}_\lambda^{evo}\textbf{B}_0 e^{i\frac{u_\mathscr{A}}{\lambda}}}_\text{absorbed by $\textbf{F}_{\lambda\mathscr{A}}^+$})\\
     &+(\mathscr{L}_2\left(\partial\textbf{Z}_\lambda,\textbf{B}_0,\frac{u}{\lambda}\right)-\underbrace{\sum_{\mathscr{A}\in\mathbb{A}}ie^{i\frac{u_\mathscr{A}}{\lambda}}
\partial u_\mathscr{A}\textbf{F}_{\lambda\mathscr{A}}^+\textbf{B}_0}_\text{absorbed by $\textbf{F}_{\lambda\mathscr{A}}^+$}-\underbrace{\sum_{(\mathscr{A},\mathscr{B})\in\mathscr{C}}ie^{i\frac{u_{\mathscr{A}}\pm u_{\mathscr{B}}}{\lambda}}(\partial u_{\mathscr{A}}\pm \partial u_{\mathscr{B}})\textbf{B}_0 \breve{\textbf{F}}_{\mathscr{A}\pm\mathscr{B}}^+}_\text{absorbed by $\breve{\textbf{F}}_{\mathscr{A}\pm\mathscr{B}}^+$})\\
&+(\textbf{Z}_\lambda\partial\textbf{Z}_\lambda-\underbrace{\boldsymbol{\mathcal{E}}_\lambda^{evo}\partial\boldsymbol{\mathcal{E}}_\lambda^{evo}}_\text{explicit in \eqref{eq:schemwaveEevosestiEevo}}-\underbrace{\sum_{\mathscr{A}\in\mathbb{A}}ie^{i\frac{u_\mathscr{A}}{\lambda}} \partial u_\mathscr{A}\Pi_{\kappa,-}(\boldsymbol{\mathcal{E}}_\lambda^{evo}\textbf{F}^+_{\lambda\mathscr{A}})}_\text{absorbed by $\breve{\textbf{F}}_{\mathscr{A}\pm\mathscr{B}}^+$}-\underbrace{\sum_{\mathscr{A}\in\mathbb{A}}ie^{i\frac{u_\mathscr{A}}{\lambda}} \partial u_\mathscr{A}\Pi_{\kappa,+}(\boldsymbol{\mathcal{E}}_\lambda^{evo}\textbf{F}^+_{\lambda\mathscr{A}}}_\text{explicit in \eqref{eq:schemwaveEevosestiEevo}}))\\
&-\underbrace{\sum_{(\mathscr{A},\mathscr{B})\in\mathscr{C}}\lambda^1 e^{i\frac{u_\mathscr{A}\pm u_\mathscr{B}}{\lambda}}\Box\breve{\textbf{F}}^+_{\mathscr{A}\pm\mathscr{B}}}_\text{resonant interactions remainder}-\underbrace{(\Box\boldsymbol{\mathcal{E}}_\lambda^{ell}-\sum_{(\mathscr{A},\mathscr{B})\in\mathscr{S}}e^{i\frac{u_\mathscr{B}\pm u_\mathscr{B}}{\lambda}}\textbf{S}_{(\mathscr{A}\pm\mathscr{B})}(\textbf{B}_0 '))}_\text{non-resonant interactions remainder}+\underbrace{\lambda^{1/2}\boldsymbol{\Xi}_\lambda}_\text{background remainder }.
    \end{align*} 
\end{rem}
In the next Propositions, the main estimates for the RHS of \eqref{eq:schemwaveEevosestiEevo} are established.
\begin{propal}
\label{propal:RHSEevosestiEevo}
    The term $[\ldots]_{\boldsymbol{\mathcal{E}}_\lambda^{evo}}$ is decomposed and estimated term by term. Under the Assumptions \ref{ass:evoass} and for any $\tau\in[0,T]$, the following bounds hold 
     \begin{align}
     &X:=\sum_{k\leq 1}\int^{\tau}_0\lambda^k||\mathscr{H}\left(\textbf{Z}_\lambda,\textbf{B}_0,\frac{u}{\lambda}\right)(t)||_{H^k}dt\leq \lambda^{1/2+}C_{0,1,2,\tau}\label{eq:estimXestiEevo},\\
&XI:=\sum_{k\leq 1}\int^{\tau}_0\lambda^k||\left(\mathscr{L}_1\left(\textbf{Z}_\lambda,\partial \textbf{B}_0,\textbf{B}'_0,\frac{u}{\lambda}\right)-\lambda^{-1/2}\sum_{\mathscr{A}\in\mathbb{A}}i\partial u_\mathscr{A}\boldsymbol{\mathcal{E}}_\lambda^{evo}\textbf{B}_0 e^{i\frac{u_\mathscr{A}}{\lambda}}\right)(t)||_{H^k}dt\leq \lambda^{1/2}C_0\frac{c_1}{c_2}e^{\tau c_2},\label{eq:estimXIestiEevo}\\
&XII:=\sum_{k\leq 1}\int^{\tau}_0\lambda^k||\left(\mathscr{L}_2\left(\partial\textbf{Z}_\lambda,\textbf{B}_0,\frac{u}{\lambda}\right)-A-B\right)(t)||_{H^k}dt\leq
\lambda^{1/2}C_0\frac{c_1}{c_2}e^{\tau 
c_2},\label{eq:estimXIIestiEevo}\\
&\text{for}\;\;A=\sum_{\mathscr{A}\in\mathbb{A}}ie^{i\frac{u_\mathscr{A}}{\lambda}}
\partial u_\mathscr{A}\textbf{F}_{\lambda\mathscr{A}}^+\textbf{B}_0\;\;\text{and}\;\;B=\sum_{(\mathscr{A},\mathscr{B})\in\mathscr{C}}ie^{i\frac{u_{\mathscr{A}}\pm u_{\mathscr{B}}}{\lambda}}(\partial u_{\mathscr{A}}\pm \partial u_{\mathscr{B}})\textbf{B}_0 \breve{\textbf{F}}_{\mathscr{A}\pm\mathscr{B}}^+\nonumber,\\
&XIII:=\sum_{k\leq 1}\int^{\tau}_0\lambda^k||(\textbf{Z}_\lambda\partial\textbf{Z}_\lambda-C-D-E)(t)||_{H^k}dt\leq \lambda^{1/2+}C_{0,1,2,\tau}+\lambda^{1/2}C_0\frac{c_1}{c_2}e^{\tau c_2}\label{eq:estimXIIIestiEevo},\\
&\text{for}\;\;C=\boldsymbol{\mathcal{E}}_\lambda^{evo}\partial\boldsymbol{\mathcal{E}}_\lambda^{evo},\;\;D=\sum_{\mathscr{A}\in\mathbb{A}}ie^{i\frac{u_\mathscr{A}}{\lambda}} \partial u_\mathscr{A}\Pi_{\kappa,-}(\boldsymbol{\mathcal{E}}_\lambda^{evo}\textbf{F}^+_{\lambda\mathscr{A}})\;\;\text{and}\;\;E=\sum_{\mathscr{A}\in\mathbb{A}}ie^{i\frac{u_\mathscr{A}}{\lambda}} \partial u_\mathscr{A}\Pi_{\kappa,+}(\boldsymbol{\mathcal{E}}_\lambda^{evo}\textbf{F}^+_{\lambda\mathscr{A}})\nonumber,\\
&XIV:=\sum_{k\leq 1}\int^{\tau}_0 \lambda^k||\sum_{(\mathscr{A},\mathscr{B})\in\mathscr{C}}\lambda^1 e^{i\frac{u_\mathscr{A}\pm u_\mathscr{B}}{\lambda}}\Box\breve{\textbf{F}}^+_{\mathscr{A}\pm\mathscr{B}}(t)||_{H^{k}}dt\leq \lambda^{1} C_0,\label{eq:estimXIVestiEevo}\\
&XV:=\sum_{k\leq 1}\int^{\tau}_0 \lambda^k||\left(\Box\boldsymbol{\mathcal{E}}_\lambda^{ell}-\sum_{(\mathscr{A},\mathscr{B})\in\mathscr{S}}e^{i\frac{u_\mathscr{B}\pm u_\mathscr{B}}{\lambda}}\textbf{S}_{(\mathscr{A}\pm\mathscr{B})}(\textbf{B}_0 ')\right)(t)||_{H^{k}}dt\leq \lambda^{1/2} C_0,\label{eq:estimXVestiEevo}\\
&XVI:=\sum_{k\leq 1}\int^{\tau}_0\lambda||\lambda^{1/2}\boldsymbol{\Xi}_\lambda(t)||_{H^k}dt\leq \lambda^{1/2}C_0,\label{eq:estimXVIestiEevo}
 \end{align}
where the functions $\mathscr{H}$, $\mathscr{L}_1$ and $\mathscr{L}_2$ are given in \eqref{eq:schemeqerrorerrterm}. Overall, we have 
\begin{align}
    \sum_{k\leq 1}\int^{\tau}_0\lambda^k||[\ldots]_{\boldsymbol{\mathcal{E}}_\lambda^{evo}}(t)||_{H^k}dt\leq C_0\lambda^{1/2}+\lambda^{1/2+}C_{0,1,2,\tau}+\lambda^{1/2}C_0\frac{c_1}{c_2}e^{\tau c_2}\label{eq:estimoverallestiEevo}.
\end{align}
\end{propal}
\begin{proof}
The quantity X is at least quadratic and involves no derivatives. Hence, the product estimate of Proposition \ref{propal:holdsobAx} yields additional smallness in $\lambda$, corresponding to case \ref{item:item2improveass}. The terms XI and XII are linear in the coupled error components, which correspond to case \ref{item:item3improveass}. Their estimates follow directly from Assumptions \ref{ass:evoass}, together with the regularity of the background.\\
For XIII, since the potentially dangerous contributions are already isolated, the remaining terms are harmless. In particular, the interaction between $\boldsymbol{\mathcal{E}}_\lambda^{evo}$ and $\tilde{\textbf{F}}^+_{\mathscr{A}\pm\mathscr{B}}$  does not involve any additional power of $\lambda$. However, since $\tilde{\textbf{F}}^+_{\mathscr{A}\pm\mathscr{B}}$ is treated as a background quantity (i.e., a decoupled error component), the resulting term is effectively linear with respect to $\boldsymbol{\mathcal{E}}_\lambda^{evo}$ (a coupled error component). This corresponds again to case \ref{item:item3improveass}. The estimate for XIV follows directly from Assumption \ref{ass:evoass} using again the decoupled nature of $\tilde{\textbf{F}}^+_{\mathscr{A}\pm\mathscr{B}}$ (case \ref{item:item1improveass}). For XV, we invoke Proposition \ref{propal:propertiesEellinfoEell}. The bound for XVI follows immediately from Definition \ref{defi:almostapproxresults} of $\boldsymbol{\Xi}_\lambda$ and corresponds to case \ref{item:item1improveass}.\\
Summing all previous bounds yields \eqref{eq:estimoverallestiEevo}.
\end{proof}

\begin{propal}
\label{propal:estiGplusestiEevo}
Under the Assumptions \ref{ass:evoass} and for any $\tau\in[0,T]$, the following inequality holds
    \begin{align}
        XVII:=\sum_{k\leq 1}\int^{\tau}_0\lambda^k||\sum_{\mathscr{A}\in\mathbb{A}}e^{i\frac{u_\mathscr{A}}{\lambda}}\lambda^{1}\textbf{G}^+_{\lambda\mathscr{A}}(t)||_{H^k}dt\leq \lambda^{1/2+}C_{0,1,2,\tau}\label{eq:estimXVIIestiEevo}.
    \end{align}
\end{propal}
\begin{proof}
    The estimate is obtained by direct computations using the Assumptions \ref{ass:evoass}, since the term is linear in $\textbf{G}^+_{\lambda\mathscr{A}}$, corresponding to case \ref{item:item3improveass}.
\end{proof} 
\begin{propal}
\label{propal:estiprojtermestiEevo}
Under the Assumptions \ref{ass:evoass}, for any $\kappa>0$ and for any $\tau\in[0,T]$,  the following inequality holds
    \begin{align}
         &XVIII:=\sum_{k\leq 1}\int^{\tau}_0\lambda^k||\sum_{\mathscr{A}\in\mathbb{A}}e^{i\frac{u_\mathscr{A}}{\lambda}}i\partial u_\mathscr{A}\Pi_{\kappa,+}(\boldsymbol{\mathcal{E}}_\lambda^{evo}\textbf{F}^+_{\lambda\mathscr{A}})(t)||_{H^k}dt\leq \lambda^{1/2+}C_{0,1,2,\tau}.\label{eq:estimXVIIIestiEevo}
    \end{align}
\end{propal}
\begin{proof}
    We begin with the $L^2$ norm: 
    \begin{align*}
XVIII.1&:=\int^{\tau}_0{||\partial u_\mathscr{A}\Pi_{\kappa,+}(\boldsymbol{\mathcal{E}}_\lambda^{evo} \textbf{F}^+_{\lambda\mathscr{A}})(t)||_{L^2}dt}\\
&\leq ||\partial u_\mathscr{A}||_{L^\infty([0,\tau],W^{1,\infty})}\int^{\tau}_0{||\Pi_{\kappa,+}(\boldsymbol{\mathcal{E}}_\lambda^{evo} \textbf{F}^+_{\lambda\mathscr{A}})(t)||_{L^2}dt}\\
&\leq C_0\int^{\tau}_0{(\int_{\mathbb{R}^3}|\widehat{\Pi_{\kappa,+}(\boldsymbol{\mathcal{E}}_\lambda^{evo} \textbf{F}^+_{\lambda\mathscr{A}})(t) }|^2(\xi)d\xi)^{1/2}dt}\\
&\leq C_0\int^{\tau}_0{(\int_{\mathbb{R}^3}|\widehat{\Pi_{\kappa,+}(\boldsymbol{\mathcal{E}}_\lambda^{evo} \textbf{F}^+_{\lambda\mathscr{A}})(t) }|^2(\xi)\frac{|\xi|}{|\xi|}d\xi)^{1/2}dt}\\
&\leq C_0\lambda^{\kappa/2}\int^{\tau}_{0}{||(\boldsymbol{\mathcal{E}}_\lambda^{evo} \textbf{F}^+_{\lambda\mathscr{A}})(t)||_{H^{1/2}}}dt\\
&\leq C_0\lambda^{\kappa/2}\int^{\tau}_{0}{||\boldsymbol{\mathcal{E}}_\lambda^{evo}(t)||_{H^{1}}||\textbf{F}^+_{\lambda\mathscr{A}}(t)||_{H^{1}}}dt\leq C_{0,1,2,\tau}\lambda^{1/2+},
\end{align*}
where we use the Plancherel Theorem, the projector on high frequencies $\Pi_{\kappa,+}$, the product estimate of Proposition \ref{propal:holdsobAx} and the Assumptions \ref{ass:evoass} on $\boldsymbol{\mathcal{E}}_\lambda^{evo}$ and $\textbf{F}^+_{\lambda\mathscr{A}}$. We observe that the extra factor $\lambda^{\kappa/2}$ allows us to recover case \ref{item:item2improveass}. On the other hand, for the $H^1$ norm,  the projector is not required to get the extra smallness.  Using the product estimate of Proposition \ref{propal:holdsobAx} and the interpolation result of Proposition \ref{propal:interpolsobAx}, we obtain
\begin{align*}
  XVIII.2&:=\lambda^1\int^{\tau}_0{||\partial u_\mathscr{A}\Pi_{\kappa,+}(\boldsymbol{\mathcal{E}}_\lambda^{evo} \textbf{F}^+_{\lambda\mathscr{A}})(t)||_{H^1}dt}\\
  &\leq \lambda^1||\partial u_\mathscr{A}||_{L^\infty([0,\tau],W^{1,\infty})}\int^{\tau}_0{||\Pi_{\kappa,+}(\boldsymbol{\mathcal{E}}_\lambda^{evo} \textbf{F}^+_{\lambda\mathscr{A}})(t)||_{H^1}dt}\\
&\leq C_0\lambda^1\int^{\tau}_0{||(\boldsymbol{\mathcal{E}}_\lambda^{evo} \textbf{F}^+_{\lambda\mathscr{A}})(t)||_{H^1}dt}\\
&\leq C_0\lambda^1\int^{\tau}_{0}{||\boldsymbol{\mathcal{E}}_\lambda^{evo}(t)||_{H^{5/4}}||\textbf{F}^+_{\lambda\mathscr{A}}(t)||_{H^{5/4}}}dt\\
&\leq C_{0,1,2,\tau}\lambda^1\lambda^{1/4}\lambda^{-1/4}\leq C_{0,1,2,\tau}\lambda^{1/2+}.
\end{align*}
We conclude that $XVIII\leq XVIII.1+XVIII.2\leq C_{0,1,2,\tau}\lambda^{1/2+}$, as desired.
\end{proof}

It remains to estimate the quadratic term $\boldsymbol{\mathcal{E}}_\lambda^{evo}\partial\boldsymbol{\mathcal{E}}_\lambda^{evo}$.
\begin{propal}
\label{propal:estiquadtermestiEevo}
Under the Assumptions \ref{ass:evoass} and for any $\tau\in[0,T]$, we have 
    \begin{align}
         &XIX:=\sum_{k\leq 1}\int^{\tau}_0\lambda^k||\boldsymbol{\mathcal{E}}_\lambda^{evo}\partial\boldsymbol{\mathcal{E}}_\lambda^{evo}(t)||_{H^k}dt\leq \lambda^{1/2+}C_{0,1,2,\tau}\label{eq:estimXIVestiEevo}.
    \end{align}
\end{propal}
\begin{rem}
\label{rem:badestiestiEevo}
   Using only product estimates and Sobolev interpolation (Propositions \ref{propal:holdsobAx} and \ref{propal:interpolsobAx}), one misses a factor $\lambda^{\varepsilon}$, for $\varepsilon>0$, which prevents closing the bootstrap directly (case \ref{item:item2improveass}). Indeed, for $0<\alpha<3/2$, $0<\beta\leq 1$ and $\alpha+\beta=3/2$ we have 
   \begin{align*}       
   ||\boldsymbol{\mathcal{E}}_\lambda^{evo}\partial\boldsymbol{\mathcal{E}}_\lambda^{evo}||_{L^2}\leq ||\boldsymbol{\mathcal{E}}_\lambda^{evo}||_{H^\alpha}||\partial\boldsymbol{\mathcal{E}}_\lambda^{evo}||_{H^\beta},
   \end{align*}
   which leads to the cases
        \begin{align*}
&1/2\leq\alpha\leq1:\;||\boldsymbol{\mathcal{E}}_\lambda^{evo}\partial\boldsymbol{\mathcal{E}}_\lambda^{evo}(\tau)||_{L^2}\leq \lambda^{-1/2+\alpha}(c_1)^2e^{2\tau c_2}, 
&1<\alpha<3/2:\;||\boldsymbol{\mathcal{E}}_\lambda^{evo}\partial\boldsymbol{\mathcal{E}}_\lambda^{evo}(\tau)||_{L^2}\leq \lambda^{1/2}(c_1)^2e^{2\tau c_2}, 
   \end{align*}
  by interpolation. Moreover, for the $\dot{H}^1$ norm, with the same techniques,
 \begin{align*}
&||\nabla\boldsymbol{\mathcal{E}}_\lambda^{evo}\partial\boldsymbol{\mathcal{E}}_\lambda^{evo}(\tau)||_{L^2}\leq \lambda^{-1/2}C_{0,1,2,\tau}, &||\boldsymbol{\mathcal{E}}_\lambda^{evo}\nabla\partial\boldsymbol{\mathcal{E}}_\lambda^{evo}(\tau)||_{L^2}\leq \lambda^{-1/2-\varepsilon}C_{0,1,2,\tau}, 
   \end{align*}   
where a small $\varepsilon$ is lost due to  the control of the $L^\infty$ norm of $\boldsymbol{\mathcal{E}}_\lambda^{evo}$. The previous estimates yield
 \begin{align*}
     XIX\leq C_{0,1,2,\tau}\lambda^{1/2-\varepsilon},
 \end{align*}
  at best. In particular, these estimates ignore the time integration and dispersive properties of the wave operator.
\end{rem}
\begin{propal}
\label{propal:naiveestiestiEevo}
Let the Assumptions \ref{ass:evoass} hold. For $0<\varepsilon'$ arbitrarily small, the following naive bounds hold:
    \begin{align}
        &||\Box \boldsymbol{\mathcal{E}}_\lambda^{evo}||_{L^1([0,\tau],L^2)}\leq C_{0,1,2,\tau}\lambda^{1/2} +C_0\lambda^{1/2},
        &||\Box\nabla\boldsymbol{\mathcal{E}}_\lambda^{evo}||_{L^1([0,\tau],L^2)}\leq C_{0,1,2,\tau}\lambda^{-1/2-\varepsilon'}+C_0\lambda^{-1/2}.
    \end{align}
\end{propal}
\begin{proof}
The result follows by combining Propositions \ref{propal:RHSEevosestiEevo}, \ref{propal:estiGplusestiEevo}, and \ref{propal:estiprojtermestiEevo}, together with the estimates of Remark \ref{rem:badestiestiEevo}. The factor $C_0\frac{c_1}{c_2}e^{\tau c_2}$ is absorbed into $C_{0,1,2,\tau}$.
\end{proof}
We now have enough material to prove Proposition \ref{propal:estiquadtermestiEevo}.
\begin{proof}[Proof of Proposition \ref{propal:estiquadtermestiEevo}] 
Define 
\begin{align*}
&XIX.1:=||\boldsymbol{\mathcal{E}}_\lambda^{evo}\partial\boldsymbol{\mathcal{E}}_\lambda^{evo}||_{L^1([0,\tau],L^2)}, &XIX.2:=\lambda||\boldsymbol{\mathcal{E}}_\lambda^{evo}\nabla\partial\boldsymbol{\mathcal{E}}_\lambda^{evo}||_{L^1([0,\tau],L^2)}, &\;\;\;XIX.3:=\lambda||\nabla\boldsymbol{\mathcal{E}}_\lambda^{evo}\partial\boldsymbol{\mathcal{E}}_\lambda^{evo}||_{L^1([0,\tau],L^2)}.
\end{align*}
Applying Lemma \ref{lem:lemma2Ax}, with $\theta\in(0,1/2)$ and exponents $p=2(1-\theta)$ and $p'=\frac{2(1-\theta)}{1-2\theta}$, we obtain 
\begin{align*} 
XIX.2&\leq \lambda||\nabla\partial\boldsymbol{\mathcal{E}}_\lambda^{evo}||_{L^p([0,\tau],L^2)}(||\boldsymbol{\mathcal{E}}_\lambda^{evo}||_{L^{p'}([0,\tau],H^2)})^\theta(||\Box \boldsymbol{\mathcal{E}}_\lambda^{evo}||_{L^1([0,\tau],L^2)}+||\boldsymbol{\mathcal{E}}_\lambda^{evo}(0)||_{\dot{H}^1}+||\partial_t\boldsymbol{\mathcal{E}}_\lambda^{evo}(0)||_{L^2})^{(1-\theta)}\\
&\leq C_{0,1,2,\tau}\lambda^{1/2}\lambda^{-\theta/2}(||\Box \boldsymbol{\mathcal{E}}_\lambda^{evo}||_{L^1([0,\tau],L^2)}+c_1\lambda^{1/2})^{(1-\theta)}\\
&\leq  C_{0,1,2,\tau}\lambda^{1/2+},
\end{align*}
where the last inequality follows from the naive estimates of Proposition \ref{propal:naiveestiestiEevo}.\\
The same argument directly apply to $XIX.1$ to obtain 
\begin{align*} 
&XIX.1\leq  C_{0,1,2,\tau}\lambda^{1/2+}.
\end{align*}
To estimate $XIX.3$, we apply Lemma \ref{lem:lemma1Ax} with $1/2>\nu>\varepsilon'$ and $p=\frac{1}{1-\nu}$, where $\varepsilon'$ is defined in Proposition \ref{propal:naiveestiestiEevo}, to get
\begin{align*} 
XIX.3&\leq C_0 \lambda ||\partial\boldsymbol{\mathcal{E}}_\lambda^{evo}||_{L^p([0,\tau],H^{1/2-\nu})}(||\Box \nabla\boldsymbol{\mathcal{E}}_\lambda^{evo}||_{L^1([0,\tau],L^2)}+||\nabla\boldsymbol{\mathcal{E}}_\lambda^{evo}(0)||_{\dot{H}^1}+||\partial_t\nabla\boldsymbol{\mathcal{E}}_\lambda^{evo}(0)||_{L^2})\\
&\leq \lambda^{1+\nu}(||\Box \nabla\boldsymbol{\mathcal{E}}_\lambda^{evo}||_{L^1([0,\tau],L^2)}+c_1\lambda^{-1/2})\\
&\leq C_{0,1,2,\tau}\lambda^{1/2+}.
\end{align*}\\ 
Combining all contributions yields
\begin{align*}
    XIX\leq C_0 XIX.1+XIX.2+XIX.3\leq  C_{0,1,2,\tau}\lambda^{1/2+}.
\end{align*}
This ends the proof of Proposition \ref{propal:estiquadtermestiEevo}.
\end{proof}

\subsubsection{Bootstrap}
\label{subsubsection:bootstrap}
    The next Proposition is a direct consequence of the results established in the previous Sections. 
\begin{propal}
\label{propal:estifullKGMLbootstrap}
For a given admissible background initial data set of Definition \ref{defi:admissiblebginitansatz}, let $\textbf{B}'_0$ be given by Proposition \ref{propal:firstapproxApprox} with size $c_0>0$. Let $\lambda>0$ and let $(\textbf{f}^+_{\lambda\mathscr{A}},\textbf{g}^+_{\lambda\mathscr{A}},\breve{\textbf{f}}^+_{\mathscr{A}\pm\mathscr{B}},\boldsymbol{\epsilon}^{evo}_\lambda,\dot{\boldsymbol{\epsilon}}^{evo}_\lambda,\boldsymbol{\epsilon}^{ell}_\lambda,\dot{\boldsymbol{\epsilon}}^{ell}_\lambda)$ be initial data \textbf{admissible for KGML} in the sense of Definition \ref{defi:admissibleKGMLcriteria}, with the associated $-3/2<\delta<-1/2$, and satisfying the \textbf{required smallness} from Definition \ref{defi:reqsmallnesscriteria}. Then, under the Assumptions \ref{ass:evoass}, there exists $\kappa_0>0$, determining the projectors from Definition \ref{defi:projfreq}, such that for any $\tau\in[0,t_\lambda]$, the solution $(\textbf{F}^+_{\lambda\mathscr{A}},\textbf{G}^+_{\lambda\mathscr{A}},\boldsymbol{\mathcal{E}}_\lambda^{evo})$ to \eqref{eq:schemfullsystemevosys} given by Proposition \ref{propal:wellposfullsystWP} satisfies the following estimates\footnote{We distinguish between $\underline{\boldsymbol{\mathcal{E}}_\lambda}^{evo}$ and $(E^{evo}_\lambda)^0$ as in Notation \ref{nota:notaepsilonunderass}.}
    \begin{align*}
    &\sum_{k\leq 1}\lambda^k( ||\underline{\boldsymbol{\mathcal{E}}^{evo}_\lambda}(\tau)||_{H^{k+1}}+||\partial_t\underline{\boldsymbol{\mathcal{E}}^{evo}_\lambda}(\tau)||_{H^{k}})+\lambda||\partial^2_{tt}\underline{\boldsymbol{\mathcal{E}}_\lambda^{evo}}(\tau)||_{L^2}\leq C_{0,1,2,\tau}\lambda^{1/2+}+\lambda^{1/2}C_0\frac{c_1}{c_2}e^{\tau c_2}+\lambda^{1/2}C_0,\\
 &\sum_{k\leq 1}\lambda^k( ||(E^{evo}_\lambda)^0(\tau)||_{H^{k+1}_{\delta}}+||\partial_t(E^{evo}_\lambda)^0(\tau)||_{
 H^{k}_{\delta+1}})+\lambda||\partial^2_{tt}(E^{evo}_\lambda)^0(\tau)||_{L^2_{\delta+2}}\leq C_{0,1,2,\tau}\lambda^{1/2+}+\lambda^{1/2}C_0\frac{c_1}{c_2}e^{\tau c_2}+\lambda^{1/2}C_0,\\
    &\max_{\mathscr{A}\in\mathbb{A}}\sum_{k\leq 1}\lambda^k( ||\textbf{F}_{\lambda\mathscr{A}}^+(\tau)||_{H^{k+1}}+||\partial_t\textbf{F}_{\lambda\mathscr{A}}^+(\tau)||_{H^{k}})+\lambda||\partial^2_{tt}\textbf{F}_{\lambda\mathscr{A}}^+(\tau)||_{L^2}\leq C_{0,1,2,\tau}\lambda^{0+}+C_0\frac{c_1}{c_2}e^{\tau c_2}+C_0,\\
   &\max_{\mathscr{A}\in\mathbb{A}}\sum_{k\leq 1}( \lambda^k||\textbf{G}^+_{\lambda\mathscr{A}}(\tau)||_{H^{k}})+\lambda||\partial_t\textbf{G}^+_{\lambda\mathscr{A}}(\tau)||_{L^2}\leq C_{0,1,2,\tau}\lambda^{0+}+C_0\frac{c_1}{c_2}e^{\tau c_2}+C_0,
\end{align*}
for constants defined in Notations \ref{nota:cst0Approx} and \ref{nota:cst12ass}.
\end{propal}
\begin{proof}
Firstly, we combine all the estimates obtained in Sections \ref{subsubsection:estiFplus}, \ref{subsubsection:estiGplus} (in particular, Proposition \ref{propal:RHSGplusestiGplus} determines the value of $\kappa_0$) and \ref{subsubsection:estiEevo}. Next, we rely on Propositions \ref{propal:estienergytpAX} and\footnote{The linear term in the RHS of \eqref{eq:estienergyweightwaveAx} and \eqref{eq:estienergytp1AX} is not problematic and is treated in the same way as the other linear terms (case \ref{item:item3improveass}).} \ref{propal:energywaveAx}, together with Proposition \ref{propal:energywaveplusAx} and the smallness of the initial data in the sense of Definition \ref{defi:reqsmallnesscriteria}, to control the spatial Sobolev norms.
Finally, we use Propositions \ref{propal:energywavederivAx} and \ref{propal:estienergytpderivtempsAX} to control the time derivatives. The smallness estimates then follow immediately. 
\end{proof}

\begin{lem}
\label{lem:improvmentbootstrap}
For $\varepsilon>0$, $c_0>0$, $\frac{1}{\eta_0}>0$, $N\in\mathbb{N}$, $T>0$ and any $C_{0}=C_{0}(c_0,\frac{1}{\eta_0},N,T)$ and $C_{0,1,2,\tau}(\cdot,\cdot,\cdot)=C_{0,1,2,\tau}(c_0,\frac{1}{\eta_0},N,T,\cdot,\cdot,\cdot)$ from Notations \ref{nota:cst0Approx} and \ref{nota:cst12ass}, there exist $\lambda_0>0$, $c_1$ and $c_2$, such that for any $0<\lambda<\lambda_0$ and for any $0\leq\tau\leq T$ we have 
\begin{align*}
C_0\frac{c_1}{c_2}e^{\tau c_2}+C_0+C_{0,1,2,\tau}(c_1,c_2,\tau)\lambda^{\varepsilon}\leq \frac{1}{2}c_1e^{\tau c_2}.
\end{align*}
\end{lem}
\begin{proof}
We first choose $c_2$ sufficiently large so that $\frac{C_0}{c_2}\leq \frac{1}{6}$. Next, we choose a sufficiently large $c_1$ so that $C_0\leq \frac{1}{6}c_1$. Finally, we choose $\lambda_0$ such that $C_{0,1,2,\tau}(c_1,c_2,T)\lambda^{\varepsilon}\leq \frac{1}{6}c_1$.
\end{proof}
\subsection{Exact solutions to KGML}
We reformulate more precisely and prove the rough Proposition \ref{propal:exaKGMLmainpropal}.
  \begin{propal}
    \label{propal:exaKGMLdone}
For a given admissible background initial data set $(a^\alpha_0,\dot{a}^\alpha_0,\phi_0,\dot{\phi}_0,v_{\mathscr{A}},\dot{v}_\mathscr{A},\psi_\mathscr{A},w^\alpha_\mathscr{A})$, let $\lambda>0$ and let $((e^{evo}_\lambda)^\alpha,\epsilon^{evo}_\lambda,(\dot{e}^{evo}_\lambda)^\alpha,\dot{\epsilon}^{evo}_\lambda,w^{+\alpha}_{\lambda\mathscr{A}},\breve{w}^{+\alpha}_{\mathscr{A}\pm\mathscr{B}},\psi^+_{\lambda\mathscr{A}},\breve{\psi}^+_{\mathscr{A}\pm\mathscr{B}},g^{+\alpha}_{W^+_{\lambda\mathscr{A}}},g^+_{\Psi^+_{\lambda\mathscr{A}}},(e^{evo}_\lambda)^\alpha,\epsilon^{ell}_\lambda,(\dot{e}^{ell}_\lambda)^\alpha,\dot{\epsilon}^{ell}_\lambda)$ be error initial data \textbf{admissible for KGML}, with the associated $-3/2<\delta<-1/2$, from Definition \ref{defi:admissibleKGMLcriteria}. Then, for any $\kappa>0$, determining the projectors from Definition \ref{defi:projfreq}, there exists a time $t_\lambda>0$ for which there exists a local exact multiphase high-frequency solution to KGML $(A^\alpha_{\lambda},\Phi_{\lambda})$ \eqref{eq:KGMLgeneKGM} corresponding to these initial data assembled as in Definition \ref{defi:glueerrorcomponentcriteria}. Moreover,
\begin{align*}
        &A^0_\lambda=A^0_{1\lambda}+Z_\lambda^0 \in\bigcap_{j=0}^2 C^{j}([0,t_\lambda],H^{2-j}_{\delta+j}),
        &A^i_\lambda=A^i_{1\lambda}+Z_\lambda^i\in\bigcap_{j=0}^2 C^{j}([0,t_\lambda],H^{2-j}),\\
         &\Phi_\lambda=\Phi_{1\lambda}+\mathcal{Z}_\lambda \in\bigcap_{j=0}^2 C^{j}([0,t_\lambda],H^{2-j}),
 \end{align*}
where $(A^\alpha_{1\lambda},\Phi_{1\lambda})$ are almost approximate solutions of order 1 to KGM in Lorenz gauge, and where $(Z_\lambda^\alpha,\mathcal{Z}_\lambda)$ is the precise error term given by 
\begin{align}   
&Z_\lambda^\alpha=\sum_{\mathscr{A}\in\mathbb{A}}\lambda^{1}\Re\left(e^{i\frac{u_\mathscr{A}}{\lambda}}\overline{W^{+\alpha}_{\lambda\mathscr{A}}}\right)+\sum_{(\mathscr{A},\mathscr{B})\in\mathscr{C}}\lambda^{1}\Re\left(e^{i\frac{u_{\mathscr{A}}\pm u_{\mathscr{B}}}{\lambda}}\overline{\breve{W}^{+\alpha}_{\mathscr{A}\pm\mathscr{B}}}\right)+(E_{\lambda}^{evo})^{\alpha}+(E_{\lambda}^{ell})^{\alpha},\label{eq:errorparametrixAbootstrap} \\
&\mathcal{Z}_\lambda=\sum_{\mathscr{A}\in\mathbb{A}}\lambda^{1}\Psi^+_{\lambda\mathscr{A}}e^{i\frac{u_\mathscr{A}}{\lambda}}+\sum_{(\mathscr{A},\mathscr{B})\in\mathscr{C}}\lambda^{1}\breve{\Psi}^+_{\mathscr{A}\pm\mathscr{B}}e^{i\frac{u_{\mathscr{A}}\pm u_{\mathscr{B}}}{\lambda}}+\mathcal{E}_\lambda^{evo}+\mathcal{E}_\lambda^{ell}.\label{eq:errorparametrixPhibootstrap}
    \end{align} 
\end{propal}
\begin{proof}
The almost approximate solution $(A^\alpha_{1\lambda},\Phi_{1\lambda})$ is constructed in Proposition \ref{propal:firstapproxApprox}. Then, we argue schematically. The existence of $\boldsymbol{\breve{F}}^+_{\mathscr{A}\pm\mathscr{B}}$ is given by Proposition \ref{propal:propertiesfabinfofab} and that of $\boldsymbol{\mathcal{E}}_\lambda^{ell}$ by Definition \ref{defi:defiEellinfoEell}, both defined on $[0,T]$. For initial data admissible for KGML, there exists a time $t_\lambda$ for which there exists a solution $(\textbf{F}^+_{\lambda\mathscr{A}},\textbf{G}^+_{\lambda\mathscr{A}},\boldsymbol{\mathcal{E}}_\lambda^{evo})$ to the system \eqref{eq:schemeqerrorerrterm}, by Proposition \ref{propal:wellposfullsystWP}. Using Remarks \ref{rem:RHSfabinfofab}, \ref{rem:roleEellinfoEell}, \ref{rem:roleFplusestiFplus}, \ref{rem:roleGplusestiGplus}  and \ref{rem:roleEevosestiEevo}, we deduce that 
\begin{align*}  
\textbf{Z}_\lambda=\sum_{\mathscr{A}\in\mathbb{A}}\lambda^{1}\textbf{F}^+_{\lambda\mathscr{A}}e^{i\frac{u_\mathscr{A}}{\lambda}}+\sum_{(\mathscr{A},\mathscr{B})\in\mathscr{C}}\lambda^{1}\breve{\textbf{F}}^+_{\mathscr{A}\pm\mathscr{B}}e^{i\frac{u_{\mathscr{A}}\pm u_{\mathscr{B}}}{\lambda}}+\boldsymbol{\mathcal{E}}_\lambda^{ell}+\boldsymbol{\mathcal{E}}_\lambda^{evo}
\end{align*} 
is a solution to \eqref{eq:schemeqerrorerrterm}. By construction, the equation \eqref{eq:schemeqerrorerrterm} depends on the background from Proposition \ref{propal:firstapproxApprox} and is designed so that, if $\textbf{Z}_\lambda$ solves \eqref{eq:schemeqerrorerrterm}, then $\textbf{F}_\lambda=\textbf{F}_{1\lambda}+\textbf{Z}_\lambda$ solves \eqref{eq:schemeqIdea}, i.e., $(A_\lambda=A_{1\lambda}+Z_\lambda,\Phi_\lambda=\Phi_{1\lambda}+\mathcal{Z}_\lambda)$ is a solution to KGML \eqref{eq:KGMLgeneKGM}. The regularity follows directly from Propositions \ref{propal:propertiesfabinfofab}, \ref{propal:propertiesEellinfoEell}, and \ref{propal:wellposfullsystWP}.
\end{proof}
\subsection{Uniform existence time in $\lambda$}
We reformulate more precisely and prove the rough Proposition \ref{propal:exasmallKGMmainpropal}.
\begin{propal}
\label{propal:exasmallKGMdone}
Under the assumption of Proposition \ref{propal:exaKGMLdone}, if the error initial data satisfy the required smallness condition in the sense of Definition \ref{defi:reqsmallnesscriteria}, then there exist $\kappa_0>0$, determining the projectors from Definition \ref{defi:projfreq}, and $\lambda_0>0$ such that $t_\lambda=T$ for $0<\lambda<\lambda_0$, i.e., the family of solutions $(A^\alpha_{\lambda},\Phi_{\lambda})_{0<\lambda<\lambda_0}$ to \eqref{eq:KGMLgeneKGM} given in Proposition \ref{propal:exaKGMLdone} exists on the interval $[0,T]$. 
\end{propal} 
\begin{proof}
This follows from a standard bootstrap argument. We assume that the estimates \eqref{eq:assumptions} of the bootstrap Assumptions \ref{ass:evoass} hold for some $0<c_1<c_2$ sufficiently large. Then, by Lemma \ref{lem:improvmentbootstrap}, there exists $0<\lambda_0$ sufficiently small such that the estimates \eqref{eq:assumptions} hold with $c_1$ replaced by $\frac{1}{2}c_1$. The bootstrap is initialised using Proposition \ref{propal:exaKGMLdone}. This implies that $t_\lambda=T$ for $0<\lambda<\lambda_0$.
\end{proof}
\subsection{Exact solutions to KGM in Lorenz gauge}
\label{subsubsection:gpropag}
We reformulate more precisely and prove the rough Proposition \ref{propal:exaKGMmainpropal}.
\begin{propal}
\label{propal:gaugepropaggpropag}
Under the assumption of Proposition \ref{propal:exaKGMLdone}, if the error initial data are \textbf{KGM in Lorenz gauge admissible} in the sense of Definition \ref{defi:admissibleKGMgaugecriteria}, then the solution $(A_\lambda,\Phi_\lambda)$ to KGML \eqref{eq:KGMLgeneKGM} given by Proposition \ref{propal:exaKGMLdone} is also a solution to KGM \eqref{eq:KGMintro} in Lorenz gauge. For all $t\in[0,t_\lambda]$, we have
    \begin{align}
        &\partial_\alpha A^\alpha_\lambda(t)=0, \label{eq:lorenzgaugegpropag}\\
         &\partial_\alpha A^\alpha_0(t)=0,\label{eq:lorenzgaugebggpropag}\\  &\forall\mathscr{A}\in\mathbb{A},\;
\partial_\alpha u_\mathscr{A}W^{\alpha}_{\mathscr{A}},(t)=0,\label{eq:polarbggpropag}\\
\nonumber
    \end{align}
    and, if the components are initially polarized as  
    \begin{align}
&\forall\mathscr{A}\in\mathbb{A},\;
\partial_\alpha u_\mathscr{A}w^{+\alpha}_{\lambda\mathscr{A}}=0,
&\forall(\mathscr{A},\mathscr{B})\in\mathscr{C},\;(\partial_\alpha u_\mathscr{A}\pm\partial_\alpha u_\mathscr{B})\breve{w}^{+\alpha}_{\mathscr{A}\pm\mathscr{B}}=0,\\
\nonumber
\end{align}
    then
        \begin{align}
&\forall\mathscr{A}\in\mathbb{A},\;
\partial_\alpha u_\mathscr{A}W^{+\alpha}_{\lambda\mathscr{A}},(t)=0,
&\forall(\mathscr{A},\mathscr{B})\in\mathscr{C},\;(\partial_\alpha u_\mathscr{A}\pm\partial_\alpha u_\mathscr{B})\breve{W}^{+\alpha}_{\mathscr{A}\pm\mathscr{B}}(t)=0,\\
\nonumber
    \end{align}
    and
    \begin{equation}
    \begin{split}
        &\sum_{\mathscr{A}\in\mathbb{A}}\left(\lambda^{1/2}\Re\left(e^{i\frac{u_{\mathscr{A}}}{\lambda}}\overline{\partial_\alpha W_{\mathscr{A}}^{\alpha}}\right)+\lambda^{1}\Re\left(e^{i\frac{u_{\mathscr{A}}}{\lambda}}\overline{\partial_\alpha W^{+\alpha}_{\lambda\mathscr{A}}}\right)\right)(t)+\sum_{(\mathscr{A},\mathscr{B})\in\mathscr{C}}\lambda^{1}\Re\left(e^{i\frac{u_{\mathscr{A}}\pm u_{\mathscr{B}}}{\lambda}}\overline{\partial_\alpha \breve{W}_{\mathscr{A}\pm\mathscr{B}}^{+\alpha}}\right)(t)\\ 
         &+\partial_\alpha (E_{\lambda A}^{ell})^{\alpha}(t)+\partial_\alpha (E_{\lambda A}^{evo})^{\alpha}(t)=0.\label{eq:gaugeerrorgpropag}
    \end{split}
    \end{equation}
\end{propal}
\begin{proof}
 Since KGM in Lorenz gauge admissible initial data satisfy the constraints, the propagation of gauge follows directly from the argument of Section \ref{subsubsection:propaggauegeneKGM}.\\
For the decomposition of the divergence of $A_\lambda$, identities \eqref{eq:lorenzgaugegpropag} and \eqref{eq:lorenzgaugebggpropag} follow from \eqref{eq:polarApprox} and \eqref{eq:lorenzgaugebgApprox}. For the propagation of polarization, Remark \ref{rem:nonschemtperrterm} shows that the RHS of the evolution equation for $W^{+\alpha}_{\lambda\mathscr{A}}$ (resp.  $\breve{W}^{+\alpha}_{\mathscr{A}\pm\mathscr{B}}$) is orthogonal to $\partial_\alpha u_\mathscr{A}$ (resp. $\partial_\beta u_{\mathscr{A}}\pm \partial_\beta u_{\mathscr{B}}$). Thus, it follows from direct computations that $\mathscr{L}_{\mathscr{A}}(W^{+\alpha}_{\lambda\mathscr{A}}\partial_\alpha u_\mathscr{A})=0$ (resp. 
$\mathscr{L}_{\mathscr{A}\pm\mathscr{B}}(\breve{W}^{+\alpha}_{\mathscr{A}\pm\mathscr{B}}(\partial_\beta u_{\mathscr{A}}\pm \partial_\beta u_{\mathscr{B}}))=0$). Finally, the last identity is obtained by collecting the remaining terms. 
\end{proof}
\begin{rem}
\label{rem:polarspecsetgpropag}
    The set we exhibit in Proposition \ref{propal:specsetconsinit} is polarized, in the sense that $\breve{w}^{+\alpha}_{\mathscr{A}\pm\mathscr{B}}=w^{+\alpha}_{\lambda\mathscr{A}}=0$.
\end{rem}
\section{Putting everything together}
\label{section:conclu}
We prove Theorems \ref{unTheorem:mainth1results} and \ref{unTheorem:mainth2results} and conclude with several general remarks.
\subsection{Proof of the Theorems}
\label{subsection:proofth}
For the first Theorem \ref{unTheorem:mainth1results}, we only detail the proof of point \ref{item:item1th1results} since the remaining statements follow directly from our construction. The point \ref{item:item1th1results} asserts that, for a given admissible initial ansatz, there exists $\lambda_0>0$ such that there is a family of multiphase high-frequency solutions $(A_\lambda,\Phi_\lambda)_{0<\lambda<\lambda_0}$ to KGM \eqref{eq:KGMintro} on $[0,T]$. Following the strategy outlined in Section \ref{subsection:Idea}, we first construct the first-order approximate solution (step \ref{item:item1Idea}) using Proposition \ref{propal:firstapproxApprox}. Then, we construct initial data for the error that are admissible in the sense of Definitions \ref{defi:admissibleKGMLcriteria}, \ref{defi:reqsmallnesscriteria} and \ref{defi:admissibleKGMgaugecriteria} using Proposition \ref{propal:specsetconsinit} (step \ref{item:item2Idea}). Next, we construct a family of exact solutions $(A_\lambda,\Phi_\lambda)_{0<\lambda}$ to KGML \eqref{eq:KGMLgeneKGM} associated with the error initial data, each solution being defined on $[0,t_\lambda]$ (step \ref{item:item3Idea}). Using Proposition \ref{propal:estifullKGMLbootstrap}, we exhibit $0<\lambda_0$ such that the sub-family $(A_\lambda,\Phi_\lambda)_{0<\lambda<\lambda_0}$ exists on the uniform interval $[0,T]$ (step \ref{item:item4Idea}). Finally, we show that the gauge condition propagates and that the family $(A_\lambda,\Phi_\lambda)_{0<\lambda<\lambda_0}$ indeed solves KGM \eqref{eq:KGMintro} while satisfying the Lorenz gauge condition (step \ref{item:item5Idea}). As a consequence, we establish the stability of the multiphase WKB analysis applied to KGM in Lorenz gauge, using only a first-order approximate solution and under low regularity assumptions.  \\\\
For the second Theorem, the proof is direct from the structure of the solution described above. In particular, the $L^p([0,T],L^\infty)$-convergence is obtained via the Strichartz inequality for waves, as presented in Theorem \ref{unTheorem:strichartzAx}, while the fact that the limit $(A_0,\Phi_0)$ is a solution to \eqref{eq:KGMnullintro} follows from the construction of the approximate solution in Proposition \ref{propal:firstapproxApprox}. This provides an explicit manifestation of a backreaction phenomenon in the high-frequency limit of solutions to KGM. 
\subsection{Comparison with the literature}
\label{subsection:jeanneompare}
It is instructive to compare our results with those of \cite{zbMATH01799448}. The author of the latter constructs monophase high-frequency solutions to the Yang-Mills-Higgs-Dirac system, which includes the KGM system \eqref{eq:KGMintro}. More precisely, approximate solutions are constructed at arbitrary order using a WKB expansion of the form
\begin{align*}
    \textbf{F}_{M\lambda}=\sum_{0\leq i\leq M-1}\lambda^{i/2}\left(\tilde{\textbf{F}}_{i}(x)+\lambda^{1/2}\overset{\star}{\textbf{F}}_{i}\left(x,\frac{u}{\lambda}\right)\right)
\end{align*}
together with exact solutions  
\begin{align*}
    \textbf{F}_{\lambda}= \textbf{F}_{M\lambda}+\textbf{Z}_\lambda,
\end{align*}
constructed on a time interval uniform in $\lambda$, provided that $M\geq n$ where $n$ denotes the spatial dimension. The functions $\textbf{F}_i(x,.)$ are decomposed into a mean part $\tilde{\textbf{F}}_{i}(x)$ and a purely oscillatory part $\overset{\star}{\textbf{F}}_{i}(x,.)$. The oscillatory profiles involve general harmonics rather than being restricted to $\cos(\frac{u}{\lambda})$ and $\sin(\frac{u}{\lambda})$. By contrast, in our construction, harmonics do not appear at the level of the (first-order) approximate solution
\begin{align*}
\textbf{F}_{1\lambda}=\tilde{\textbf{F}}_0(x)+\lambda^{1/2}\overset{\star}{\textbf{F}}_0\left(x,\frac{u}{\lambda}\right):=\textbf{F}_0(x)+\lambda^{1/2}\textbf{F}(x)e^{i\frac{u}{\lambda}}.
\end{align*}
and our method naturally accommodates multiphase superpositions. \\\\ 
In \cite{zbMATH01799448}, the following bounds are obtained for the Yang-Mills potential and the scalar field (corresponding to $A_\lambda$ and $\Phi_\lambda$ in the KGM system):
\begin{align*}
    &\forall \lambda<\lambda_0,\; \lambda^{M/2}||\textbf{F}_{\lambda}||_{H^{1/2}}+||\textbf{Z}_\lambda||_{H^{1/2}}\leq \lambda^{M/2}C, 
\end{align*}
for some constant $C$ depending on the initial data. Formally extrapolating these estimates to the case $M=1$, which is excluded in \cite{zbMATH01799448}, corresponds to the bounds obtained in Theorem \ref{unTheorem:mainth1results}. In the present work, we construct the error term already at order $M=1$, that is, below the dimensional threshold $M=3$.
In \cite{zbMATH01799448}, the obstruction to this low-order construction originates from the use of specific product estimates (Proposition 3.1.1 in \cite{zbMATH01799448} and 2.2.2 in \cite{zbMATH02124168}). This obstruction is consistent with the more general results of \cite{METIVIER2009169} for first-order hyperbolic systems. Higher-order approximate solutions typically require the construction of a large hierarchical system of equations, as well as a detailed analysis of harmonics and phase interactions arising from nonlinear effects. Our first-order construction bypasses these difficulties.
\subsection{Regularity of the background}
\label{subsection:regubg}
The background $\textbf{B}'_0:=(A_0,\Phi_0,w_\mathscr{A},\psi_\mathscr{A},\textbf{d} u_{\mathscr{A}})$ defined in Proposition \ref{propal:firstapproxApprox} is more regular than the full solution (and, in particular, the error term). This additional regularity is mandatory for the well-posedness of \eqref{eq:schemfullsystemevosys} and \eqref{eq:schemtpfabinfofab}, as well as for the estimates carried out in Section \ref{subsubsection:evosys}. For example, the equation satisfied by $\boldsymbol{\mathcal{E}}_\lambda^{evo}$ contains the source term $\Box\textbf{F}_\mathscr{A}$,  which formally leads to the regularity relation $\partial\boldsymbol{\mathcal{E}}_\lambda^{evo}\sim\Box\textbf{F}_\mathscr{A}$. Consequently, $\textbf{B}_0=(A_0,\Phi_0,w_\mathscr{A},\psi_\mathscr{A})$ must be one degree more regular than the error term. Moreover, the phases satisfy the regularity bound
\begin{align*}
&\max_{\mathscr{A}\in\mathbb{A}}\sum_{j=0}^5||u_{\mathscr{A}}||_{C^{j}([0,T],H^{5-j}_{{\delta_1}+j}}\leq c_0.
\end{align*}
This level of regularity is required to ensure the well-posedness of \eqref{eq:firstapproxsysApprox} with solutions in $H^3$ and to be able to use the commutator estimates of Proposition \ref{propal:commut1Ax} in the analysis of \eqref{eq:schemtpGplusestiGplus}. In fact, it would suffice to assume
$\Box u_{\mathscr{A}}$ in $H^{3}$ and $u_{\mathscr{A}}$ in $H^4_{\delta_1}$. Constructing the error term for higher-order approximations would require an even more regular background.
\section{Appendix}
\label{section:Appendix}
\subsection{Estimates for evolution equations}
\label{subsection:estimatesappendix}
\begin{propal}
\label{propal:energywaveAx}
   Let $\delta'\in\mathbb{R}$ and $m\in\mathbb{N}$. For a function $f$ defined on $[0,T']\times\mathbb{R}^3$, for some $T'>0$, solution to 
\begin{equation}
\begin{cases}
\label{eq:waveeqAX}
        \Box f=h, \\
        f(t=0)=f_0\in H^{m+1}_{\delta'-1}\;\; \partial_tf(t=0)=f_1\in H^m_{\delta'},
\end{cases}
\end{equation}
the following energy estimates for $\partial f=(\partial_tf,\nabla f)$ holds
  \begin{align}
    ||\partial f(t)||_{H^m_{\delta'}}\leq C(||\partial 
f(0)||_{H^m_{\delta'}}+\int_0^t||h(s)||_{H^m_{\delta'}}+||\partial f(s)||_{H^m_{\delta'}}ds).  \label{eq:estienergyweightwaveAx}
  \end{align}
Moreover, in the case $\delta'=0$,
   \begin{align}
    ||\partial f(t)||_{H^m}\leq C(||\partial 
f(0)||_{H^m}+\int_0^t||h(s)||_{H^m}ds) .  \label{eq:estienergyweightwave2Ax}
  \end{align}
\end{propal} 
\begin{proof} 
We present the estimate for $m=0$ and for any ${\delta'}\in\mathbb{R}$. Multiplying the equation by $(1+|x|^2)^{\delta'}\partial_tf$ gives 
\begin{align*}
&\Box f(1+|x|^2)^{\delta'}\partial_tf=h(1+|x|^2)^{\delta'}\partial_tf.
\end{align*}
Integrating over $\mathbb{R}^3$, we obtain
\begin{align*}
\int_{\mathbb{R}^3}\Box f(1+|x|^2)^{\delta'}\partial_tfdx&=\int_{\mathbb{R}^3}h(1+|x|^2)^{\delta'}\partial_tfdx\\
&=-\frac{d}{dt}\int_{\mathbb{R}^3}(1+|x|^2)^{\delta'}\frac{\partial_t f^2+|\nabla f|^2}{2}dx-2\delta'\int_{\mathbb{R}^3}(1+|x|^2)^{\delta'-1}x_i\partial_if\partial_tfdx.    
\end{align*}
Since $|(1+|x|^2)^{\delta'-1}x_i|\leq (1+|x|^2)^{\delta'}$, the second term can be controlled by the $L^2_{\delta'}$ norm. Thus, 
\begin{align*}
& \frac{d}{dt}\frac{1}{2}||\partial f(t)||^2_{L^2_{\delta'}}\leq C(||\partial f(t)||_{L^2_{\delta'}}||h(t)||_{L^2_{\delta'}}+||\partial f(t)||^2_{L^2_{\delta'}}),
\end{align*}
and so,
\begin{align*}
&||\partial f(t)||_{L^2_\delta}\leq C(||\partial f(0)||_{L^2_\delta}+\int_0^t||h(s)||_{L^2_\delta}+||\partial f(s)||_{L^2_\delta}ds). 
\end{align*}
The general case follows by commuting the equation with $\partial_{\overrightarrow{m}}$ where $\overrightarrow{m}$ is a multi-index with $|\overrightarrow{m}|=m$, multiplying by $\partial_{t \overrightarrow{m}}f$ instead of $\partial_tf$, and using the adjusted weight.\\
\end{proof}
Moreover, bounds on $\partial f(t)$ together with bounds on $f_0$ yield bounds on $f(t)$ for finite time, using
 \begin{align}
    \frac{d}{dt}||f(t)||_{L^2}\leq ||\partial_tf(t)||_{L^2}.
    \label{eq:derivosefAx}
\end{align}
\begin{propal}
\label{propal:energywaveplusAx}
    Let $f$ be a solution to \eqref{eq:waveeqAX}. Then, for all $ t\in[0,T^\prime]$, $\delta'\in\mathbb{R}$ and for every $m\in\mathbb{N}$
\begin{align}
    ||f(t)||_{H^{m+1}_{\delta'-1}}\leq (||f_0||_{H^{m+1}_{\delta'-1}}+t\sup_{s\in[0,T^\prime]}||\partial f(s))||_{H^{m}_{\delta'}}).\label{eq:estienergywavefullAx}
\end{align}
\end{propal}
\begin{proof} 
The control of the $H^{m}_{\delta'-1}$ norm of $f$ is obtained by integrating \eqref{eq:derivosefAx} for higher-order weighted Sobolev norms. The control of the $(m+1)$-th derivatives of $f$ in $L^2_{\delta'+m}$ follows directly from the bound on $||\partial f||_{H^m_{\delta'}}$.\\
\end{proof}

\begin{propal}
\label{propal:energywavederivAx}
 Let $f$ be a solution to \eqref{eq:waveeqAX}. Then, for all $ t\in[0,T^\prime]$, $\delta'\in\mathbb{R}$,
and for every $m\in\mathbb{N}$, the time derivatives of $f$ obey the following inequalities  
\begin{align}
    &||\partial_tf(t)||_{H^m_{\delta'}}\leq C(||\partial 
f(0)||_{H^m_{\delta'}}+\int_0^t||h(s)||_{H^m_{\delta'}}+||\partial f(s)||_{H^m_{\delta'}}ds),  \label{eq:estienergywavederivt1Ax}\\
&\text{and if $m\in\mathbb{N}^\star$}\nonumber\\
  &||\partial_{tt}^2 f(t)||_{H^{m-1}_{\delta'+1}}\leq (|| 
f(t)||_{H^{m+1}_{\delta'-1}}+||h(t)||_{H^{m-1}_{\delta'+1}}). \label{eq:estienergywavederivt2Ax} 
\end{align}

\end{propal}
\begin{proof}
    The first inequality is a specific case of \eqref{eq:estienergyweightwaveAx} and the second follows directly using the wave equation in \eqref{eq:waveeqAX}.
\end{proof}

\begin{propal}
\label{propal:estienergytpAX}
       Let $m\in\mathbb{N}$. For a function $f$ defined on $[0,T']\times\mathbb{R}^3$, for some $T'>0$, solution to 
\begin{equation}
\begin{cases}
\label{eq:tpeqAx}
     2\partial^\alpha u \partial_\alpha f+f\Box u=j,\\
    f(t=0)=f_0 \in H^m,
\end{cases}
\end{equation}
 with $u$ a smooth characteristic phase satisfying  
 \begin{align}
 &\underset{(t,x)\in[0,T']\times Supp(f)}{\min}|\partial_tu(t,x)|>\eta,
 &||\partial u||_{H^{m'}}+||\Box u||_{H^{m'}}+||\partial_t\partial u||_{H^{m'-1}}+||\partial_t\Box u||_{H^{m'-1}}\leq C_u,
\end{align}
for some $\eta\in\mathbb{R}^\star_+$, some $C_u\in\mathbb{R}^\star_+$ and for $m'=\max(m,2)$, the following energy estimate holds
  \begin{align}
    ||f||_{L^\infty([0,T'],H^m)}\leq C_{u}(||
f(0)||_{H^m}+\int_0^{T'}||j(s)||_{H^m}+||f(s)||_{H^m}ds).\label{eq:estienergytp1AX}  
  \end{align}
\end{propal}

\begin{proof}
In the case $m=0$, we multiply \eqref{eq:tpeqAx} by $f$
\begin{align*}
    &\partial_\alpha(\partial^\alpha u_\mathscr{A}f^2)=jf,\\
\end{align*}
which yields a conservative transport equation. By Stokes' Theorem\footnote{Boundary terms are ignored since they play no role in our application.}, we get 
 \begin{align*}
    \int_{\mathbb{R}^3} (\partial^0u_\mathscr{A}f^2)(t)dx=\int_{\mathbb{R}^3} (\partial^0u_\mathscr{A}f^2)(0)dx+\int^{t}_0{\int_{\mathbb{R}^3} (jf)(s)dx}ds,
  \end{align*}
so that, using the lower bounds on $|\partial_t u|$,
 \begin{align*}
    &\eta||f(t)||_{L^2}^2\leq C_u||f(0)||_{L^2}^2+||f||_{L^\infty([0,T'],L^2)}\int^t_0{||j(s)||_{L^2}}ds,
  \end{align*}
  which implies 
  \begin{align*}
     &||f||_{L^\infty([0,T'],L^2)}\leq C_{u,\eta}(||f(0)||_{L^2}+\int^{T'}_0{||j(s)||_{L^2}}ds).
  \end{align*}
For $m\geq1$, commuting the equation with $\partial_{\overrightarrow{m}}$, where $\overrightarrow{m}$ is a multi-index with $|\overrightarrow{m}|=m$, yields 
\begin{align*}
    &(2\partial^\alpha u \partial_\alpha+\Box u)\partial_{\overrightarrow{m}} f=[2\partial^\alpha u \partial_\alpha+\Box u,\partial_{\overrightarrow{m}}]f+\partial_{\overrightarrow{m}}j,
\end{align*}
where
\begin{align*}
    ||[2\partial^\alpha u \partial_\alpha+\Box u,\partial_{\overrightarrow{m}}]f(t)||_{L^2}\leq C(||\partial u(t)||_{H^{m'}}+||\Box u(t)||_{H^{m'}})||f(t)||_{H^m},
\end{align*}
for $m'=\max(m,2)$ and a universal constant $C\in\mathbb{R}^\star_+$.
\end{proof}
\begin{propal}
\label{propal:estienergytpderivtempsAX}
 Let $f$ be a solution to \eqref{eq:tpeqAx}. Then, for all $m\in\mathbb{N}$, the time derivatives of $f$ obey the following inequalities 
      \begin{align}
    &||\partial_tf||_{L^\infty([0,T'],H^m)}\leq C_{u,\eta}(||
f||_{L^\infty([0,T'],H^{m+1})}+||j||_{L^\infty([0,T'],H^m)}),\label{eq:estienergytpderivt1Ax} \\
    &||\partial^2_{tt}f||_{L^\infty([0,T'],H^{m-1})}\leq C_{u,\eta}(||
f||_{L^\infty([0,T'],H^{m+1})}+||j||_{L^\infty([0,T'],H^m)}). \label{eq:estienergytpderivt2Ax}
  \end{align}
\end{propal}
\begin{proof}
The first estimate follows directly from the transport equation.
The second follows by differentiating the transport equation in time and using the previous estimate to control the appearing first-order time derivatives.
\end{proof}

\subsection{Commutator}
\label{subsection:commut}
\begin{propal}
\label{propal:commut1Ax}
Let $g$ be a smooth Lorentzian metric. We define $\Box:=D_\alpha D^\alpha$, where $D$ denotes the covariant derivative and indices are contracted with respect to $g$. Let $f$, defined on $[0,T']\times\mathbb{R}^3$, be compactly supported and solution to 
\begin{equation}
    \begin{cases}
    \label{eq:tpeqcommutAx}
          2\partial^\alpha u \partial_\alpha f+\Box uf=j,\\
    f(t=0)=f_0\in H^2,
    \end{cases}
\end{equation}
for $u$ a smooth characteristic phase such that, for every compact set $\Omega\subset\mathbb{R}^3$, there exists $C_{u,\Omega}$ such that 
\begin{align}
    ||\Box\partial u||_{L^\infty([0,T'],L^\infty(\Omega))}+||\Box\Box u||_{L^\infty([0,T'],H^{1/2}(\Omega))}+||\partial\partial u||_{L^\infty([0,T'],L^\infty(\Omega))}\leq C_{u,\Omega}.\label{eq:phaseassumecommutAx}
\end{align}
Assume that 
\begin{equation}
\label{eq:initialboxcommutAx}
    (\Box f)|_{t=0}=G_0,\; G_0\in H^1.
\end{equation}
Let $G$, defined on $[0,T']\times\mathbb{R}^3$, be solution to 
\begin{equation}
\begin{cases}
\label{eq:tpeqcommutGAx}
    2\partial^\alpha u \partial_\alpha G+\Box uG=\Box j+[(2\partial^\alpha u\partial_\alpha+\Box u),\Box]f,\\
    G(t=0)=G_0.
    \end{cases}
\end{equation}
Then, $G$ belongs to $H^1$ for all $t\in[0,T']$ and
\begin{equation}
\label{eq:estiGcommutax}
    ||G(t)||_{H^1}\leq C_{g,u,\Omega}\int^t_0(||\Box j(s)||_{H^1}+||\partial j(s)||_{H^1}+|| f(s)||_{H^2})ds+||G_0||_{H^1},
\end{equation}
where we use the pointwise-in-time estimate
\begin{equation}
\label{eq:esticommutcommutax}
    ||[(2\partial^\alpha u\partial_\alpha+\Box u),\Box]f(t)||_{L^2}\leq C_{g,u}(||G(t)||_{L^2}+||\partial j(t)||_{L^2}+|| f(t)||_{H^1}).
\end{equation}
\end{propal}
\begin{proof}
Firstly, if $u=k^\alpha x_\alpha$ for $k$ a constant vector, that is, if $u$ is the phase of a plane wave, then $\partial_\alpha u=k_\alpha$ so that 
\begin{equation}
    [(2\partial^\alpha u\partial_\alpha+\Box u),\Box]=[(2k^\alpha\partial_\alpha),\Box]=0.
\end{equation}
The result follows immediately. We now treat the general case. For simplicity, we work in wave coordinates satisfying $g^{\mu\nu}\Gamma^{\gamma}_{\mu\nu}=0$. Then 
\begin{align*}
[(2\partial^\alpha u\partial_\alpha+\Box u),\Box]f&=2\partial^\alpha u\partial_\alpha(g^{\mu\nu})\partial_{\mu}\partial_{\nu}f-2\partial^\alpha u\partial_\alpha(g^{\mu\nu}\Gamma^{\gamma}_{\mu\nu})\partial_\gamma f-4\Box\partial^{\alpha}u\partial_\alpha f-4g^{\mu\nu}\partial_\mu\partial^{\alpha}u\partial_\nu\partial_{\alpha}f-\Box\Box uf\\
&+g^{\mu\nu}\Gamma^{\gamma}_{\mu\nu}2\partial^\alpha\partial_\gamma u\partial_\alpha+g^{\mu\nu}\Gamma^{\gamma}_{\mu\nu}\partial_\gamma \Box u\\
&=-4(\Gamma^\mu_{\rho\alpha}g^{\rho\nu}\partial^\alpha u+g^{\mu\alpha}\partial_\alpha\partial^{\nu}u)\partial_\mu\partial_\nu f-4(\Box\partial^\alpha u)\partial_\alpha f-\Box\Box uf\\
&=-4(\nabla^\mu\partial^\nu u)\partial_\mu\partial_\nu f-4(\Box\partial^\alpha u)\partial_\alpha f-\Box\Box uf.\\
\end{align*}
We decompose the metric as $$g^{\mu\nu}=-\frac{1}{2}(\partial^{\mu}u\partial^{\nu}\bar{u}+\partial^{\mu}\bar{u}\partial^{\nu}u)+\delta^{MN}e_M^\mu e_N^\nu,$$ where $e_{M,N}$ are spacelike vector fields orthonormal to $\partial u$. The vector field $\partial\bar{u}$ is the unique null vector satisfying $g(\partial u,\partial\bar{u})=-2$ and $g(E_ {M},\partial\bar{u})=0$.\\\\
We know that $\partial_\nu uD^\mu\partial^\nu u=0$ \eqref{eq:eikonalgenephase} and $\partial_\nu uD^\nu\partial^\mu u=0$ \eqref{eq:geodgenephase}. Therefore, locally, the tensor $\nabla^\mu\partial^\nu u$ can be written as
$$a\partial u\otimes\partial u+b(\partial u\otimes e_{M,N}+ e_{M,N}\otimes\partial u )+c(e_{M,N}\otimes e_{M,N}),$$
$a$, $b$ and $c$ depend on $\partial\partial u$ and $\partial g$, while the other components vanish. 
Hence, schematically 
\begin{equation*}
    [(2\partial^\alpha u\partial_\alpha+\Box u),\Box]f=f_1(\partial\partial u,\partial g)\partial(\partial^\alpha u\partial_\alpha f)+f_2(\partial\partial u,\partial g)\partial f+f_3(\partial\partial u,\partial g)\partial_M\partial_N f+\Box\partial u\partial f+\Box\Box uf,
\end{equation*}
for suitable coefficient functions $f_1,f_2,f_3$.\\\\
The tangential Hessian $\partial_M\partial_N f$ can be estimated in $L^2$ on surfaces of constant $u$ and $\bar{u}$ via $\Delta_{M,N} f$ and $\partial f$, see \cite{szeftel2012parametrix}.
Using the identity
\begin{equation*}
    \Delta_{M,N} f=\Box f+\frac{1}{2}(\partial^{\mu}u\partial^{\nu}\bar{u}+\partial^{\mu}\bar{u}\partial^{\nu}u)\partial_\mu\partial_\nu f,
\end{equation*}
we obtain
\begin{equation*}
    ||\partial_M\partial_N f||_{L^2}\leq C_{u,g}(||\Box f||_{L^2}+||\partial(\partial^\alpha u \partial_\alpha f)||_{L^2}+||\partial f||_{L^2}).
\end{equation*}
Putting everything together, we deduce 
\begin{equation*}
   ||[(2\partial^\alpha u\partial_\alpha+\Box u),\Box]f(t)||_{L^2}\leq C_{g,u,\Omega}(||\Box f||_{L^2}+||\partial(\partial^\alpha u \partial_\alpha f)||_{L^2}+||f||_{H^1})=  C_{u,g}(||G||_{L^2}+||\partial j||_{L^2}+||f||_{H^1}).
\end{equation*}
This is the desired inequality.
\end{proof}
\subsection{Sobolev inequalities}
\label{subsection:soboineq}
In this Section, we recall classical Sobolev embedding and interpolation results.
\begin{propal}
\label{propal:embedsoboAx}
    Let $n\in\mathbb{N}^\star$ be the considered dimension. Let $k>0$, $1\leq p<\infty$, $1\leq q<\infty$, $k>l$ and $\frac{1}{p}\geq \frac{1}{q}\geq \frac{1}{p}-\frac{k-l}{n}$. Then
$$W^{k,p}(\mathbb{R}^n)\hookrightarrow W^{l,q}(\mathbb{R}^n).$$
\end{propal} 
\begin{proof}
    From \cite{zbMATH05633610}.
\end{proof} 

\begin{propal}
\label{propal:embedlebAx}
     Let $n\in\mathbb{N}^\star$ be the considered dimension. Let $k>0$, $1\leq p<\infty$, $1\leq q<\infty$. Then
\begin{align*}
    \frac{1}{p}\geq \frac{1}{q}\geq\frac{1}{p}-\frac{k}{n}\geq 0 &\implies W^{k,p}(\mathbb{R}^n)\hookrightarrow L^q(\mathbb{R}^n),\\
     \frac{1}{p}-\frac{k}{n}<0 &\implies W^{k,p}(\mathbb{R}^n)\hookrightarrow L^\infty(\mathbb{R}^n).\\
\end{align*}
\end{propal}
\begin{proof}
    From \cite{zbMATH05633610}.
\end{proof} 

\begin{propal}
\label{propal:embedholderAx}
    Let $n\in\mathbb{N}^\star$ be the considered dimension. Let $m\in\mathbb{N}$, $1\leq p\leq\infty$, and $\infty>s>n/p$. For weights $\delta',\delta^{\prime\prime}$ satisfying $\delta'\leq\delta^{\prime\prime}+\frac{n}{p}$, we have the embedding of weighted Sobolev spaces $W^{s+m,p}_{\delta^{\prime\prime}}$ into the weighted Hölder spaces $C^{m}_{\delta'}$: 
    \begin{align*}
        W^{s+m,p}_{\delta^{\prime\prime}}(\mathbb{R}^n)\hookrightarrow C^{m}_{\delta'}(\mathbb{R}^n).
    \end{align*}
\end{propal}
\begin{proof}
From \cite{Choquetbruhat1969ConstructionDS}.
\end{proof}
\begin{propal}
\label{propal:interpolsobAx}
    Let $n\in\mathbb{N}^\star$ be the considered dimension. For 
$s_\theta,s_0,s_1\in\mathbb{R}$, 
$1\leq p_\theta,p_0,p_1\leq \infty$ and $0<\theta<1$ with
$\frac{1}{p_\theta}=\frac{1-\theta}{p_0}+\frac{\theta}{p_1}$ and $s_\theta=(1-\theta)s_0+(\theta)s_1$, we have
\begin{align}
\label{eq:interpolsobAx}
    ||f||_{W^{s_\theta,p_\theta}(\mathbb{R}^n)}\leq (||f||_{W^{s_0,p_0}(\mathbb{R}^n)})^{1-\theta}(||f||_{W^{s_1,p_1}(\mathbb{R}^n)})^{\theta}.
\end{align}
\end{propal} 
\begin{proof}
    From \cite{zbMATH05633610}.
\end{proof}

\begin{propal}[Gagliardo-Nirenberg inequality]
\label{propal:gagliardoAx}
    Let $n\in\mathbb{N}^\star$ be the considered dimension.
    Let $1\leq r,q\leq\infty$, $j,m\in\mathbb{N}$ such that $j<m$ and $p\geq 1$, then for 
    \begin{align*}
        \frac{1}{p}=\frac{j}{n}+\theta\left(\frac{1}{r}-\frac{m}{n}\right)+(1-\theta)\frac{1}{q},\;\; \frac{j}{m}\leq\theta\leq1,
    \end{align*}
    there exists some $C>0$ such that
    \begin{align*}
        ||\partial^j_xf||_{L^p(\mathbb{R}^n)}\leq C (||\partial^m_xf||_{L^r(\mathbb{R}^n)})^\theta (||f||_{L^q(\mathbb{R}^n)})^{1-\theta}.
    \end{align*}
\end{propal}
\begin{proof}
    From \cite{zbMATH05633610}.
\end{proof}

\subsection{Laplacian inversion}
\label{subsection:lapinversion}
\begin{defi} 
\label{defi:laplaciandefAx}
 Let $n\in\mathbb{N}^\star$ be the considered dimension. For $1<p<\infty$, $\delta\in\mathbb{R}$ and $m\in\mathbb{N}$, we define the classical Laplacian as $\Delta=\partial_i\partial^i$ where
$$\Delta: W^{2+m,p}_{\delta}(\mathbb{R}^n)\xrightarrow{} W^{m,p}_{\delta+2}(\mathbb{R}^n).$$
\end{defi}
\begin{unTheorem}
\label{unTheorem:laplcianAx}
    For $1<p<\infty$ and $n\geq 3$, we set $p'=1-1/p$. Then, for $\delta\in\mathbb{R}$ such that $\delta\not\equiv-2+\frac{n}{p^{\prime}}\mod(\mathbb{N})$ and $-\delta\not\equiv\frac{n}{p}\mod(\mathbb{N})$ with 
    \begin{align*}
        -n/p<\delta<-2+n/p^{\prime},
    \end{align*} 
    the Laplacian operator is an isomorphism from $ W^{2+m,p}_{\delta}(\mathbb{R}^n)$ to $W^{m,p}_{\delta+2}(\mathbb{R}^n)$ and is therefore invertible.\\\\
Moreover, there exists $C>0$ such that for any $h\in L^2_{\delta+2}$ and for $f\in H^{2}_{\delta}$ the solution to 
\begin{equation*}
    -\Delta f=h,
\end{equation*}
we have 
\begin{equation}
    ||f||_{H^2_\delta}\leq C ||h||_{L^2_{\delta+2}}.
\end{equation}
\end{unTheorem}
\begin{proof}
    From \cite{zbMATH03663635}.
\end{proof} 
\begin{rem}
\label{rem:laplaciansubcaseAx}
    In the case $n=3$, $p=2$ and $p^{\prime}=2$, we obtain  
$-3/2<\delta<-1/2$.
For $m=0$, the source term must satisfy $h\in L^2_{\delta+2}\subset{L^2}$ (since $\delta+2>0$). If $Supp(h)\subset{K}$ for $K$ compact, then 
\begin{equation*}
    ||h||_{L^2_{\delta+2}}\leq C||h||_{L^2(K)}.
\end{equation*}
\end{rem}
We deduce the following useful Lemma.\\
\begin{lem}
\label{lem:laplcianpluscompactAx}
Let 
\begin{equation}
\label{eq:deflaplacianpluscompactAx}
    L=-\Delta+g,
\end{equation}
where $g\geq 0$ is a compactly supported function in $H^2$ satisfying $||g||_{H^2}\leq C_g$. Then, for $\delta\in(-3/2,-1/2)$, the operator $L$ is an isomorphism from $ H^{2}_{\delta}$ to $L^2_{\delta+2}$ and is therefore invertible.\\\\
Moreover, there exists $C>0$ such that for any $h\in L^2_{\delta+2}$ and for $f\in H^{2}_{\delta}$ the solution to 
\begin{equation}
\label{eq:laplacianpluscompactequationAx}
    L f=h,
\end{equation}
we have 
\begin{equation}
\label{eq:laplacpluscompactestimAx}
    ||f||_{H^2_\delta}\leq C ||h||_{L^2_{\delta+2}}.
\end{equation}
\end{lem}
\begin{proof}
    First, we prove that the multiplication by $g$ defines a compact operator from $H^{2}_{\delta}$ to $L^2_{\delta+2}$.\\
    Let $B$ be a ball of finite radius such that $Supp(g)\subset B$. For $f\in H^{2}_{\delta}$ with $||f||_{H^{2}_{\delta}}\leq C$, we have 
    \begin{align*}
        ||gf||_{H^1(B)}\leq ||g||_{H^{2}(B)}||f||_{H^{2}(B)}\leq c(B)||f||_{H^{2}_{\delta}}||g||_{H^{2}}\leq c(C,B,C_g).
    \end{align*}
    By Rellich-Kondrachov Theorem, the set $\{gf|f\in H^{2}_{\delta}, ||f||_{H^{2}_{\delta}}\leq C\}$ is a compact subset of $L^2(B)$ and therefore a compact subset of $L^2_{\delta+2}$. \\\\Next, we prove invertibility. The classical Laplacian is already an isomorphism from $H^{2}_{\delta}$ to $L^2_{\delta+2}$ for $\delta\in(-3/2,-1/2)$ and is therefore Fredholm of index $0$. Adding a compact perturbation does not modify the Fredholm index of the operator. Thus, the operator $L$ is Fredholm of index 0. Consequently, the injectivity implies the invertibility. \\\\
    To prove injectivity, we use the maximum principle for functions with Sobolev regularity. Let $f\in H^2_\delta$ satisfy
    \begin{equation*}
        Lf=0.
    \end{equation*}
    Since $g\geq 0$, we may apply the maximum principle (see Chapter 8 of \cite{zbMATH01554166}). For any set $\Omega$ of finite radius $R$ with smooth boundary, 
    \begin{align*}
       & \sup_{\Omega}f\leq  \sup_{\partial\Omega}f|_{\partial\Omega}, & \inf_{\Omega}f\geq  \inf_{\partial\Omega}f|_{\partial\Omega},
    \end{align*}
    where the boundary trace is well defined since $f$ is continuous by the Sobolev embedding into Hölder spaces (Proposition \ref{propal:embedholderAx}). Moreover,
        \begin{align*}
        |f(x)|\leq \frac{C||f||_{H^2_\delta}}{|x|^\beta},
    \end{align*}
    for some $C>0$ and for $0\leq\beta<\delta+3/2$ with $\delta\in(-3/2,-1/2)$. As a consequence,
    \begin{align*}
       & |\sup_{\Omega}f|\leq  \frac{C||f||_{H^2_\delta}}{R^\beta}, & |\inf_{\Omega}f|\leq  \frac{C||f||_{H^2_\delta}}{R^\beta}.
    \end{align*}
    Letting $R\to+\infty$, we deduce that $f$ is identically $0$. This proves the injectivity of $L$ and thus its invertibility.\\\\
    Finally, to prove the inequality \eqref{eq:laplacpluscompactestimAx}, we work by contradiction.\\
    Since $L$ is asymptotic to a Laplacian in the sense of \cite{zbMATH03964921} (for $q=4$), we have the following inequality (Theorem $1.10$ of \cite{zbMATH03964921}) for every $f$ in $H^2_\delta$
        \begin{align}
        \label{eq:ineqbartnikAx}
       ||f||_{H^2_\delta}\leq C (||Lf||_{L^2_{\delta+2}}+||f||_{L^2(B_R)}).
    \end{align}
     The constants $C$ and $R$ depend on $g$ and $\delta$.
    If the inequality \eqref{eq:laplacpluscompactestimAx} is false, then there exists a sequence $f_n$ such that 
    \begin{align*}
        &||f_n||_{H^2_{\delta}}=1, &\lim_{n->\infty}||Lf_n||_{L^2_{\delta+2}}=0.
    \end{align*}
    Since the sequence is uniformly bounded, there exists a subsequence that converges weakly to $f$ in $H^2_{\delta}$ and that converges strongly in $L^2(B_R)$ by Rellich-Kondrachov.\\
    Testing against $\forall \varphi\in C^\infty_c$, we obtain 
    \begin{align*}
        \lim_{n->\infty}\int{Lf_n\varphi dx}=0,
    \end{align*}
    by assumptions, and 
    \begin{align*}
        \int{(Lf_n-Lf)\varphi dx}=\int{(-\Delta f_n+\Delta f)\varphi dx}+\int{g(f_n-f)\varphi dx}\leq\int{-\Delta(f_n-f)\varphi dx}+||gu||_{L^2(\Omega)}||f_n-f||_{L^2(\Omega)},
    \end{align*}
    which implies that 
    \begin{align*}
        \lim_{n->\infty}\int{Lf_n\varphi dx}=\int{Lf\varphi dx}.
    \end{align*}
    Thus, $Lf=0$ and therefore $f=0$ by injectivity. Now, using the inequality \eqref{eq:ineqbartnikAx}, for any $\varepsilon>0$ there exists $m\in\mathbb{N}$ such that for $n\geq m$
     \begin{align*}
       1=||f_n||_{H^2_\delta}\leq C (\varepsilon+||f_n||_{L^2(B_R)}).
    \end{align*}
    Since $\lim_{n->\infty}||f_n||_{L^2(B_R)}=||f||_{L^2(B_R)}$, as the convergence is strong, we deduce that $f\neq0$, which is a contradiction. Therefore \eqref{eq:laplacpluscompactestimAx} holds.
\end{proof} 
 
\subsection{Strichartz estimates}
\label{subsection:strichartz}
\begin{defi}
\label{defi:strichartzAx}
   Let $n\in\mathbb{N}^\star$ be the considered dimension. A pair $(q,r)\in(\mathbb{R}_+^\star)^2$ is said to be Strichartz admissible if
    \begin{align*}
        &q\geq 2,
&\frac{2}{q}\leq (n-1)\left(\frac{1}{2}-\frac{1}{r}\right),\\
&(q,r,n)\neq(2,\infty,3),
&r<\infty.
    \end{align*}

\end{defi}
\begin{unTheorem}
\label{unTheorem:StrichartzAx}
    Let $n\in\mathbb{N}^\star$ be the considered dimension. For $(q_1,r_1)$ and $(q_2,r_2)$ Strichartz admissible satisfying
    \begin{equation}
    \label{eq:gapconditionAx}
\frac{1}{q_1}+\frac{n}{r_1}=\frac{1}{q^{\prime}_2}+\frac{n}{r^{\prime}_2}-2=\frac{n}{2}-\gamma,
    \end{equation}
then there exists $C>0$, such that, for any $f$ solution to the Cauchy problem \\
\begin{equation}
\begin{cases}
\label{eq:waveeqstrichartzAx}
\Box f=j,\\
f(0)=f_0,\;\partial_t f(0)=\dot{f}_0,
\end{cases}
\end{equation}\\
we have
\begin{equation}
\label{eq:estimstrichartzAx}
    ||f||_{L^{q_1}([0,T],L^{r_1})}\leq C (||j||_{L^{q^{\prime}_2}([0,T],L^{r^{\prime}_2})}+||f_0||_{\dot{H}^\gamma}+||\dot{f}_0||_{\dot{H}^{\gamma-1}}).
\end{equation}
\end{unTheorem}
\begin{proof}
      From \cite{zbMATH03579251} and \cite{zbMATH01215570}.
\end{proof}

\subsection{Product estimates}
\label{subsection:prodesti}
We recall the following Sobolev product estimate.\\
\begin{propal}
\label{propal:holdsobAx}
    Let $n\in\mathbb{N}^\star$ be the considered dimension. Let $1<p<\infty$, $r,s\in\mathbb{R}$, $r,s<\frac{n}{p}$, be such that $r+s>\max\left(\frac{n}{p}-\frac{n}{p'},0\right)$ for $\frac{1}{p}+\frac{1}{p'}=1$. If $p>2$, we additionally assume that either $f\in W^{p',-r}$ or $g\in W^{p',-s}$. Then, we have 
\begin{align}
\label{eq:holdsobAx}
    ||fg||_{W^{p,t}}\leq ||f||_{W^{p,r}}||g||_{W^{p,s}},
\end{align}
for $t=r+s-\frac{n}{p}$.
\end{propal} 
\begin{proof}
    From \cite{zbMATH01584967}.
\end{proof} 
We now state two lemmas providing product estimates based on Strichartz inequalities.\\
\begin{lem}
\label{lem:lemma1Ax}
    Let $f$ and $g$ be two functions defined on $[0,T']\times\mathbb{R}^3$ for $T'>0$. 
    For $\nu\in(0,1/2)$ and $p=\frac{1}{1-\nu}$, the following inequality holds 
    \begin{align}
        ||fg||_{L^1([0,T'],L^2)}\leq C||f||_{L^p([0,T'],H^{1/2-\nu})}(||\Box g||_{L^1([0,T'],L^2)}+||g(0)||_{\dot{H}^1}+||\partial_tg(0)||_{L^2}).
    \end{align}
\end{lem}
\begin{proof} 
For $\frac{1}{p}+\frac{1}{p'}=1$, we define $2<\alpha=\frac{3}{1+\nu}<3$ and $\frac{1}{2}=\frac{1}{\alpha}+\frac{1}{\beta}$. Then,
\begin{align*}
    ||fg||_{L^1([0,T'],L^2)}&\leq ||f||_{L^p([0,T'],L^{\alpha})}||g||_{L^{p'}([0,T'],L^{\beta})}\\
    &\leq||f||_{L^p([0,T'],H^{1/2-\nu})}||g||_{L^{p'}([0,T'],L^{\beta})}.
\end{align*}
We set $(r_2=2,q_2=\infty)$, which is Strichartz admissible. We look for $p'$ such that $(\beta,p')$ is Strichartz admissible too and satisfies
\begin{equation*}
    \frac{1}{p'}+\frac{3}{\beta}=\frac{1}{q^{\prime}_2}+\frac{3}{r^{\prime}_2}-2=\frac{1}{2}.
\end{equation*}
This implies that $p'=\frac{\alpha}{3-\alpha}$ and so $p=\frac{\alpha}{2\alpha-3}=\frac{1}{1-\nu}$. Applying Strichartz estimates to $||g||_{L^{p'}([0,T'],L^{\beta})}$ leads to the desired inequality.
\end{proof}
\begin{lem}
\label{lem:lemma2Ax}
   Let $f$ and $g$ be two functions defined on $[0,T']\times\mathbb{R}^3$ for $T'>0$. 
    For $\theta\in(0,1/2)$, $p=2(1-\theta)$ and $p'=\frac{2(1-\theta)}{1-2\theta}$, the following inequality holds 
\begin{align*}
    ||fg||_{L^1([0,T'],L^2)}\leq ||f||_{L^p([0,T'],L^2)}(||g||_{L^{p'}([0,T'],H^2)})^\theta(||\Box g||_{L^1([0,T'],L^2)}+||g(0)||_{\dot{H}^1}+||\partial_tg(0)||_{L^2})^{1-\theta}.
\end{align*}
\end{lem}
\begin{proof}
First, Hölder inequalities yield
\begin{align}
        ||fg||_{L^1([0,T'],L^2)}\leq ||f||_{L^p([0,T'],L^2)}||g||_{L^{p'}([0,T'],L^{\infty})},
\end{align}
Using the Gagliardo–Nirenberg inequality from Proposition \ref{propal:gagliardoAx}, we get 
\begin{align}
        ||fg||_{L^1([0,T'],L^2)}\leq ||f||_{L^p([0,T'],L^2)}(||g||_{L^{p'}([0,T'],H^2)})^\theta(||g||_{L^{p'}([0,T'],L^{r_1})})^{1-\theta},
\end{align}
where $r_1=6\left(\frac{1-\theta}{\theta}\right)>6$ with $\theta\in(0,1/2)$. We set $(r_2=2,q_2=\infty)$ and $p_1=p'$ and search for $(r_1,p_1)$ Strichartz admissible with 
\begin{equation*}
    \frac{1}{p_1}+\frac{3}{r_1}=\frac{1}{q^{\prime}_2}+\frac{3}{r^{\prime}_2}-2=\frac{1}{2}.
\end{equation*}
This yields $p'=\frac{2(1-\theta)}{1-2\theta}$ so that $p=2(1-\theta)$. Applying Strichartz estimates to $||g||_{L^{p'}([0,T'],L^{r_1})}$ leads to the desired inequality.
\end{proof}

\printbibliography
\end{document}